\documentclass[nospthms]{svjour2}
\usepackage[total={5.5in, 8in},marginparwidth=2.5cm,marginparsep=.5cm]{geometry}
\usepackage{times}
\usepackage[utf8]{inputenc}
\usepackage[english]{babel}
\usepackage{amsmath}
\usepackage{amsfonts}
\usepackage{amssymb}
\usepackage{amsthm}
\usepackage{mathtools}
\usepackage{graphicx,import}
\usepackage{comment}
\usepackage{marginnote}
\usepackage{xargs}                      
\usepackage[pdftex,dvipsnames]{xcolor}
\usepackage{caption}
\usepackage{subcaption}
\usepackage{microtype}
\usepackage{sub_JP}
\usepackage{mhequ}
\usepackage{enumerate}
\usepackage{sub_JP}
\usepackage{dsfont}
\makeindex
\let\theta=\vartheta
\let\kappa=\varkappa
\graphicspath{ {figures/} }
\numberwithin{equation}{section}
\DeclareMathAlphabet{\mathbbold}{U}{bbold}{m}{n}
\DeclareMathAlphabet\mathbfcal{OMS}{cmsy}{b}{n}

\newcommand{\figuresfolder}{.}

\def\jp#1{#1}
\def\jpp#1{#1}

\def\ad#1{#1}

\newcommand{\h}[2]{h_{r,#1}(#2) }

\def\UQv1{L_{r,v}(\xvi)}

\def\BB{\mathcal B}

\def\NN{\mathcal N}

\def\DD{\mathcal D}
\def\OO{\mathcal O}
\def\xvi{x{{{{}}}}}

\newcommand{\weight}{a}
\newcommand{\crw}{\jpp{c^{r,\weight}}}
\newcommand{\crwi}{\jpp{c^{r,\weight}_i}}
\newcommand{\Ov}{\pc{v^{-1}\mathbb N_0}^d}

\newcommand{\qq}[2]{Q_{#1,v}(#2)}
\newcommand{\uu}{{U}}
\newcommand{\zz}{\zeta_r}

\newcommand{\dtp}[2]{\langle #1, #2 \rangle}

\newcommand{\eref}[1]{(\ref{#1})}
\newcommand{\fref}[1]{Fig.~\ref{#1}}
\newcommand{\tref}[1]{Theorem~\ref{#1}}
\newcommand{\lref}[1]{Lemma~\ref{#1}}
\newcommand{\aref}[1]{Assumption~\ref{#1}}
\newcommand{\dref}[1]{Def.~\ref{#1}}
\newcommand{\rref}[1]{Remark~\ref{#1}}
\newcommand{\cref}[1]{Corollary~\ref{#1}}
\newcommand{\cndref}[1]{Condition~\ref{#1}}
\newcommand{\pref}[1]{Proposition~\ref{#1}}

\newcommandx{\unsure}[2][1=]{\todo[linecolor=red,backgroundcolor=red!25,bordercolor=red,#1]{#2}}
\newcommand{\pc}[1]{\left(#1\right)}
\newcommand{\pq}[1]{\left[#1\right]}
\newcommand{\pg}[1]{\left\{#1\right\}}

\newcommand{\eg}{\emph{e.g.},~}

\newcommand{\R}{\mathcal R}

\newcommand{\F}{\mathcal F}
\newcommand{\I}{\mathcal I}
\newcommand{\ii}{\iota}
\newcommand{\kk}{\iota}

\newcommand{\Ip}{\mathcal I(\mathcal P)}

\newcommand{\px}[2]{\mathbb P_{#1}\left[ #2\right]}
\newcommand{\p}[1]{\mathbb P \left[ #1 \right]}

\newcommand{\RR}{\mathbb R_+}
\newcommand{\Rr}{\mathbb R}

\newcommand{\C}{\mathcal C}

\newcommand{\ie}{{\it i.e.}, }
\def\S{\mathcal S}

\newcommand{\Lv}{\mathcal L_v}

\newcommand{\net}{$(\S,\C,\R)$~}

\newcommand{\sign}{{\rm sign}}

\newcommand{\supp}{{\rm supp}\,}

\newcommand{\inn}{{\rm in}}
\newcommand{\outt}{{\rm out}}
\newcommand{\PP}{\mathcal P}
\newcommand{\WW}{{\mathcal W}}

\newcommand{\VV}{\mathcal V}

\newcommand{\onevec}{1\!\!1}

\newcommand{\xt}{x(t)}
\newtheorem{theorem}{Theorem}[section]
\newtheorem{remark}[theorem]{Remark}
\newtheorem{df}[theorem]{Definition}
\newtheorem{definition}[theorem]{Definition}
\newtheorem{lm}[theorem]{Lemma}
\newtheorem{ex}[theorem]{Example}
\newtheorem{exxx}{Example}
\newtheorem{coro}[theorem]{Corollary}
\newtheorem{cla}[theorem]{Claim}
\newtheorem{axx}{Assumption}
\newtheorem{cnd}[theorem]{Condition}

\newtheorem{pro}[theorem]{Proposition}
\newtheorem{proposition}[theorem]{Proposition}

\newcommand{\abbr}[1]{{\small\sc\lowercase{#1}}}

\newcommand{\dfn}[1]{\begin{df} #1 \end{df}}

\newcommand{\rmk}[1]{\begin{remark} #1 \end{remark}}

\newcommand{\trm}[1]{\begin{theorem} #1 \end{theorem}}
\newcommand{\exm}[1]{\begin{ex} #1 \end{ex}}

\newenvironment{myitem}
{\begin{itemize}
  \setlength{\itemsep}{1pt}
  \setlength{\parskip}{0pt}
  \setlength{\parsep}{0pt}}
{\end{itemize}}

\newenvironment{myenum}
{\begin{enumerate}[(a)]
  \setlength{\itemsep}{1pt}
  \setlength{\parskip}{0pt}
  \setlength{\parsep}{0pt}}
{\end{enumerate}}
\def\WWji{{\F}}
\def\Wji{{\F^*}}
\def\WWed{{(\WW_{j,\ii}^*)^{\epsilon_j,\delta_j}}}
\def\WWedd{(\WW(\PP)_{j,\ii}^*)^{\epsilon_j,\delta_j}}
\def\Xed{(\F^*)^{\epsilon ,\delta }}
\def\Cc{K_3}

\def\Clambda{K_1}
\def\Cepsilon{K_2}

\def\Cmonomial{K_5}
\def\Cv{K_4}
\def\d{{\rm d}}
\def\dt{{\d t}}

\def\emx#1{\emph{#1}\index{#1}}
\let\epsilon=\varepsilon
\let\rho=\varrho
\let\kappa=\kappa
\def\Ctwo{A+2B}
\def\Cthree{3B}
\def\cornertwo{(1,2)}
\def\cornerthree{(0,3)}

\title{\Large On the Geometry of Chemical Reaction Networks:\\
 Lyapunov Function and Large Deviations}
\author{A.~Agazzi \and A.~Dembo \and J.-P.~Eckmann}
\institute{
A. Agazzi \and J.-P. Eckmann \at D\'epartement de Physique Th\'eorique, Universit\'e de Gen\`eve, CH-1211 Gen\`eve 4 (Switzerland)
\and
A. Agazzi \and A. Dembo \at Statistics Department, Stanford University, CA 94305 Stanford (USA)
\and
J.-P. Eckmann \at Section de Math\'ematiques, Universit\'e de Gen\`eve, CH-1211 Gen\`eve 4 (Switzerland)
}
\date{\today}

\begin{document}
\maketitle
\begin{flushright}\emph{ To Herbert, J\"urg, and Tom, friends and inspirations.}\end{flushright}
\begin{abstract}
 In an earlier paper, we proved the validity of large deviations theory
 for the particle approximation of quite general chemical reaction networks (\abbr{crn}s). In this paper, we \jpp{extend} its scope and
 present a more geometric insight into the mechanism of
 that proof, exploiting the notion of spherical image of the reaction
 polytope. This allows to view the asymptotic behavior of the vector
 field describing the mass-action dynamics of chemical reactions as
 the result of an interaction between the faces of this polytope in
 different dimensions. We also illustrate some local aspects of the
 problem in a discussion of Wentzell-Freidlin (\abbr{wf}) theory, together with
 some examples.
\end{abstract}

\tableofcontents
\section{Introduction}

\jp{Our earlier paper \cite{agazzi17} establishes, among other things, a Large Deviation Principle (\abbr{ldp}) for the particle approximation of \abbr{crn}s.
In this work we extend the scope of that result, and present an alternative, constructive proof for it. To make this paper self-contained we also repeat some
definitions and statements of \cite{agazzi17} in quite some detail.
We further expand} the presentation by several examples which illustrate the difficulties
of finding adequate conditions for the \abbr{ldp} to
hold on the full phase space.
\jp{The alternative proof we present stresses the geometric aspects of}
the so-called ``toric
jet'' method \cite{gopal13} used in \cite{agazzi17}.
We will see that the geometric method, while similar in spirit to the
work of \cite{gopal13}, exhibits the simple geometric ideas which concern
the asymptotic nature of \abbr{ODE}'s with polynomial nonlinearities. \jp{In particular, the
 constructive proof in the present paper provides explicit stability estimates on the systems at hand, as opposed to
the proof by contradiction in \cite{agazzi17} and in \cite{gopal13}.}

Modeling the dynamics of large sets of chemical reactions has gained interest in recent years thanks to the increased amount of biological data on complex chemical systems. These dynamics take place in high dimensional spaces (with dimension equal to the number of species involved in the chemical reactions) and display highly complex dynamical behavior (all classes of attractors can be realized by models of chemical reaction systems \cite{magnasco97,vidal80}). The theoretical study of such dynamics is a topic of central interest in domains such as chemical reaction network theory \cite{feinberg87,horn72} and, more generally, systems biology, dedicated to the understanding of the laws governing large scale biochemical systems.
Progress in this direction can be made by considering stochastic effects in chemical systems. Indeed, as is the case in equilibrium statistical mechanics, large fluctuations induce transitions between the different attractors of the system, whose state finally stabilizes, in the probabilistic sense, in a neighborhood of the (in general) unique attractor with the lowest potential energy. In this sense, stochastic models and their probabilistic distribution (or the evolution of their stochastic paths) are ideal for studying the short and long term-behavior of complex biochemical systems.
Indeed, the mathematical study of stochastic chemical dynamics has been an active area of research already at the end of last century \cite{either86}. Recently, probability distributions for \abbr{crn}s in detailed balance have been fully characterized through an analogy with Jackson Networks in \cite{anderson15}. The steady states of such systems are equilibrium states and are unique for every invariant manifold of the dynamical system \cite{feinberg87,gunawardena03}. However, many systems of practical interest in cell biology fall in the category of nonequilibrium systems, whose study is still an open topic even at the fundamental level \cite{touchette09}.

One promising tool for the systematic study of the stochastic dynamics of such nonequilibrium systems is potential landscape theory for chemical reaction systems, and more generally for biochemical systems. This theory has recently been investigated in, \eg \cite{neqpot1}. Here a large deviations rate function from \abbr{wf} theory (called the \abbr{wf} quasipotential) has been indicated as a promising candidate for a potential landscape function.
However, the intuition developed in \cite{neqpot1} is not always rigorously justifiable in the case of mass action systems. In particular, the \abbr{WKB} approximation used for the derivation of the Hamilton Jacobi Equation (\abbr{HJE}) for the \abbr{WF} potential is not guaranteed to converge,  {and} a rigorous framework for infinite-dimensional integration in the space of paths has not yet been established. Moreover, it is well known \cite{toth99,vankampen92} that for mass action systems the diffusion processes studied in \cite{neqpot1} are large-volume approximations of the microscopically justified Markov jump models and that the potential landscapes predicted by \abbr{wf} theory differ in these two processes.

In this article, we initiate a formal study of potential landscape theory for Markov jump models of \abbr{crn}s with mass action kinetics in the large volume limit through the rigorous establishment of estimates \`a la  \abbr{WF}.
Markov jump models are the framework of choice for the modeling of the dynamics of \abbr{crn}s because they embody the discrete character of the interacting particles at the microscopic level. Such processes can then be scaled \cite{anderson15,either86,LDT2} to study the behavior of large amounts of reacting molecules by considering reactors with volume $v$ and taking $v$ as scaling parameter.
The effect of this scaling on the model is a reduction of the amplitude of the stochastic fluctuations in phase space. The study of fluctuations in finite time of stochastic systems in the large volume limit is the object of large deviations theory \cite{LDT1}.
The exponential estimates obtained with  {this} theory can be extended to infinite time intervals through \abbr{wf} theory, allowing in particular to give exponential estimates of exit times from compact sets in phase space, transition times between different attractors and invariant measure densities.
Furthermore, \abbr{WF} theory establishes a mathematically rigorous framework for potential landscape theory for biochemical systems through the definition of a quasipotential function $\VV(x)$. This quasipotential is one of the most promising candidates for the generalization of equilibrium potentials in statistical mechanics to nonequilibrium systems \cite{touchette09}.

The paper is structured as follows. In the first section we introduce the reader to deterministic and stochastic mass action models for the dynamics of \abbr{crn}s. We then outline some typical statements of large deviations theory with particular attention to problems in theoretical biochemistry. Appealing to the companion paper \cite{agazzi17}, we then introduce a class of \abbr{crn}s for which the applicability of a \abbr{ldp}  and ultimately \abbr{wf} estimates {is}
rigorously established. This class of \abbr{crn}s ({generalizing somewhat}
the strongly endotactic \abbr{crn}s of \cite{agazzi17,gopal13}) is characterized solely on the base of the topology of the underlying network of reactions and is therefore independent of reaction constants. This has the advantage of avoiding difficult estimates on the high-dimensional stochastic dynamics of the corresponding perturbed dynamical system. We also provide an alternative, constructive proof of geometric character of the results in \cite{agazzi17}. This approach is somewhat related to tropical geometry \cite{mihalkin15}. Finally, we discuss the extension of large deviations estimates to the infinite time horizon through \abbr{wf} theory and give a dynamically nontrivial example of a \abbr{crn} to which such results can be applied.

\section{The model}

We consider a set $\S := \{s_1, \dots, s_d\}$ of $d$ interacting \emph{chemical species}. The transitions between
different species are described by a set $\R = \{r_1, \dots, r_m\}$ of $m$ \emph{chemical reactions}.
Every reaction $r \in \R$ can uniquely be written in the form
$$r = \pg{ \sum_{i = 1}^d (c_\inn^r)_i s_i \rightharpoonup \sum_{i = 1}^d (c_\outt^r)_i s_i }~,$$
where the nonnegative integer vectors $c_\inn^r, c_\outt^r$ counting the species multiplicities as inputs and outputs of the
reaction are called, respectively, the \emph{input} and \emph{output complexes} of $r$. We finally denote by
$\C := \{c_\#^r~:~r \in \R\,,\,\# \in \{``\inn",\,``\outt"\}\}$
the \emph{set of complexes}, and for each $r$ we define a \emph{reaction vector}
\begin{equ}
 c^r := c_\outt^r - c_\inn^r~, \label{e:cr}
\end{equ}
describing the net effect of $r$ on the system. A \abbr{crn} is defined by the triple
\net.

\rmk{We extend the definitions to include open chemical networks, by
 associating with sinks and external nonautocatalytic sources the complex zero vector (with $d$
 components), denoted by $\emptyset$\,.}

\begin{exxx}
 The system
 \begin{equ}
  \Ctwo  \mathrel{\mathop{\rightleftharpoons}^{\mathrm{r_1}}_{\mathrm{r_2}}} \Cthree
  \label{e:ex1}
 \end{equ}
 is a \abbr{crn} with $\S = \{A,B\}$ and $\R =
 \{r_1,r_2\}$\,. The set of complexes of this reaction is $\C =
 \{\{\Ctwo \},\{\Cthree \}\} = \{\cornertwo ,\cornerthree \}$ (in the basis spanned by $(A,B)$).
 \label{ex:1}
\end{exxx}

\subsection{Deterministic mass action kinetics}

The state of the system is described by the \emx{vector of
 concentrations}\index{concentration} $x \in \RR^d$ (the set of $d$-dimensional nonnegative reals) of the $d$ species. Its evolution is
commonly described by the
\emx{mass action
 kinetics} model, which is the set of \abbr{ODE}s
\begin{equ}
 \frac{\d x}{\dt} = \sum_{r \in \mathcal R} \lambda_r(x) c^r~,
 \label{e:ma}
\end{equ}
where $\lambda $ is the monomial
\begin{equ}
 \lambda_r(x) := k_r \prod_{i = 1}^d x_i^{(c^r_{\rm in})_i}
 \label{e:rates}
\end{equ} describing the  \emx{reaction rate} of reaction $r$\,.
The constants $\{k_r\} \in (0,\infty)$ are called \emx{reaction
rate constants}.

\rmk{Note that the set $\RR^d$ is invariant under the dynamics described by \eref{e:ma}: For every $s_i \in \S$,
 any reaction $r \in \R$ reducing the amount of $s_i$ in the reactor will have $(c_\inn^r)_i > 0$, implying by \eref{e:rates} that $\lambda_r(x) = 0$ on
 $\{x \in \RR^d~:~x_i = 0\}$. \\
 Furthermore, again by \eref{e:ma}, for each initial condition $x_0 \in \RR^d$, $x(t)$ remains in $S_{x_0} := (x_0 + \text{span} \{c^r~:~r\in\R\}) \cap \RR^d$\,. This space is called the \emx{stoichiometric
  compatibility class} of $x_0$ and can be less than $\RR^d$\,.
 \label{r:RRisinvariant}}

\addtocounter{exxx}{-1}

\begin{exxx}[continued] The time derivative of a trajectory $x(t)$ for Network \eref{e:ex1} is given by
 $$\frac{\d x}{\dt} = k_1 \lambda_{r_1}(x)\, c^{r_1} + k_2 \lambda_{r_2}(x)\,c^{r_2}  =  k_1  x_A x_B^2 \begin{pmatrix}-1\\1\end{pmatrix} + k_2
 x_B^3
 \begin{pmatrix}1\\-1\end{pmatrix}~,
 $$
 for reaction rate constants $k_1,k_2 \in (0, \infty)$. A trajectory
 starting at $x_0 \in \RR^2$ cannot leave the stoichiometric compatibility
 class
 $$S_{x_0} = \left\{x_0 +\text{\rm span}\begin{pmatrix}-1\\1\end{pmatrix}\right \}\cap
 \RR^2~.
 $$
\end{exxx}

\begin{exxx}{\label{ex:ex2} Consider the (open) \abbr{crn}
  \begin{equ}
   \emptyset~\mathop{\rightharpoonup}^{\mathrm{r_1}}~\Ctwo ~\mathop{\rightharpoonup}^{\mathrm{r_2}}~\Cthree ~\mathop{\rightharpoonup}^{\mathrm{r_3}}~A~.
   \label{e:ex2}
  \end{equ}
  The dynamics of this \abbr{CRN} is described by the solution $x(t)$ of the equation
  \begin{equ}
   \frac{\d x}{\dt} = k_1 \begin{pmatrix}1\\2\end{pmatrix} + k_2  x_A x_B^2 \begin{pmatrix}-1\\1\end{pmatrix} + k_3  x_B^3 \begin{pmatrix}1\\-3\end{pmatrix}~,
   \label{e:vfex1}
  \end{equ}
  where $x_A$ and $x_B$ are the concentrations of $A$ and $B$ respectively. Fig.~\ref{f:vfex1} represents the vector field of
  \eref{e:vfex1}.
  In this case the stoichiometric compatibility class of the network is
  $S_{x_0} = \RR^2$\, for all $x_0 \in \RR^2$\,.
  \begin{figure}[t]
   \centering
   \includegraphics[scale=0.25]{\figuresfolder/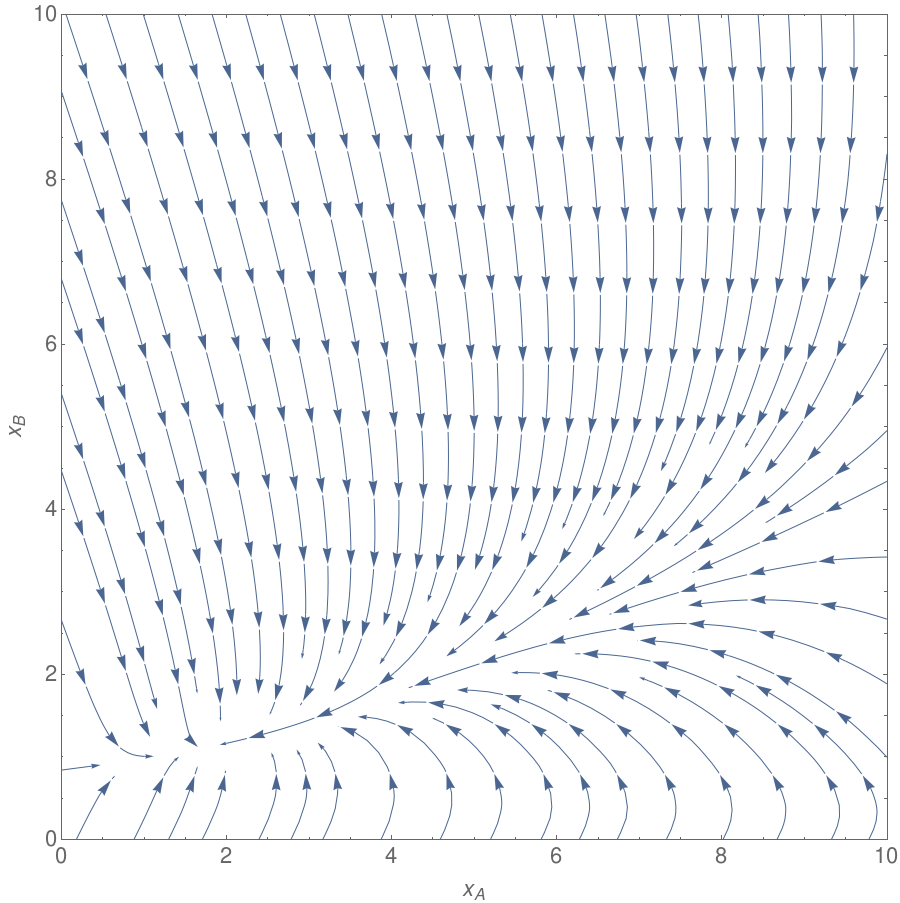}
   \caption{Vector field of \eref{e:vfex1} with $k_1 = k_2 = k_3 = 1$\,.}
   \label{f:vfex1}
  \end{figure}
  \label{ex:2}}
\end{exxx}

\setcounter{theorem}{3}

\subsection{Stochastic mass action kinetics}

We now replace the \abbr{ODE} by a stochastic process, which appears when one
takes into account that chemical reactions are formed by a discrete set
of molecules. We are interested in the approximation of perfect
mixing, \ie we only consider numbers of molecules, but not their
spatial distribution.
This is then a mean field pure jump Markov process \cite{either86}
whose central object is the random vector $N_t \in \mathbb N_0^d$ (the set of $d$-dimensional nonnegative integers), representing the number of molecules of the $d$ species at time $t$\,.
Every reaction $r \in \R$ describes a possible jump of the process $N_t$ as a transition $$N_t \rightarrow N_t + c^r~,$$ where $c^r$ is the \emx{reaction vector}
associated with $r \in \R$ as defined in \eref{e:cr}.

The state-dependent mean field rate $ \Lambda_r(N_t)$ of jumps of the process $N_t$ in direction $c^r$ under
mass action kinetics is then proportional to the number of unordered configurations of the (distinguishable) molecules and is modeled as
\begin{equ}
 \Lambda_r(N_t) = k_r \prod_{i = 1}^d \binom{(N_t)_i }{ (c^r_{\rm in})_i} (c^r_{\rm in})_i! ~,
 \label{e:ladef}
\end{equ}
where $k_r$ is the reaction rate constant of reaction $r$.
The generator $\mathcal L$ of the process is
\begin{equ}
 \mathcal Lf(N_t) = \sum_{r \in \mathcal R}{\Lambda_r( N_t)} \left(f( N_t + c^r)-f(N_t)\right)~.
 \label{e:Lmarkov_0}
\end{equ}

We are interested in the scaling of the dynamics of the system for large volumes, \ie as we increase the number of particles by a large multiplicative factor.
We denote by $v$ the volume of the system under consideration and we define the \emx{vector of concentrations} $X_t$ as the linear scaling of $N_t$ in $v$, \ie
\begin{equ}
 X_t^v := v^{-1}{N_t}~.
 \label{e:scaling}
\end{equ}
The natural scaling in volume $v$ of the reaction rate constants is
\begin{equ}
 k_r^{(v)} := {v^{-\|c^r_{\rm in}\|_1}}k_r~,
 \label{e:kscaling}
\end{equ}
where $\|\cdot\|_1$ is the $1$-norm in $\mathbb R^d$\,.
This results in  extensive laws for the reaction rates, appearing in the generator of the scaled process $X_t^v$:
\begin{equ}
 \mathcal L_v f(x) := v \sum_{r \in \mathcal R}{\Lambda_{r,v}(x)} \left(f(x + v^{-1}c^r)-f(x)\right)~, \qquad x \in (v^{-1} \mathbb N_0)^d ~,
 \label{e:Lmarkov_1}
\end{equ}
where we define
\begin{equ}
 \Lambda_{r,v}(x) =
 k_r^{(v)} \prod_{i = 1}^d
 \binom{v x_i}{(c^r_{\rm in})_i} (c^r_{\rm in})_i! = \lambda_r(x) + o_v(1) \qquad \text{as}~v\to\infty~.
 \label{e:jrates}
\end{equ}



For a fixed volume $v$ it follows that
the phase space of $X_t^v$ is $(v^{-1}\mathbb N_0)^d$.
In general,
this process is not irreducible, as the set of possible transitions of
$X_t^v$ might not reach every point in phase space (\eg a
reaction might map even numbers of species only on even numbers).

It is well known \cite{either86} that for any time $T > 0$, in the
limit $v \rightarrow \infty$ the sample paths of the process $X_t^v$
with initial condition $\lim_{v \to \infty}X_0^v = x_0 \in \RR^d$
converge almost surely on $[0,T]$ towards the solution of the \abbr{ode}s \eref{e:ma} starting at $x_0$\, provided that such a solution exists on $[0,T]$. This extends
to strong convergence through the
establishment of a Central Limit Theorem for the trajectories of
$X_t^v$ \cite{either86}\,.


\section{Large deviations theory}

We now study the $v$-approximation in the context of large
deviations theory to estimate
the probabilities of finding paths which deviate from the deterministic model
\cite{LDT1,LDT2}. This will allow to establish exponential tail
estimates on the probability measure in the space of paths. In our
case, by the jumping nature of the process $X_t^v$, this space is the
Skorokhod space, \ie the space of functions that are everywhere right
continuous and having left limits (also called \emph{c\`adl\`ag}
functions). Throughout, we denote by $D_{0,T}(\RR^d)$ the Skorokhod
space, or space of \emph{c\`adl\`ag} functions $z~:~[0,T] \to \RR^d$,
equipped with the topology of uniform convergence.
In this space, the typical statement of large deviations theory is summarized in the
following definition from \cite{LDT1}.

\dfn{Fix $T$ finite and a lower semi-continuous mapping $I:D_{0,T}(\RR^d) \to [0,\infty]$
 such that for any $\alpha \in \RR$,
 the level set $\{z
 :~I(z) \leq \alpha\}$
 is a compact subset of $D_{0,T}(\RR^d)$.
 The probability distribution of sample paths of the process $\{X_t^v\}_{v \in \mathbb N}$ with fixed initial condition $X_0^v \to x \in \RR^d$
 obeys a \abbr{LDP} with good rate function $I(\cdot)$ if for any measurable $\Gamma \subset D_{0,T}(\RR^d)$ we have
 \begin{equs}
  - \inf_{z \in \Gamma^o} I(z) &\leq \liminf_{v \to \infty} \frac{1}{v} \log \px{X_0^v}{X_t^v \in \Gamma}\\& \leq \limsup_{v \to \infty} \frac{1}{v} \log \px{X_0^v}{X_t^v \in \Gamma} \leq - \inf_{z \in \bar \Gamma} I(z)~,    \label{e:ldp}
 \end{equs}
 where $\Gamma^o, \bar \Gamma$ denote the interior and the closure of the set $\Gamma$ respectively. \label{d:ldp}}

In the case of pure jump Markov processes, it is well known \cite{LDT2} that a candidate for a good rate function can be obtained by analogy with the theory of Lagrangian mechanics.
Indeed, by defining the Lagrangian function
\begin{equ}
 L(\lambda,\xi) =
 \sup_{\theta \in \Rr^d} \pg{\langle \theta, \xi \rangle - \sum_{r \in \R} \lambda_r \pq{\exp (\langle \theta, c^r \rangle) - 1} }~,
\end{equ}
one defines the rate function $I_{x_0,T}$ for fixed $x_0$, $T < \infty$ as the corresponding action along a path $z \in D_{0,T}(\RR^d)$, \ie
\begin{equ}
 I_{x_0,T}(z) = \begin{cases} \int_0^T L(\lambda(z), \dot z)\, \dt & \text{if } z \in AC_{0,T}(\RR^d),\,z(0) = x_0
 \\\infty & \text{otherwise}
 \end{cases}~,
 \label{e:grf}
\end{equ}
where $AC_{0,T}(\RR^d)$ is the set of absolutely continuous paths $z~:~[0,T] \to \RR^d$.
Continuing the analogy with Lagrangian mechanics, the \abbr{LDP} expressed in \eref{e:ldp} corresponds
to the minimum action principle, and more precisely with the path integral formulation of quantum mechanics
stated in rigorous probabilistic terms.



Standard large deviations tools allow us to deduce (see \cite{LDT2}) that the sample paths of process $X_t^v$ obey a \abbr{LDP} with the rate functions of \eref{e:grf} under the following
\begin{cnd}
 The rates $\lambda_r~:~\RR^d \to \Rr$ are, on $\RR^d$,
 \begin{myenum}
  \item[$(0)$] uniformly bounded away from $0$, and
  \item[$(\infty)$] uniformly Lipschitz continuous.
 \end{myenum}
 \label{a:lipschitz}
\end{cnd}

In the framework of stochastic mass action kinetics, \cndref{a:lipschitz} is in general not satisfied.
Indeed, neither the reaction rates $\Lambda_{r,v}$ are in general bounded away from zero (see~\rref{r:RRisinvariant})
nor are they uniformly Lipschitz continuous on $(v^{-1}\mathbb N_0)^d$, as can be verified by inspection of \eref{e:jrates}.

\cndref{a:lipschitz} $(0)$ guarantees the nondegeneracy of $X_t^v$ within $S_{x_0}$, whereas \cndref{a:lipschitz} $(\infty)$ is used in this framework to ensure
that the process leaves a certain, large, compact with a probability that is negligibly low at the exponential scales of our estimates.
While \cndref{a:lipschitz} fails for general \abbr{crn}s,
we show directly
that the \abbr{ldp} \eqref{e:ldp} holds for a large class of \abbr{crn}s,
which we will call \abbr{ASE}. To this end, we require the following
more technical definition.
\dfn{\label{d:exptight} The process $X_t^v$ satisfies an exponential compact containment condition if there is, for every \jp{$0 < \alpha$, $\gamma$, $T < \infty$}\,, a finite
 $\rho_{\alpha,\gamma, T}$, so that
 \begin{equ}
  \limsup_{v \rightarrow \infty}
  \frac{1}{v}\log\Big(
  {
   \sup_{\|x_0^v\|_1 \le \gamma}} \,
  \mathbb P_{x_0^v}\big[\sup_{t \in [0,T]} {\|X_t^v\|_1} >
  \rho_{\alpha,\gamma, T}
  \big]\Big) \leq - \alpha~.
  \label{e:exptight}
 \end{equ}}

\rmk{Even though \dref{d:exptight} does not directly guarantee that the process $X_t^v$ is exponentially tight in $D_{0,T}(\RR^d)$, such exponential tightness of $X_t^v$ can be proved
 by application of the Arzela-Ascoli theorem to
 the continuous process $\widetilde X_t^v$ obtained by linearly interpolating the sample paths of $X_t^v$ between its jumps. As $\widetilde X_t^v$ is exponentially equivalent \cite{LDT1} to $X_t^v$, the result carries over to the original process.}

As demonstrated in our next example,
\cndref{a:lipschitz} $(\infty)$ is not necessary for
\eqref{e:exptight} in the class of models under consideration.

\begin{exxx}{The \abbr{crn} $2A \rightleftharpoons \emptyset$
  does not satisfy \cndref{a:lipschitz} $(\infty)$
  but \eref{e:exptight}, still holds: setting \abbr{wlog} $k_1 = k_2 = 1$, using standard tail estimates on Poisson distribution, and denoting by $R(\Delta t)$ the number of jumps of a unit rate Poisson process in a time interval $\Delta t$, we obtain
  \begin{equs}
   \px{x_0^v}{\sup_{t\in [0,T]} \|X_t^v\|_1 >
    {\|x_0^v\|_1+2}
    \rho} & \leq { \p {R(vT) > v \rho} \approx}
   \frac{e^{-vT}\pc{evT}^{v \rho}}{(v \rho)^{v \rho}}
   \\
   & = \exp\pq{-v \pc{T + \rho \log (\rho/(e T))}}\,,
  \end{equs}
  hence \eref{e:exptight} holds for
  $\rho_{\alpha,\gamma, T} = {\gamma + 2} T e^{(\alpha/T)+1}$\,.}
\end{exxx}

There are of course \abbr{crn}s which do not even satisfy \eref{e:exptight}.
\exm{The \abbr{crn} $2A \rightharpoonup 3A$ clearly does not satisfy condition \eref{e:exptight}. Indeed, setting \abbr{wlog} the reaction rate constant $k = 1$, the solution
 of the postasymptotic \abbr{ode} $\dot x_A = x_A^2$ diverges at $T_\infty(x_0) < \infty$. Hence, for any $\rho$ the \abbr{lhs} of \eref{e:exptight} goes to $0$ for $T \geq T_\infty(x_0)$.}

\noindent Working via the compact containment of \dref{d:exptight},
we show in \cite{agazzi17} that \cndref{a:lipschitz} can be relaxed,
so the \abbr{LDP} \eref{e:ldp} holds under the following assumption about
the existence of a suitable Lyapunov function \jp{$U(\cdot) > 0$}.
\begin{axx}[{\cite[Ass.~A.1]{agazzi17}}] Let $X_t^v$ be the solution of the martingale problem generated by $\Lv$ of \eref{e:Lmarkov_1}\,. We assume
 \begin{myenum}
  \item There exists a $b<\infty$ such that for all $\rho>0$ one can find
  a $v^*(\rho)$ such that for all $x $ with $\|x\|_1 < \rho$, the following holds for all
  $v>v^*$ when
  $x\in (v^{-1} \mathbb N_0)^d$:
  \begin{equ}
   (\Lv U^v) \, (x) \leq
   e^{b v}~,
  \end{equ}
  where $U^v$ represents the $v$-th power of $U$.
  \item  The Markov chain associated to $X_t^v$ reaches a state $x_+$ in the strictly positive orthant $(v^{-1} \mathbb N)^d$ with positive probability.
 \end{myenum}
 \label{a:1}
\end{axx}

\aref{a:1} is satisfied by the network in
\setcounter{exxx}{2}
\begin{exxx} Defining $U(x) := e^x$ we have that
 \begin{equs}
  \Lv U^v(x) & = v \Lambda_1^{(v)}(x) \pc{e^{v x - 2} - e^{v x} } + v \Lambda_2^{(v)}(x) \pc{e^{v x + 2} - e^{v x} }\\& = v e^{vx} ( vx (vx - 1) (e^{-2}-1) + (e^2-1))  < 0~,
 \end{equs}
 when $x> 4/v$, thereby proving that this system satisfies \aref{a:1} (a). \aref{a:1} (b) is also trivially satisfied by the existence of the reaction $\emptyset \rightharpoonup 2A$.
\end{exxx}
\setcounter{theorem}{7}

\begin{remark}
 \aref{a:1} implies that the \abbr{ode}s \eref{e:ma} have a global solution, as shown in \cite{agazzi17}.
\end{remark}
%

\section{Topological conditions}

In concrete applications, \eg in biochemistry where typically $d \sim
100$, establishing estimates such as \eref{e:exptight} can be
particularly challenging. For this reason, in \cite{agazzi17} a large
class of networks has been shown to automatically verify the
conditions in \aref{a:1}. We illustrate next the ideas behind these conditions.\\
By definition \eref{e:rates}, some of the jump rates of $X_t^v$ vanish
at the boundaries of phase space (\ie where the concentration of some
of the species vanishes). This implies that on those boundaries the
vanishing reactions will be canceled from the network \net, making the
dynamics there qualitatively different from the one in the bulk
$(\RR^d)^o$. To take this into account, the definitions introduced
next have boundary-specific character.\\

For the chemical \abbr{ODE}, an orbit cannot end on a
boundary in finite time if it has started in the interior of the
positive orthant of concentrations. But for the \emph{discrete}
particle approximation in finite volume, a species might disappear
with finite (albeit small) probability. If there is no reaction which
re-creates this species, the system gets stuck in there, and the \jp{\abbr{LDP}} fails in the sense
that ergodicity is broken. Therefore, for our results to hold, we need
criteria on chemical networks which guarantee that this ``sticking''
does not happen. As this sticking could happen also with several
species at once, we need to define an exit condition from general
subspaces (which we call $\PP$, below). It will be seen that these
exit conditions can be algorithmically verified.

We now proceed to define the necessary quantities for the study of configurations where some species are extinct. We classify $x \in (v^{-1} \mathbb N_0)^d$ depending on its support, \ie the set $\PP$ of species that are present in strictly positive concentrations. Fixing $\PP \subseteq \S$ with cardinality $d_\PP := |\PP|$, we denote by $\pi_\PP: \Rr^d \to \Rr^{d_\PP}$ the projection onto the coordinates with indices in $\PP$. For any $w \in \Rr^{d_\PP}$ and any set $\mathcal A \subset \Rr^{d_\PP}$, we define the set $\mathcal A_w$ as the set of points in $\mathcal A$ that maximize the scalar product with $w$. We also denote by $\R(\PP)$ the set of reactions with inputs in $\PP$ (\ie $\supp\, c_\inn^r \subseteq \PP$). For all $w\in\Rr^d$ with non-zero projection $w_\PP := {\pi_\PP} w$,
we then let $\R(\PP)_w$ denote the reactions in $\R(\PP)$ maximizing the inner product $\dtp{w}{c^r_\inn}$, \ie those reactions corresponding to the elements of $\C_\inn(\PP)$ that are exposed by the vector $w$.
Clearly, $\R(\PP)_w$ depends only on $w_\PP$
which \abbr{wlog} is in the
$(d_\PP-1)$-dimensional unit sphere $S^{d_\PP-1}$. When $\PP=\S$ we will write $\R_w$
instead of $\R(\S)_w$.

\dfn{[\ad{adapting} {\cite[Def.~3.6]{agazzi17}}] For any \jpp{$\weight \in {\mathbb R_{>0}^{d}}$,} $w_\PP \in S^{d_\PP-1}$, \jpp{we define} $\crwi \jpp{:= c^r_i\,\weight_i}$ and say that a reaction $r \in \R(\PP)$ is:
 \begin{myitem}
  \item  \emph{$\jpp{(w,\weight)}$-dissipative} if $\supp\{ c^r_{\rm out} \} \not \subseteq \PP$ or if $\langle w_\PP, {\pi_\PP} \crw \rangle < 0$\,,
  \item \emph{$\jpp{(w,\weight)}$-null} if $\supp\{ c^r_{\rm out} \} \subseteq \PP$ and $\langle w_\PP, {\pi_\PP} \crw \rangle = 0$\,,
  \item \emph{$\jpp{(w,\weight)}$-explosive} if $\supp\{ c^r_{\rm out} \} \subseteq \PP$ and $\langle w_\PP, {\pi_\PP} \crw \rangle > 0$\,.
\end{myitem}
\label{d:dissipative}}

\begin{df}[\ad{extending} {\cite[Def. 3.8]{agazzi17}}]
 \jpp{For any $\weight \in \mathbb R_{>0}^{d}$ with $\|\weight\|_1 = d$}, a \abbr{crn} $(\S,\C,\R(\PP))$ is called
 \emph{strongly \jpp{$(\PP,\weight)$}-endotactic} if
 the set $\R(\PP)_{w}$ contains at least one $\jpp{(w,\weight)}$-dissipative reaction,
 and no $\jpp{(w,\weight)}$-explosive reaction for any $w \in \Rr^d$ with non-zero projection onto $\PP$ (or $w_\PP \in S^{d_\PP-1}$).\\
 Finally, we say that a network is \emph{strongly endotactic} if \jpp{there exists $\weight \in \mathbb R_{>0}^d$ such that it is strongly $(\S,\weight)$-endotactic.}
 \label{d:Pendo1}
\end{df}

\begin{remark}
 \jpp{We denote throughout by $\onevec \in \Rr^d$ the vector with $\onevec_i = 1$ for all $i \in \S$.} For $\PP=\S$ \jpp{and $\weight = \onevec$ our \dref{d:Pendo1} of strongly \jpp{$(\PP,\weight)$}-endotactic \abbr{crn} and} our \dref{d:dissipative} of $\jpp{(w,\weight)}$-dissipative and $\jpp{(w,\weight)}$-explosive
 reactions coincide with \jpp{\cite[Def. 3.8]{agazzi17} of strongly endotactic \abbr{crn} and} \cite[Def. 6.15]{gopal13} of
 $w$-sustaining and $w$-draining reactions, respectively.
 The nomenclature was changed \jpp{to generalize the definition in the first case and} to stress the behavior of reactions
 for $\|x\|_1 \gg 1$ in the second. Indeed, as we will see below, dissipative [explosive] reactions contribute to the decrease [increase] of a certain Lyapunov function \jpp{$\uu(x)$} along trajectories far away from the origin.
\end{remark}

\begin{lm}[{\cite[Lemma 3.9]{agazzi17}}]
 \jpp{For any $\weight \in \mathbb R_{>0}^d$,} if \net is strongly \jpp{$(\S,\weight)$}-endotactic then for all $\PP \subset \S$ with $\R(\PP) \neq \emptyset$, $(\S,\C,\R(\PP))$ is strongly {$(\PP,\weight)$}-endotactic.
 \label{l:endotacticorpositive}
\end{lm}

At this point we can introduce a visual representation of \abbr{crn}s. A \abbr{crn} \net can be uniquely represented in $\mathbb N_0^d$ by drawing for each $r \in \R$ the vector $c^r$ starting at
$c_\inn^r$. We call this diagram the \emph{complex diagram} of \net and denote the convex hull
of the set $\C_{\inn} := \{c_\inn^r~:~r \in \R\}$ by the \emph{complex polytope} $\WW$ of the network.

\begin{remark}
\jpp{
Similarly to \cite{gopal13}, we observe that a \abbr{crn} is strongly endotactic if, for some $\weight \in \Rr_{>0}^d$, {the requirements of \dref{d:Pendo1}
hold} on all the faces of $\WW$. For each {such face $\F$
with normal $n_\F$} one can find the cone
 $\BB_\F \subset \Rr_{> 0}^d$ of all $\weight$'s {such that
$\langle n_\F,c^{r,\weight}\rangle < 0$ for all $r \in \R_{n_\F}$.
The \abbr{crn} is thus strongly $(\S,\weight)$-endotactic iff
$\weight \in \bigcap_{\{\F\}} \BB_\F =: \BB$, and
strongly endotactic iff $\BB \neq \emptyset$.}}
\end{remark}

\setcounter{exxx}{1}
\begin{exxx}[continued]
 \begin{figure}
  \centering
  \def\svgwidth{.43\textwidth}
  \begingroup%
  \makeatletter%
  \providecommand\color[2][]{%
    \errmessage{(Inkscape) Color is used for the text in Inkscape, but the package 'color.sty' is not loaded}%
    \renewcommand\color[2][]{}%
  }%
  \providecommand\transparent[1]{%
    \errmessage{(Inkscape) Transparency is used (non-zero) for the text in Inkscape, but the package 'transparent.sty' is not loaded}%
    \renewcommand\transparent[1]{}%
  }%
  \providecommand\rotatebox[2]{#2}%
  \ifx\svgwidth\undefined%
    \setlength{\unitlength}{261.75109863bp}%
    \ifx\svgscale\undefined%
      \relax%
    \else%
      \setlength{\unitlength}{\unitlength * \real{\svgscale}}%
    \fi%
  \else%
    \setlength{\unitlength}{\svgwidth}%
  \fi%
  \global\let\svgwidth\undefined%
  \global\let\svgscale\undefined%
  \makeatother%
  \begin{picture}(1,1.16564357)%
    \put(0,0){\includegraphics[width=\unitlength,page=1]{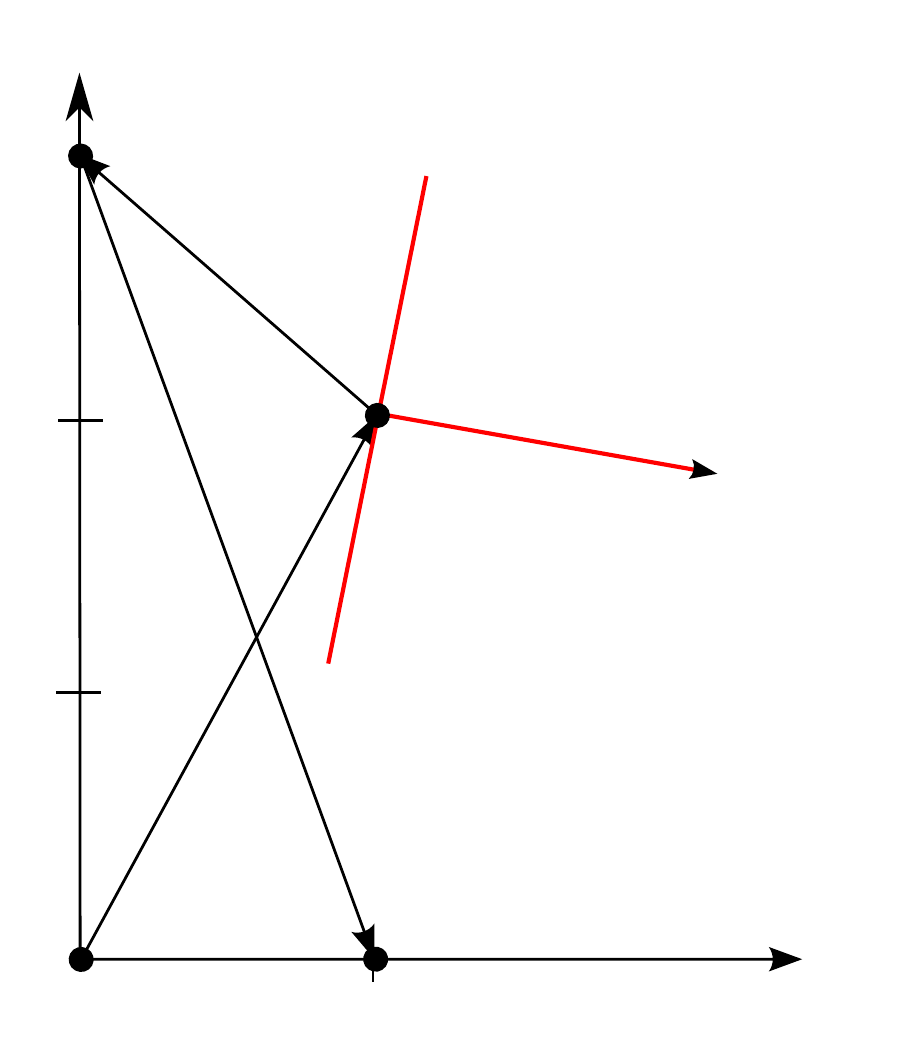}}%
    \put(0.12711614,0.9900558){\color[rgb]{0,0,0}\makebox(0,0)[lb]{\smash{$3B$}}}%
    \put(-0.0071036,1.10062474){\color[rgb]{0,0,0}\makebox(0,0)[lb]{\smash{$B$}}}%
    \put(0,0){\includegraphics[width=\unitlength,page=2]{complex_ex_swapred1.pdf}}%
    \put(0.44428534,0.74021397){\color[rgb]{0,0,0}\makebox(0,0)[lb]{\smash{$A + 2B$}}}%
    \put(0.66686468,0.57413762){\color[rgb]{0,0,0}\makebox(0,0)[lb]{\smash{$w$}}}%
    \put(0.83966758,0.0144651){\color[rgb]{0,0,0}\makebox(0,0)[lb]{\smash{$A$}}}%
    \put(0.42769341,0.12713839){\color[rgb]{0,0,0}\makebox(0,0)[lb]{\smash{$A$}}}%
    \put(0.00848286,0.01780078){\color[rgb]{0,0,0}\makebox(0,0)[lb]{\smash{$\emptyset$}}}%
    \put(0,0){\includegraphics[width=\unitlength,page=3]{complex_ex_swapred1.pdf}}%
  \end{picture}%
\endgroup%
  \caption{The complex diagram of network \ref{e:ex2}. The shaded region is the complex polytope of the network. A vector $w \in \Rr^d$ has been drawn in red, together with its normal hyperplane, in order to identify the $w$-maximal subset of the input complexes $\C_\inn(\S)_w$: The complex $\Ctwo $\,.}
  \label{f:ex2}
 \end{figure}
 Consider the reaction network \eref{e:ex2} described in \fref{f:ex2}. There are $3$
 faces:\\
 $$F_{13} := (\Cthree ,\Ctwo )~,\qquad F_{32} := (\Ctwo ,\emptyset)~,\qquad F_{12} := (\emptyset,\Cthree )~,$$ and
 $3$ reactions: $$r_1 := \{\Cthree \rightharpoonup A\}~, \qquad r_2 := \{\emptyset
 \rightharpoonup \Ctwo \}~, \qquad r_3 := \{\Ctwo  \rightharpoonup \Cthree \}~.$$
 One checks that for the normal $n_1$ to $F_{13}$,
 $\dtp{n_1}{ c^{r_1}}<0$, $\dtp{n_1}{ c^{r_3}}=0$ and $r_2$ is not
 considered because it does not originate in $F_{13}$\,.\\
 On the face $F_{32}$ the scalar products are $\dtp{n_2}{ c^{r_3}}<0$,
 $\dtp{n_2}{c^{r_2}}=0$, and  $r_1$ need not be considered.
 Finally, we have $\dtp{n_3 }{c^{r_1}}<0$,
 $\dtp{n_3}{c^{r_2}}<0$, and  $r_3$ is not considered.\\
 Thus, the system
 $\emptyset \rightharpoonup \Ctwo  \rightharpoonup \Cthree  \rightharpoonup A$
 is strongly endotactic \jpp{(for $\weight = \onevec$)}.
\end{exxx}

{
\noindent We next demonstrate} \jpp{that the class of networks in \dref{d:Pendo1}
is larger than the one introduced in \cite{gopal13}.
\begin{ex} \label{ex:ex1}
 Consider the \abbr{crn} defined by the reactions \begin{equ} \label{e:ex13} A \to \emptyset \to B \to 2A~.\end{equ} Its diagram, shown in \fref{f:diag1}, contains a reaction pointing outward of $\WW$. Therefore, the network does not satisfy the strongly $(\S,\onevec)$-endotactic property from \cite{agazzi17,gopal13}. However, this network does satisfy \dref{d:Pendo1} for $\weight = (1/2,1)$ and is therefore strongly endotactic.
 \begin{figure}
  \centering
   \def\svgwidth{.37\textwidth}
   \begingroup%
  \makeatletter%
  \providecommand\color[2][]{%
    \errmessage{(Inkscape) Color is used for the text in Inkscape, but the package 'color.sty' is not loaded}%
    \renewcommand\color[2][]{}%
  }%
  \providecommand\transparent[1]{%
    \errmessage{(Inkscape) Transparency is used (non-zero) for the text in Inkscape, but the package 'transparent.sty' is not loaded}%
    \renewcommand\transparent[1]{}%
  }%
  \providecommand\rotatebox[2]{#2}%
  \ifx\svgwidth\undefined%
    \setlength{\unitlength}{215.76516724bp}%
    \ifx\svgscale\undefined%
      \relax%
    \else%
      \setlength{\unitlength}{\unitlength * \real{\svgscale}}%
    \fi%
  \else%
    \setlength{\unitlength}{\svgwidth}%
  \fi%
  \global\let\svgwidth\undefined%
  \global\let\svgscale\undefined%
  \makeatother%
  \begin{picture}(1,0.53588278)%
    \put(0,0){\includegraphics[width=\unitlength,page=1]{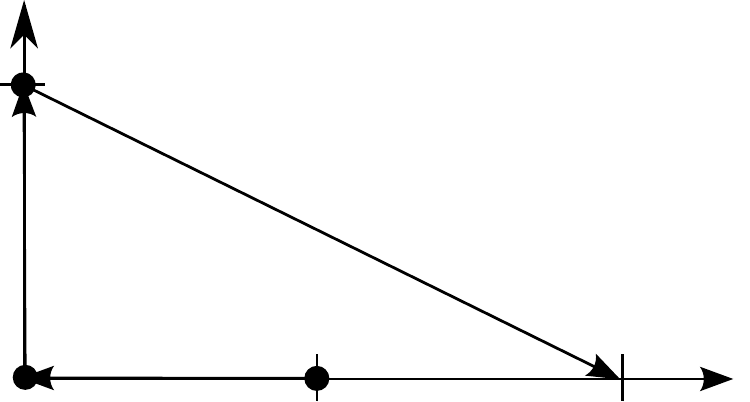}}%
    \put(0.0500676,0.48995914){\color[rgb]{0,0,0}\makebox(0,0)[lb]{\smash{$B$}}}%
    \put(0.88711822,0.05867122){\color[rgb]{0,0,0}\makebox(0,0)[lb]{\smash{$A$}}}%
    \put(0,0){\includegraphics[width=\unitlength,page=2]{diag1.pdf}}%
  \end{picture}%
\endgroup%
  \caption{\jpp{Reaction diagram of the \abbr{crn} \eref{e:ex13}. In this case there is a reaction vector pointing outside of the face with normal vector $(1,1)$. Therefore, this network is not strongly $(\S,\onevec)$-endotactic but is included in the more general \dref{d:Pendo1}.}}
\label{f:diag1}
 \end{figure}
 \end{ex}}

\setcounter{exxx}{3}
\begin{exxx}
 We give two examples of networks that do not meet {some of}
 the requirements of \dref{d:Pendo1}.
 \begin{figure}[h]
  \centering
  \begin{subfigure}[b]{0.31\textwidth}
   \centering
   \def\svgwidth{.9\textwidth}
   \begingroup%
  \makeatletter%
  \providecommand\color[2][]{%
    \errmessage{(Inkscape) Color is used for the text in Inkscape, but the package 'color.sty' is not loaded}%
    \renewcommand\color[2][]{}%
  }%
  \providecommand\transparent[1]{%
    \errmessage{(Inkscape) Transparency is used (non-zero) for the text in Inkscape, but the package 'transparent.sty' is not loaded}%
    \renewcommand\transparent[1]{}%
  }%
  \providecommand\rotatebox[2]{#2}%
  \ifx\svgwidth\undefined%
    \setlength{\unitlength}{266.44174805bp}%
    \ifx\svgscale\undefined%
      \relax%
    \else%
      \setlength{\unitlength}{\unitlength * \real{\svgscale}}%
    \fi%
  \else%
    \setlength{\unitlength}{\svgwidth}%
  \fi%
  \global\let\svgwidth\undefined%
  \global\let\svgscale\undefined%
  \makeatother%
  \begin{picture}(1,1.13153258)%
    \put(0,0){\includegraphics[width=\unitlength,page=1]{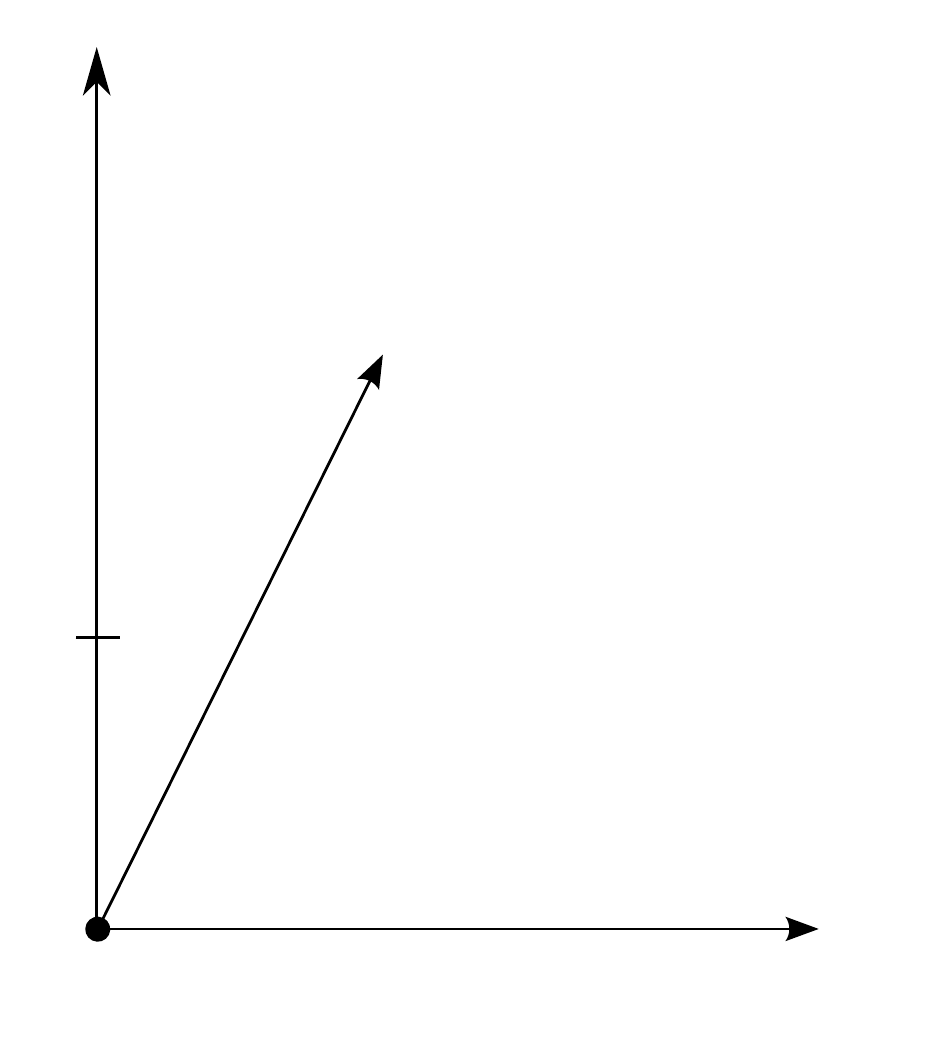}}%
    \put(0.8424902,0.01235612){\color[rgb]{0,0,0}\makebox(0,0)[lb]{\smash{$A$}}}%
    \put(-0.00697854,1.06765839){\color[rgb]{0,0,0}\makebox(0,0)[lb]{\smash{$B$}}}%
    \put(0,0){\includegraphics[width=\unitlength,page=2]{complex_ex11_more_shaded.pdf}}%
    \put(0.74772872,0.1535512){\color[rgb]{0,0,0}\makebox(0,0)[lb]{\smash{$2A$}}}%
    \put(0.42952868,1.01089522){\color[rgb]{0,0,0}\makebox(0,0)[lb]{\smash{$A+3B$}}}%
    \put(0,0){\includegraphics[width=\unitlength,page=3]{complex_ex11_more_shaded.pdf}}%
    \put(0.42327379,0.15120593){\color[rgb]{0,0,0}\makebox(0,0)[lb]{\smash{$A$}}}%
    \put(0.73200473,0.26677292){\color[rgb]{0,0,0}\makebox(0,0)[lt]{\begin{minipage}{0.25414291\unitlength}\raggedright \end{minipage}}}%
    \put(0.40718887,0.76833375){\color[rgb]{0,0,0}\makebox(0,0)[lb]{\smash{$A+2B$}}}%
    \put(0,0){\includegraphics[width=\unitlength,page=4]{complex_ex11_more_shaded.pdf}}%
  \end{picture}%
\endgroup%
  \end{subfigure}
  $\qquad$ 
  \begin{subfigure}[b]{0.57\textwidth}
   \centering
   \def\svgwidth{.9\textwidth}
   \begingroup%
     \makeatletter%
     \providecommand\color[2][]{%
       \errmessage{(Inkscape) Color is used for the text in Inkscape, but the package 'color.sty' is not loaded}%
       \renewcommand\color[2][]{}%
     }%
     \providecommand\transparent[1]{%
       \errmessage{(Inkscape) Transparency is used (non-zero) for the text in Inkscape, but the package 'transparent.sty' is not loaded}%
       \renewcommand\transparent[1]{}%
     }%
     \providecommand\rotatebox[2]{#2}%
     \ifx\svgwidth\undefined%
       \setlength{\unitlength}{380.32722168bp}%
       \ifx\svgscale\undefined%
         \relax%
       \else%
         \setlength{\unitlength}{\unitlength * \real{\svgscale}}%
       \fi%
     \else%
       \setlength{\unitlength}{\svgwidth}%
     \fi%
     \global\let\svgwidth\undefined%
     \global\let\svgscale\undefined%
     \makeatother%
     \begin{picture}(1,0.62072737)%
       \put(0,0){\includegraphics[width=\unitlength,page=1]{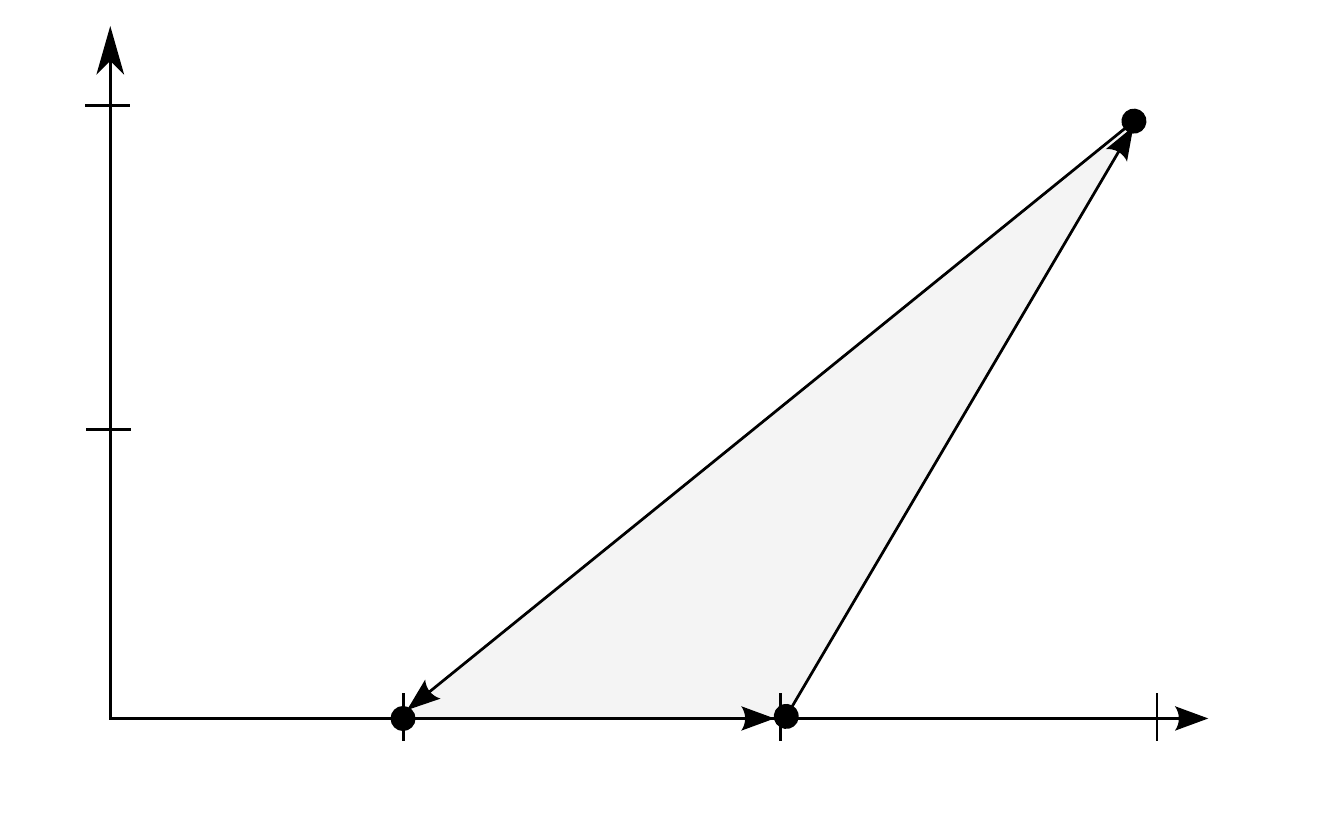}}%
       \put(0.62936611,0.55260684){\color[rgb]{0,0,0}\makebox(0,0)[lb]{\smash{$3A+2B$}}}%
       \put(0.64848937,0.11064697){\color[rgb]{0,0,0}\makebox(0,0)[lb]{\smash{$2A$}}}%
       \put(0.88965505,0.0086562){\color[rgb]{0,0,0}\makebox(0,0)[lb]{\smash{$A$}}}%
       \put(0.22777756,0.10639732){\color[rgb]{0,0,0}\makebox(0,0)[lb]{\smash{$A$}}}%
       \put(-0.00488888,0.57597972){\color[rgb]{0,0,0}\makebox(0,0)[lb]{\smash{$B$}}}%
     \end{picture}%
   \endgroup%
   \end{subfigure}
  \caption{Complex diagrams of the \abbr{crn}s in Example~\ref{e:ex3} (\eref{e:ex31} on the left and \eref{e:ex32} on the right).}
  \label{f:ex3}
 \end{figure}
 \begin{myenum}
  \item The network
  \begin{equ}
   \emptyset \rightharpoonup \Ctwo  \rightharpoonup \Cthree  \rightharpoonup A \quad 2A \rightleftharpoons A+ \Cthree ~,
   \label{e:ex31}
  \end{equ}
  represented in \fref{f:ex3} left, is not strongly endotactic. Indeed,
\jpp{for $n^\star=(1,1/3)$ the normal to the $(A+\Cthree ,2A)$ face,
the set $\R_{n^\star}$ consists of the reactions
$2A \rightleftharpoons A+ \Cthree$ for which
$\dtp{n^\star}{\crw} \jpp{ = \pm (\weight_1-\weight_2)}$.
So, either both reactions are $n^\star$-null,
or one of them is $n^\star$-explosive.}
  \item The network
  \begin{equ}
   A \rightharpoonup 2A \rightharpoonup 3A+2B  \rightharpoonup A
   \label{e:ex32}
  \end{equ}
  represented in \fref{f:ex3} right, is strongly endotactic. \\
  For $\PP = \{A\}$ $(\S,\C,\R(\PP))$ is strongly \jpp{$(\PP,\onevec)$}-endotactic: For $w_\PP = (1) \in \Rr^1$, the reaction $r_+ = \{2A \rightharpoonup 3A+\Cthree \} \in \R(\PP)_{w_\PP}$ is \jpp{$(w_\PP,\onevec)$}-dissipative
  (it has a positive component not in $\PP$) and for $w_\PP = (-1)$ the reaction $r_- = \{A \rightharpoonup 2A\} \in \R(\PP)_{w_\PP}$ is \jpp{$(w_\PP,\onevec)$}-dissipative ($\dtp{w_\PP}{\jp{c^{r_-}}} = -1$). \\
  For $\PP = \{B\}$
  the network is not strongly \jpp{$(\PP,\onevec)$}-endotactic, {but}
  $\R(\PP) = \emptyset$\,.\\
 \end{myenum}
 \label{e:ex3}
\end{exxx}
\dfn{\label{d:asiphonic}A {non-empty}
 subset {$\PP \subseteq \S$} is called a \emx{siphon} if every reaction
 $r \in \R$ with {at least one output from $\PP$ also has
  some input species from $\PP$.}\\
 Furthermore, we say that a \abbr{crn} \net with no siphons $\PP \subset \S$ is \emph{asiphonic}.}

\rmk{\dref{d:asiphonic} comes from the theory of Petri Nets \cite{petri62}, where it is used to characterize systems that can recover from the extinction of any of their components. This definition is equivalent to the one of \emph{exhaustive} networks presented in \cite{kammash14}\label{r:asiphonic}}

\exm{Consider the \abbr{crn}
 \begin{equ}
  A\ \mathop{\rightharpoonup}^{\mathrm{r_1}}\ 2A\ \mathop{\rightharpoonup}^{\mathrm{r_2}}\ \emptyset~.
 \end{equ}
 It is easy to see that this network is strongly endotactic but has a siphon. This possibly implies issues with the irreducibility of the process $X_t^v$ (seen as a Markov chain): When $X^v = 2/v$ there is a nonvanishing probability that the next jump will occur in direction $c^{r_2} = -2$, resulting in the invariant state $X^v = 0$\,.
}

\setcounter{exxx}{3}
\begin{exxx}[continued] Network \eref{e:ex31} is asiphonic: The reaction $\emptyset \to \Ctwo $ has no input from either $\PP = \{A\}, \{B\}$, preventing such spaces from being siphons. Network \eref{e:ex32}, on the other hand, has a siphon $\PP = \{A\}$: All reactions have $A$ as an input.
\end{exxx}

We make the following assumption on the topological structure of \abbr{CRN}s.
{We call \net\ an
 \emph{Asiphonic Strongly Endotactic} (\abbr{ASE}) network if it satisfies
 \begin{axx}[\ad{adapting} {\cite[Ass.~A.2]{agazzi17}}]
  The \abbr{CRN} \net has the properties:
  \begin{myenum}
   \item It is strongly endotactic, as in \dref{d:Pendo1},
   \item It is \emph{asiphonic}, \ie it has no siphon $\PP \subseteq\S$\,.
  \end{myenum}
  \label{a:ase}
 \end{axx}}

\begin{ex}\label{ex:tetrahedron} It is easy to verify that the network
 \begin{equ}
  \emptyset \rightleftharpoons A \rightharpoonup B \rightharpoonup C \rightharpoonup A
 \end{equ}
 is \abbr{ase}: The asiphonic property is verified by the subnetwork
 $\emptyset \rightharpoonup A \rightharpoonup B \rightharpoonup C$,
 while the fact that \net has a single strongly connected component directly implies that the network is strongly endotactic \cite{gopal13}.
\end{ex}

By \rref{r:asiphonic}, \abbr{ase} networks naturally satisfy \aref{a:1} (b). This class of networks automatically satisfies \aref{a:1} (a) as well.\\
This is formulated in our main result:

\trm{[Existence of a Lyapunov function, \ad{extending} {\cite[Proposition 1.12]{agazzi17}}]
 \jpp{For any $\weight \in \mathbb R_{>0}^d$ with $\|\weight\|_1 = d$, let}
 \begin{equ}
  \jpp{\uu_\weight(x)} :=  d + 1 + \sum_{i = 1}^d \weight_i {x}_i  (\log {x}_i - 1) \, : \RR^d \to \Rr_{\ge 1} \,.
  \label{e:lyapunov1}
 \end{equ}
 If the network is \abbr{ASE}, then for the generator $\Lv$ of \eref{e:Lmarkov_1} \jpp{and any vector $\weight$ from \dref{d:Pendo1}, }there exists a constant $b<\infty$ such that for all $\rho>0$ one can find
 a $v^*(\rho)$ such that for all $x $ with $\|x\|_1 < \rho$, we have for all
 $v>v^*$ when
 $x\in (v^{-1} \mathbb N_0)^d$:
 \begin{equ}
  (\Lv \jpp{\uu_\weight^v}) \, (x) \leq
  e^{b v}~.
  \label{e:expdrift2}
 \end{equ}
 \label{t:thm}
}
\noindent In other words, for \abbr{ASE} \abbr{crn}s, the operator $\Lv$ satisfies
\aref{a:1} (a) with \jpp{some $\uu_\weight(x)$} of the form \eref{e:lyapunov1}.  Note that
$\jpp{\uu_\weight(x)} \ge 1$ for all $x$.

In this paper, we provide {a} constructive proof of this result,
\jpp{denoting $\uu_\weight$ by $\uu$ whenever the choice of the vector
$\weight$ is clear from the context.}
To do so, we use a reparametrization of the phase space $\RR^d$ that was introduced in \cite{agazzi17,gopal13} allowing to treat large-$\|x\|_1$ asymptotics of our problem in a particularly natural way.
To introduce such a parametrization, we define \emph{toric rays}, using throughout $w \in \mathbb R^n$, $z \in (\RR^n)^o$ and $\theta\in \mathbb R_{> 1}$ and the operators
\begin{align*}
 \log(z)  & := (\log z_1, \dots , \log z_n) \in \mathbb R^n~,               \\
 z^w      & := (z_1^{w_1},\dots ,z_n^{w_n}) \in (\mathbb R_{+}^n)^o~,       \\
 \theta^w & := (\theta^{w_1},\dots, \theta^{w_n}) \in (\mathbb R_{+}^n)^o~.
\end{align*}
\dfn{[\cite{agazzi17,gopal13}] To each $w$ in the unit sphere $S^{n-1}$ we associate the $w$-\emx{toric ray}
 $$
 T^w = \bigcup_{\theta > 1}\,\theta^w \subset \RR^n \,.
 $$
 We also introduce the toric-ray parameters
 \begin{align}
  \theta(z) := \exp(\|\log(z)\|_2)\,, \qquad
  w(z) :  & =\frac{1}{\log \theta(z)} \log(z)\,, \nonumber \\
  \pc{\theta,w} : \pc{\RR^{n}}^o \setminus \{1\} \rightarrow  \Rr_{> 1} \times S^{n-1}~,
  \quad z & = \theta(z)^{w(z)}~\,.
  \label{dfn:topic-p}
 \end{align}
 \label{d:toric-ray}}
\begin{remark}
 The parametrization \eref{dfn:topic-p} does not cover the point $z = (1,\dots,1)$, but it is irrelevant for our asymptotic analysis.
\end{remark}

This parametrization allows to separate in a natural way regions of phase space where the behavior of the vector field \eref{e:ma} is, asymptotically in $\theta$, particularly uniform, as shown in \fref{f:backparam}. We use this decomposition of phase space to establish \eref{e:expdrift2} in each of these regions separately, for the Lyapunov function \jpp{$\uu_\weight(x)$} of \eqref{e:lyapunov1}.

\begin{figure}[t]
 \centering
 \includegraphics[scale=0.25]{\figuresfolder/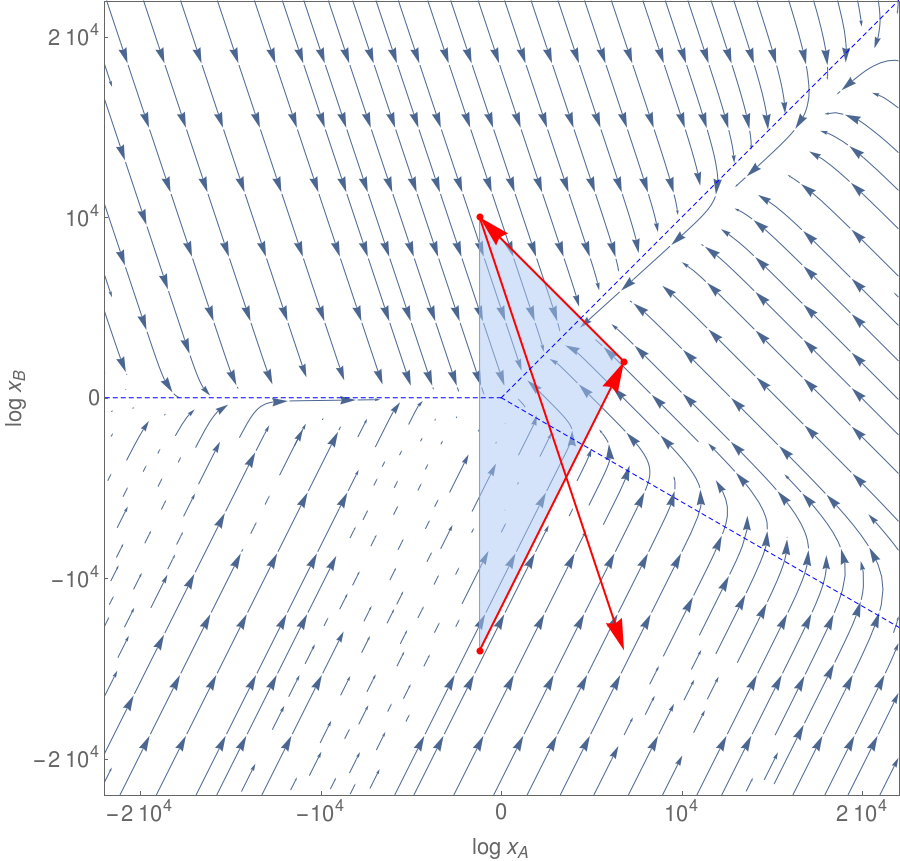}
 \caption{Dominant components of the vector field of \eref{e:vfex1} with $k_1 = k_2 = k_3 = 1$\, under reparametrization of the phase space as $(\RR^d)^o \to \mathbb R_{> 1} \times S^{d-1} \simeq \mathbb R^d\setminus \{1\}$ as in \eref{dfn:topic-p}. Note that the tangent bundle of the manifold has not been reparametrized.  The red vectors represent the reactions of the network \eref{e:ex2}, and the complex polytope is displayed in light blue. The normals to the edges of this polytope separate the space radially into dominance regions, where the direction of the vector field is asymptotically constant and corresponds to the direction of the dominant reaction.}
 \label{f:backparam}
\end{figure}

\label{r:explainTandU} We next discuss why this reparametrization and
the associated Lyapunov function \jpp{$\uu_\weight(\cdot)$} are useful.
Note that along a $w$-toric ray
\begin{equ}
 \jpp{\partial_i U_\weight(\theta^w) = \weight_i \log(\theta^{w_i}) =
 (\log \theta) w_i \weight_i }\,,
 \label{e:grad-U}
\end{equ}
while the derivative of the \abbr{ode} \eqref{e:ma} at a point on such a ray is
\begin{equ}
 \frac{\d x}{\dt} \Big|_{x = \theta^w} = \sum_{r \in \R}
 \lambda_r(x) c^r \Big|_{x = \theta^w} = \sum_{r \in \R} k_r
 \left(\theta^w\right)^{c^r_{\rm in}} c^r = \sum_{r \in \R} k_r \theta^{\langle w , c^r_{\rm in} \rangle} c^r~\,.
 \label{e:asymptotic}
\end{equ}
Thus, at $x = \theta^w$ the time derivative of \jpp{$\uu_\weight(\xt)$} for the solution $\xt $ of \eref{e:ma} is
\begin{equ}
 \frac{\d}{\d t}\jpp{\uu_\weight}(\xt )\Big|_{x = \theta^w} = \langle \nabla \jpp{\uu_\weight}(x) , \frac{\d x}{\d t} \rangle \Big|_{x = \theta^w}= (\log \theta) \sum_{r \in \R} k_r  \langle w , \crw \rangle \theta^{\langle w , c^r_\inn \rangle}~\,.
 \label{e:lyapunovdominance}
\end{equ}

\begin{figure}[t]
 \centering
 \includegraphics[scale=0.25]{\figuresfolder/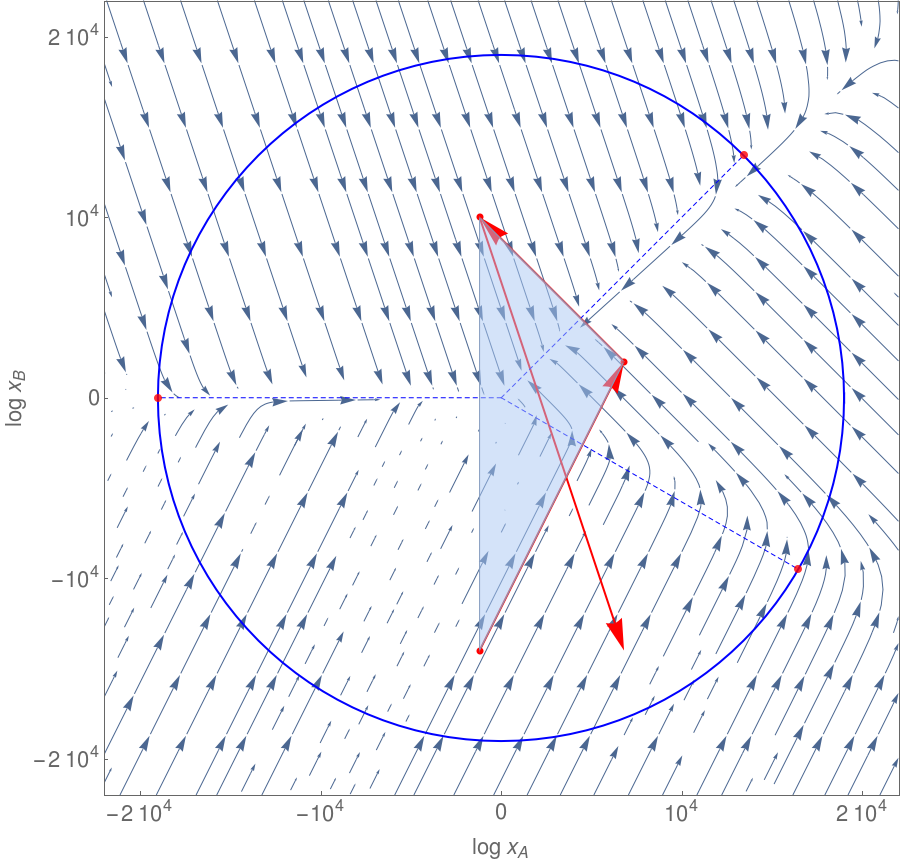}
 \caption{\jpp{When $\weight = \onevec$, }the negativity of \eref{e:lyapunovdominance} uniformly on $w \in S^{d-1}$ can be visualized as the condition for a large enough radius $\rho > 0$, the $d$-ball $\mathcal B_\rho(0)$ \jp{absorbs} the vector field of \eref{e:ma}  in the new parametrization. As an example, we show that this condition is satisfied by the vector field of \eref{e:vfex1} with $k_1 = k_2 = k_3 = 1$ for $\rho = 1.9 \times 10^4$. The complex polytope is also plotted, and different dominance regions are separated by the dashed lines.}
 \label{f:backparam2}
\end{figure}

\jpp{The choice of $\uu_\weight$ is natural
for a strongly $(\S,\weight)$-endotactic \abbr{crn}}
because of \eref{e:lyapunovdominance}. Indeed,
for fixed $w$ and $\theta \gg 1$ the sum on the \abbr{rhs} of
\eref{e:lyapunovdominance} is dominated by the set $\R_{w}$ of those
reactions $r$ which
maximize $\langle w , c^r_\inn \rangle$.
As is seen in \fref{f:backparam}, far away from the origin the vector
field is divided into cells with the flow lines parallel to each
other. There are, in this figure, $1$-dimensional cells (the isolated
lines) and $2$-dimensional cells (the 3 regions in the ``Mercedes''
star).
Each $2$-dimensional cell corresponds to a set of $w \in S^{d-1}$ for which
one\footnote{It can happen that several reactions maximize the scalar
product. This will be discussed later.}
\emph{fixed} reaction $r$ is dominating.
Furthermore, the direction of the vector field in that region \jp{(in the original parametrization)} is
given by the jump vector $c^r$ of that reaction.
In strongly \jpp{$(\S,\weight)$-endotactic} \abbr{crn}s, where at least one such reaction contributes negatively to this sum
by having $\langle w , \crw\rangle < 0$, and no other reaction $r$ in
$\R_w$ contributes positively to it, the \abbr{lhs} of
\eref{e:lyapunovdominance} will also be negative for all large enough $\theta$.

In \cite{gopal13}, this geometry was manifestly observed by the
authors. The difficulty is to show that the heuristic ideas apply
uniformly in $w$. As can be seen from \fref{f:backparam}, the
uniformity seems obvious in the interior of the $2$-dimensional regions,
while there are deviations from parallelity near the boundaries of
these regions. The question is then whether these curved parts of the
vector field
point in the right direction or not. We will show that this is
indeed the case (as shown in \fref{f:backparam2}), by giving a constructive proof. This hopefully
clarifies what is really going on in the logarithmic representation of
the chemical reaction.

\rmk{The logarithmic picture does not only apply in the regions where
 $w_i$ is positive, but also in the opposite case. This
 corresponds to approaching, as $\theta\to\infty$\,, a boundary of
 vanishing concentration (in the variable $x_i$). In those directions, the negativity of \eref{e:lyapunovdominance} establishes a (polynomial) lower bound on the rate of escape of the \abbr{ODE} solution from the boundary within compact sets \cite[Lemma 2.4]{agazzi17} provided that the solution does not blow up in finite time. The \abbr{ldp} then implies that the same estimate holds for $X_t^v$ starting at the boundary up to an exponentially decaying probability in $v$.}

\subsection{Proof of the main theorem (\tref{t:thm})}

Following the intuition developed in the previous section the proof is essentially a sequence of asymptotic estimates. They are obtained by covering \jp{the} phase space by open sets which allow to distinguish the different cells and their (carefully) chosen neighborhoods. There is a close relationship between our visualization of the problem and the notion of \emph{spherical image} of polytopes as developed by Alexandrov in \cite[§1.5]{alexandrov}.
This is a partition of the sphere generated by the set of vectors orthogonal to the polytope, forming cones of dimension $d-j$ on the $j$-dimensional faces of the polytope.
If $\WW$ is the complex polytope of the network \net,
we denote by $\WW_j$ the collection of
all its $j$-dimensional closed faces. For any $j < d$, the elements of $\WW_j$ are indexed by the set $\I_j$ and denoted by $\WW_{j,\ii}$ for $\ii \in \I_j$, so that $\WW_j=\cup_{\ii \in \I_j} \WW_{j,\ii}$.
We also use the notation $\partial \WW_{j,\ii}$
for the collection of $(j-1)$-dimensional faces of $\WW$, if any, constituting the boundary of $\WW_{j,\ii}$, and set $\I^*:=\{(j, \ii)~:~j \in (0, \dots d-1),~\ii \in\I_j \}$.

Throughout, for each $j$-dimensional face $\WW_{j,\ii}$ we denote by $\mathcal N(\WW_{j,\ii})$
the set of normals to the $(d-1)$-dimensional faces \jp{$\WW_{d-1,\kk'}$}
with $\WW_{j,\ii} \subseteq \jp{\WW_{d-1,\kk'}}$ and
define the dual $\WW_{j,\ii}^*$ as the $(d-j-1)$-dimensional intersection
\begin{equ}\label{e:Fs-def}
 \WW_{j,\ii}^* := \text{Co}(\mathcal N(\WW_{j,\ii})) \cap S^{d-1}~,
\end{equ}
where $\text{Co}(A)$ denotes the open conic hull of $A$.
This is precisely the set of $d$-dimensional
unit normal vectors exposing the facet $\WW_{j,\ii}$.
For any convex polytope $\WW$, the collection $\{\WW_{j,\ii}^*\}$ forms
a unique partition $\WW^*$ of $S^{d-1}$ called the spherical image
of $\WW$ (see \cite[§1.5]{alexandrov}). To summarize, we have
$$
\WW^* :=
\{\WW_j^*\}_{j \in \{0,\dots,d-1\}}\,, \qquad \text{with} \qquad  \WW_j^* := \{\WW_{j,\ii}^*\}_{\ii \in \I_j}~.
$$
\rmk{The spherical polyhedral complex $\WW^*$, used for the study of the asymptotic behavior of polynomials, is well known in the literature under different names. For instance, in tropical calculus it is referred to as the \emph{Bieri-Groves complex} \cite{bieri84} (or as \emph{Bergman fan} \cite{bergman71} for the cone over such a spherical complex), while in algebraic geometry this object arises from the intersection of the so-called \emph{normal fan} \cite{strumfels15,strumfels05} with the unit sphere.}

Using the following parameters
we next cover $S^{d-1}$ by suitable neighborhoods of the elements of $\WW^*$.
\begin{definition}\label{def:416} We
define, for each $j \in (0, \dots, d-1)$,
 \begin{equ}\label{e:choiceofdeltas}
  \begin{cases}\delta_j(\theta) &:= C_{2j}/\log \theta
     \\ \epsilon_j(\theta)& := C_{2j+1}/\log \theta
  \end{cases} \qquad \text{with} \qquad \begin{cases}C_{2j+1} &= K_0 e^{2 c_* C_{2j}}\\C_{2j+2}& = C_{2j+1} + 10^{-3} \end{cases}~,
 \end{equ}
 for some \jpp{finite $c_*$,}
  $C_0 > 0$, $K_0 \ge 1$ that are specified in the sequel (see
 \jpp{\eqref{e:cstar-def} and} \eqref{e:K0-choice}).
\end{definition}

\noindent
For $\mathcal A \subset S^{d-1}$ let $\mathcal A^\delta := \{w \in S^{d-1}~:~\inf_{v \in \mathcal A}{\|v-w\|_2} < \delta\}$ and for the positive parameters
of \dref{e:choiceofdeltas} we define
$(\WW_{j,\ii}^*)^{\epsilon_j, \delta_j} := (\WW_{j,\ii}^* \setminus (\partial \WW_{j,\ii}^*)^{\epsilon_j})^{\delta_j}$ for $j<d-1$ and
$(\WW_{j,\ii}^*)^{\epsilon_j, \delta_j} := (\WW_{j,\ii}^*)^{\delta_j}$
when $j=d-1$.
This induces the covering
\begin{equ} \label{e:covering} S^{d-1}=\bigcup_{j=0}^{d-1}\bigcup_{\ii \in \I_j}
 (\WW^*_{j,\ii})^{\epsilon_j, \delta_j} \,,
\end{equ}
where having $0 < \delta_j < \epsilon_j$ guarantees that
for $j<d-1$, each open set
$(\WW^*_{j,\ii})^{\epsilon_j, \delta_j}$
does \emph{not} intersect the boundary of $\WW^*_{j,\ii}$, see \fref{fig:partition}. Specifically,
through our choice of the $\theta$-dependent $\{\epsilon_j,\delta_j\}$
of \dref{def:416}, we cover $S^{d-1}$ by neighborhoods that approach,
asymptotically in $\theta$, the faces of $\WW^*$.

For any subset of non-extinct species
$\PP \subseteq \S$, the partitions
introduced above naturally
generalize to the $d_\PP$-dimensional complex polytope $\WW(\PP)$ of
the network
$(\S,\C,\R(\PP))$. In particular,
we write
$\WW(\PP) = \{\WW(\PP)_{j,\ii}~:~
(j,\ii) \in \I(\PP)^* \}$, where
$\Ip^* := \{(j, \ii)~:~ j \in (0, \dots, d_\PP-1), \ii \in \I(\PP)_j \}$ and $\I(\PP)_j$
indexes the $j$-dimensional faces of $\WW(\PP)$.
Moreover, we denote by $\WW(\PP)^* = \{\WW(\PP)^*_{j,\ii}\}
$
the spherical image
of $\WW(\PP)$
on $S^{d_\PP-1}$ and by $\{(\WW(\PP)_{j,\ii}^*)^{\epsilon_j,\delta_j}\}$ the corresponding covering of $S^{d_\PP-1}$.

\begin{figure}[t]
 \centering
 \includegraphics[width=.55 \textwidth]{\figuresfolder/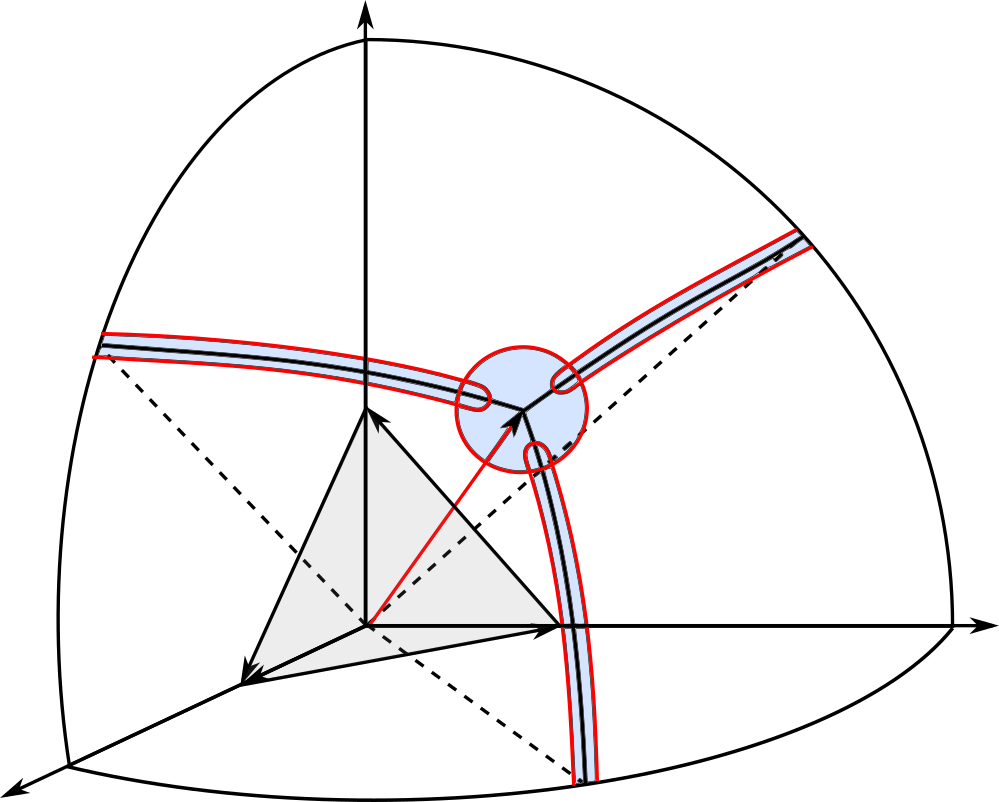}
 \caption{Covering of the positive orthant of $S^2$ by sets $\{\WWed\}_{j \in (0,1,2), \ii \in \I_j}$ for the network in Example~\ref{ex:tetrahedron}, whose complex polytope $\WW$ is drawn, rescaled, in grey. The only element of $\WW^*_2$ in this orthant is the vertex corresponding to the equilateral face of $\WW$, identified by the vector $n = (1,1,1)/\sqrt 3 \in S^2$
  drawn in red in the figure. The $\delta_2$ neighborhood of this
  point, $(\WW_{2,0}^*)^{\delta_2}$, is circled in red. The sets
  $\{(\WW_{1,\ii}^*)^{\epsilon_1,\delta_1}\}_{\ii \in \I_1 }$,
  containing the part of $\WW_{1,\ii}^*$ that has not been covered by
  $\{(\WW_{2,\ii}^*)^{0,\delta_2}\}_{\ii \in \I_2}$ are also shown in red. The sets
  defined by the complement of this partial covering, one for each vertex of $\WW$, are finally contained in
  $\{\WW_{0,\ii}^{\epsilon_0,\delta_0}\}_{\ii \in \I_0}$, not shown in the
 picture.}
 \label{fig:partition}
\end{figure}
We proceed to prove \tref{t:thm} by showing for each face $\WW_{j,\ii}$ the (somewhat stronger) bound
\begin{proposition}\label{p:lyap}
 For any $j,\ii$\,, \jpp{some $\weight \in \mathbb R_{>0}^{d}$ from \dref{d:Pendo1}} and sufficiently large ${\rho_0 \ge 3 d}$, one has
 for all $w \in (\WW_{j,\ii}^*)^{\epsilon_j, \delta_j}$,
 $\rho > \rho_0$, $v>v^*(\rho):= e^\rho$ and
 $x=\theta^w \in (v^{-1} \mathbb N_0)^d$ with
 $\rho_0< \|x\|_1 <\rho$, that
 \begin{equ} \label{e:lyap}
  {\Lv \uu_\weight^v(\theta^w) = v \sum_{r \in \R}
   \Lambda_{r,v}(\theta^w) \pc{\uu_\weight^v(\theta^w + v^{-1} c^r) -
    \uu_\weight^v(\theta^w)} < 0~.} \end{equ}
  \end{proposition}

  \begin{definition}\label{d:st-cond}
   Throughout, we refer to the relations set in \pref{p:lyap}
   between the quantities $\rho_0, \rho, w, x, v^*$ and $v$ as the \emph{standard
   conditions}. These choices will continue to hold throughout the section,
   where in particular $\log \theta \ge 1$
   (since $d \theta \ge \|x\|_1 > 3 d$).
  \end{definition}
  \rmk{\label{r:proofequivalence} Proving \pref{p:lyap} results with
   \eref{e:expdrift2}. Indeed, \jpp{for any $\weight \in \Rr_{>0}^d$ as in \dref{d:Pendo1},} fixing $\rho_0$ large enough
   (that accommodates all $j,\ii$ pairs),
   for $v > v^*(\rho)$ and $\rho_0 < \|x\|_1 < \rho$
   the inequality \eref{e:expdrift2} is satisfied by \eref{e:lyap},
   while for $\|x\|_1 \leq \rho_0$ we get \eref{e:expdrift2} from
   \begin{equs}
    \Lv \uu^v(x) &= v\sum_{r \in \R} \Lambda_{r,v}(x) \pc{\uu^v(x + v^{-1} c^r) - \uu^v(x)}\\&  \leq m \, \sup_{r \in \R, \|x\|_1 \leq \rho_0}
    \big\{ \lambda_r(x) \big\} \,
    \sup_{\|x\|_1 \leq \rho_0+ \max_r \|c^r\|_1} \Big\{ v \uu(x)^v \Big\}
    \leq e^{b(\rho_0)v}~.\end{equs}}

   \rmk{\jp{\dref{def:416}} can be understood
    heuristically by inspection of \eref{e:lyapunovdominance}: For every face $\WW_{j,\ii}^*$ through our choice of $\{\epsilon_j, \delta_j\}$ we are ensuring that the reaction rate of dominant reactions is larger than any other by at least a factor of $e^{K_c C_{2j+1}}$ for a constant $K_c > 1$. This choice also guarantees that the scalar product contribution, \ie $\dtp{w}{\crw} \log \theta$, of at least one among such reactions is negative and bounded away from zero, while all the others are at most $C_{2j}$ (from the definition of strongly endotactic \abbr{crn}). Choosing $C_{2j+1}$ large enough \abbr{wrt} $C_{2j}$ establishes the desired bound \eref{e:lyap}. }

   \subsubsection{Face-dependent estimates}

   To prove \pref{p:lyap}, consider now a \jpp{fixed $\weight \in \mathbb R_{>0}^{d}$ as in \dref{d:Pendo1} and a} fixed face $\F = \WW_{j,\ii}$ with spherical image $\F^* = \WW_{j,\ii}^*$. For $w \in \F^*$ we define a partition of $\R = \R_+ \cup \R_0 \cup \R_-$ (which depends on $\ii$\jpp{, $j$ and $\weight$):}
   \begin{myitem}
    \item $\R_-$ are those reactions with $c^r_\inn \in \F$ that are \jpp{$(w,\weight)$}-dissipative,
    \item $\R_0$ are those reactions with  $c^r_\inn\in \F$ that are \jpp{$(w,\weight)$}-null,
    \item $\R_+$ are those reactions for which $c_\inn^r\notin \F$.
  \end{myitem}
  The case $c^r_\inn\in\F$ and $\dtp{w}{\jpp{\crw}}>0$ is excluded because the \jpp{network is}
  assumed to be strongly \jpp{$(\S,\weight)$-endotactic.} This property also ensures that the set  $\R_-$ is nonempty for all $w$. This partition of $\R$
  naturally generalizes to a partition of $\R(\PP)$ generated by
  any face of
  the polytope $\WW(\PP)$ and the corresponding dual from $\WW(\PP)^*\,$.
  \begin{remark}\label{r:duality}
   Consider the set $\R_\F=\pg{r\in\R~:~ c_\inn^r\in\F}$.
   Note that $r \in \R_\F$ maximizes $\dtp{w}{c_\inn^r}$ for
   all $w \in \Wji$. That is, $\dtp{w}{c_\inn^{\bar r} - c_\inn^{r}} < 0$
   whenever $\bar r\notin \R_\F$. In particular, $\R_w=\R_w(\S) = \R_\F
   $ for any such $w$. Further, by continuity
   $\dtp{n}{c_\inn^{\bar r} - c_\inn^r} \le 0$ for all $n \in \NN(\WWji)$, with equality for a $\bar r \notin \R_\F$ iff ${c_\inn^{\bar r}} \in \WW_{d-1,\kk}$
   and $\F \subset \partial \WW_{d-1,\kk}$
   for some face $\WW_{d-1,\kk}$ to which $n$ is normal.
  \end{remark}
  \begin{remark}\label{r:dec-Fstar}
   By \rref{r:duality} the preceding decomposition does not depend on $w\in \Wji$
   (\ie $\dtp{w}{\crw} < 0$ for any $r \in \R_-$ and $w \in \Wji$).
   Indeed, if $\dtp{w_0}{\crw}=0$ and $\dtp{w_1}{\crw}<0$ for some $w_0, w_1 \in \Wji$ and $r \in \R_\WWji$,
   then since $\F^*$ is open $w_2 := w_0-\eta w_1 \in \F^*$ for some $\eta >0$ small enough, with $\langle w_2, \crw \rangle > 0$ and $r \in \R_{w_2}$
   contradicting our assumption
   that the \abbr{crn}
   is strongly \jpp{$(\S,\weight)$-endotactic.}
  \end{remark}

  Defining \begin{equ}\label{e:Q}
  {\qq{r}{x} := {\uu(x)} \big[ \uu^v(x + v^{-1}c^r)/\uu^v(x)-1\big]}~,
  \end{equ}
  we will obtain \eref{e:lyap} by showing that under our standard conditions, for any $r_- \in \R_-$\,, and all $w \in (\F^*)^{\epsilon, \delta}$\,, we have ${\qq{r_-}{x}} < 0$ and
  \begin{equ}\label{e:dividesums}
   \sum_{{r \in \R_+ \cup \R_0}} \frac{\Lambda_{r,v}(\theta^w)}{\Lambda_{r_-,v}(\theta^w)}  {\qq{r}{\theta^w}}
   < -{\qq{r_-}{\theta^w}}~.
  \end{equ}

  As the map \eref{dfn:topic-p} has $(\RR^d)^o$ as domain, in order to establish \tref{t:thm} on $\partial \RR^d$ the program summarized above has to be carried out on each $\PP \subseteq \S$ separately.

  \subsubsection{Bounding terms}

  We devote this section to bounds on the summands of
  \eref{e:lyap} depending on their classification into
  $\R(\PP)_+, \R(\PP)_0$ or $\R(\PP)_-$. We do this in two phases. First, in \cref{r:noindicatoratboundary}, we show that under our standard conditions
  there exists {$\Clambda<\infty$} such that for all $r\in\R(\PP)$, $w_\PP \in S^{d_\PP-1}$:
  \begin{equ}\label{e:lambdas}
   \frac{\Lambda_{r,v}(\theta^{w_\PP})}{\Lambda_{r_-,v}(\theta^{w_\PP})} \leq \Clambda \theta^{\dtp{w_\PP}{c_\inn^r - c_\inn^{r_-}}}~.
  \end{equ}
  Then, in Lemmas~\ref{l:epscr}--\ref{l:pointingoutside}, we bound the terms {$\qq{r}{x}$} from above. For $r \in \R(\PP)_0$ or $\R(\PP)_+$ this limits the possibly positive contribution of such terms, while ensuring a negative enough contribution for $r_- \in \R(\PP)_-$. We call those terms the
  \emph{Lyapunov} terms, as opposed to the \emph{monomial} terms in
  \eref{e:lambdas}.

  We now proceed with the first part of our program. Fixing $\PP \subseteq \S$ and
  $(j,\ii) \in \Ip^*\,$,
  consider $w = w_\PP \in (\WW(\PP)_{j,\ii}^*)^{\epsilon_j, \delta_j}$.

  The estimate from
  \cite[(3.18)]{agazzi17} presented in \lref{l:key2} is
  crucial for the understanding of the dynamics near vanishing
  densities when one establishes \eref{e:lambdas}.
  We illustrate the ideas in the proof of this result for the $2$-dimensional case of \fref{f:ex2}.
  Given any vector $w$ of length 1, we say that it \emph{exposes} the
  reaction(s)
  originating in the $0$-dimensional face of
  $\WW$ to which $w$ is dual. In the example, writing $w=(\cos \phi,\sin
  \phi)$ we see that $w$ exposes
  \begin{equa}
   \cornertwo ,& \text{ if } \phi\in(-\pi/6,\pi/4)~,\\
   \cornerthree ,& \text{ if } \phi\in(\pi/4,\pi)~,\\
   (0,0),& \text{ if } \phi\in(\pi,11\pi/6)~.\\
  \end{equa}
  Here, $\cornertwo $ is the origin of the reaction $\Ctwo \to \Cthree $, and similarly
  for the others.
  For our example, the 3 cases have, respectively, 0, 1, or 2 points on
  an axis.\\
  Consider one of the arcs above (or its exposed reaction). We call such
  an arc \emph{non-critical} if its closure contains a direction which coincides
  with a negative axis, and otherwise, critical.
  So $\cornerthree $ is non-critical, as it  contains the negative $A$ axis;
  $(0,0)$  is non-critical, containing the negative $A$ and $B$ axis;
  $\cornertwo $ is critical.

  The issue is now that the discrete reaction rates \eref{e:jrates} may
  differ significantly from the
  continuous ones \eref{e:rates}. For the example, we need to compare (setting \abbr{wlog} $k_1=k_2=k_3=1$)
  \begin{alignat*}{3}
   \cornertwo :   & \quad  x_Ax_B^2~, &   & \text{ and ~~} {v^{-3}}\binom{v x_A}{1}\binom{v x_B}{2}1!\,2!~, \\
   \cornerthree : & \quad x_B^3~,     &   & \text{ and ~~} {v^{-3}}\binom{v x_B}{3}3!~,                     \\
   (0,0):         & \quad  1~,        &   & \text{ and ~~} 1~,                                              \\
  \end{alignat*}
  depending on the continuous model, respectively the discrete one in
  volume $v$.

  We start with the non-critical case $(0,0)$. Here, the rates are
  obviously equal. In  the case $\cornerthree $, which is also non-critical, we
  see that the rates do \emph{not} depend on $x_A$: $(c_\inn^r)_A = 0~$. In the corresponding
  arc, $w_B \ge 0$. Hence, $x_B \geq 1$ and
  $$
  \frac{{v^{-3}}\binom{v x_B}{3}3!}{x_B^3} = 1 - \mathcal
  O\pc{\frac{1}{v x_B}} ~,
  $$
  when $v$ is large.

  The critical case $\cornertwo $ is more interesting: If the angle $\phi$ is
  positive, then both components of $\theta^w$ diverge to $\infty$ as
  $\theta\to\infty$, and therefore
  \begin{equ}\label{e:q}
   \frac{\Lambda_{r,v}(x)}{\lambda_{r}(x)}\equiv\frac{{v^{-3}}\binom{v x_A}{1}\binom{v x_B}{2}1!\, 2!}{x_A x_B^2} = 1-\mathcal O\pc{\frac{1}{v x_B}} ~.
  \end{equ}
  However, when $\phi\in[-\pi/6,0]$, a more careful estimate is needed.
  We reparameterize $v$ and $x$ as follows:
  We fix $\rho$, a distance from the origin, and require, as in our standard conditions,
  \begin{equ}\label{e:star}
   v \geq e^\rho ~,
  \end{equ}
  and
  \begin{equ}\label{e:starstar}
   (x_A,x_B)=\theta^w, \text{ with } \rho \geq \|\theta^w\|_1 \geq \theta^{1/2}~,
  \end{equ}
  where in the last inequality we have used that for our choice of $\phi$, we have $w_A > 1/2$.
  In this case, the orbit of $\theta^w$ approaches the $\{B=0\}$ axis as
  $\rho$ grows, but the
  distance from the $\{B=0\}$ axis is \emph{bounded below} by
  $\theta^{-\cos(\pi/6)} = \theta^{-1/2}$.
  Therefore, we again find that
  \begin{equ}
   vx_B \ge e^\rho \cdot (2\rho^2)^{-1/2}~,
  \end{equ}
  which again grows beyond bounds as $\rho\to\infty$.
  Therefore, \eref{e:q} holds again, with a r.h.s.~$1-\OO(\rho e^{-\rho})$.
  We will use a slightly stronger variant:
  Since $e^\rho \theta^{-1/2}$ diverges and rate terms behave polynomially in $\rho$, we find for any $\mu\in(0,1)$
  $$\lim_{\rho \to \infty}\theta^{w_B} v^\mu = \infty~,$$
  which we will use in the form
  \begin{equ}\label{e:theequ}
   \lim_{\rho \to \infty} \mu \log v + w_B \log \theta = \infty~,
  \end{equ}
  with our choice of \eref{e:star} and \eref{e:starstar}.

  Having motivated the example in all detail, the extension to an
  arbitrary chemical network is straightforward, but we need to
  generalize properly the critical case. We do so in the following lemma, where we establish another bound similar to \eref{e:theequ} for later use.
  \begin{lm}\label{l:key2}
   Consider a face $\F \in \WW(\PP)$ and the corresponding decomposition of $\R(\PP)$. Then for any
   $\mu \in (0,1]$, $M < \infty$, there exists
   $\rho_0(\mu, M, C_{2j})$, such that
   under our standard conditions, for $w \in (\F^*)^{\epsilon, \delta}$ and
   $x=\theta^w$,
   \begin{equ}\label{e:minus-wk}
    \mu \log v - w_i \log \theta > M \,.
   \end{equ}
   Under the same conditions, either
   $(c_\inn^r)_i =0$ for all $r \in \R(\PP)_- \cup \R(\PP)_0$ or
   \begin{equ}\label{e:plus-wk}
    \mu \log v + w_i \log \theta > M \,.
   \end{equ}
  \end{lm}
  \begin{proof} Setting $\beta := 1/\sqrt{d_\PP}$ and
\begin{equation}\label{e:cstar-def}
 c_* := \max_{r,r_- \in \R} \{\|c_\inn^r-c_\inn^{r_-}\|_2,\|\crw\|_2\} \,,
\end{equation}
we fix $w \in (\F^*)^{\epsilon, \delta}$ and
prove the desired results in the following cases:
   \begin{myitem}
    \item
    If $\max_\ell \{w_\ell\} \ge
    \mu \beta/(2 c_*) =: \eta$, then we have $d \theta \ge \|x\|_1 > \rho_0$,
    $\log v > \|x\|_1 \ge \theta^\eta$ and it follows that
    \begin{equ}
     \mu \log v {- |w_i|}
     \log \theta > \inf_{\theta > \rho_0/d} \{ h(\theta) \}
     > M \,,
     \label{e:p1key2}
    \end{equ}
    for $h(\theta):=\mu \theta^\eta - \log \theta$ and
    $\rho_0=\rho_0(\mu,M)$ large enough, concluding the proof in this regime.
    \item Alternatively,
    $\max_\ell \{w_\ell\} {< \eta < \beta}$
    requires $w_{\ell'} \le -\beta$ for some $\ell'\in\PP$\,. Consequently, as
    $v x_{\ell'} \in \mathbb N$, we have that $v
    \geq \theta^\beta$. With $\eta \le \eta c_* = \mu \beta/2$, we
    obtain
    \[
     \mu \log v - w_i \log \theta \ge (\mu \beta - \eta) \log \theta
     > M \,,
    \]
    provided $\rho_0(\mu,M)$ is large enough, which
    proves \eqref{e:minus-wk}.
    To prove \eref{e:plus-wk}, we fix some $r\in \R(\PP)_-\cup\R(\PP)_0\,$
    with $(c_\inn^r)_i > 0$
    and show that \eqref{e:plus-wk} holds. Recall from
    \aref{a:ase} (b) that $c_\inn^{r_\star}=0$
    for some $r_\star \in \R\,$. Hence,
    we have from \rref{r:duality}
    that $0 \le \dtp{w'}{c_\inn^r}$ for all $w' \in \F^*$.
    Having $w' \in \F^*$ with $\|w-w'\|_2 < \delta$
    yields
    $$
    0 \le \dtp{w'}{c_\inn^r} \le
    \dtp{w}{c_\inn^r} + \delta c_*
    < w_i (c_\inn^r)_i + (\eta + \delta) c_* \,,
    $$
    which in turn implies that $w_i > - (\eta + \delta) c_*$. Combining this with
    $\delta = C_{2j}/\log \theta$ (see \dref{def:416}), we finally have
    $$
    \mu \log v + w_i \log \theta > (\mu \beta -\eta c_*)  \log \theta -
    c_* C_{2j} \geq M~,$$
    provided that $\rho_0 \geq d \exp(2(M+ c_* C_{2j})/(\mu \beta))$. \qed
  \end{myitem}\let\qed\relax
 \end{proof}

 \noindent We show that \eref{e:lambdas} is a direct consequence of
 \lref{l:key2}.
 \begin{coro}\label{r:noindicatoratboundary}
  Under our standard conditions, \eqref{e:lambdas}
  holds for some
  $\Clambda$ finite, any $(j,\ii) \in \Ip^*$
  and all $r_- \in \R(\PP)_-$, $r \in \R(\PP)$.
 \end{coro}
 \begin{proof}
  Letting $\xi(k):=k! k^{-k}$ for $k \in \mathbb N$ and $\xi(0)=1$,
  we set
  \[
   \kappa_r := \prod_{i=1}^d \xi ( (c^r_\inn)_i) > 0 \;.
  \]
  Comparing \eref{e:jrates} and \eref{e:rates}
  we have $\Lambda_{r,v}(x) \le \lambda_r(x)$ for any
  $x \in (v^{-1}\mathbb N_0)^{d_\PP}$, $v \ge 1$.
  Further, the
  ratio $\Lambda_{r_-,v}(x)/\lambda_{r_-}(x)$ is non-decreasing
  in each $v x_i$ and equals $\kappa_{r_-}$ when  $v x = c_\inn^{r_-}$.
  Since by \eqref{e:plus-wk} with $M = \log c_*$ and $\mu = 1$ we have $v x_i \ge (c_\inn^{r_-})_i$ for any
  $v > v^*(\rho_0)$,
  $r_- \in \R(\PP)_-
  \cup \R(\PP )_0
  $, $w \in (\F^*)^{\epsilon,\delta}$, $i \le d$\,,
  the claim follows by setting
  $\Clambda := \max_{r} \{k_r\} /
  \min_{r_-} \{\kappa_{r_-}  k_{r_-} \}$.
 \end{proof}

 We now turn to the second part of this section, dedicated to
 upper bounds on the Lyapunov terms {$\qq{r}{x}$}.
 We will first estimate these terms by approximating, in \lref{l:taylor}, $\uu(x+v^{-1} c^r)/\uu(x)$ as an exponential, whose argument is then bounded from above in Lemmas~\ref{l:epscr}--\ref{l:pointingoutside}  based on geometric considerations.
 Note that $\uu(x+v^{-1} c^r)-\uu(x)$ is a sum of terms of the
 form
 \def\xxx{(x_i +v^{-1} c_i^r)}
 $$
 T_i=\jpp{\weight_i} \pc{\xxx \log \xxx-\xxx -x_i \log x_i + x_i}~.
 $$
 We bound $T_i$ according to 3 cases:
 \begin{myenum}
  \item If $i\in\PP$ then we know $x_i\geq v^{-1}$ and we bound $v T_i$ by
  $\crwi \log x_i + \jp{g_r}(v,x)$. We show below that $|g(\cdot)|$ is
  globally bounded.
  \item If $i \in \supp \{c^r_\outt \} \cap \PP^c$ then
  $v T_i=\crwi \log (v^{-1} c^r_i)+\jp{g_r}(v,x)$.
  \item In the remaining case \jp{$c^r_i=0$} and hence $T_i=0$.
 \end{myenum}

 In view of the above, we define
 as in \cite{agazzi17},
 \begin{equs}
  \pc{\nabla_{r,v} \uu}_i :=
  \left\{\begin{array}{c@{,}l}
  \log x_i	         & 	\qquad i \in \PP  \,	\\
  \log (v^{-1} c^r_i)  & \qquad i \in \supp \{c^r_\outt \} \cap \PP^c \,
  \label{dfn:gradU-P} \\
  0                    &  \qquad \text{ otherwise }
  \end{array}\right..
 \end{equs}

 \jpp{\noindent We now extend \cite[Lemmas 3.4, 3.6]{agazzi17} to the class of Lyapunov functions $\uu_\weight$ at hand.}

 \begin{lm}
  \jpp{For any $\weight \in \Rr_{>0}^d$ as in \dref{d:Pendo1} }there exist finite $v_0$ and $\zeta_*$ such that for any $\PP \subseteq \S$,
  $r \in \R(\PP)$, $v \geq v_0$ and
  $x \in \Ov$
  with $\supp \{x\} = \PP$ and $x + v^{-1} c^r \in \Ov$, one has
  \begin{align}\label{d:zetar}
   \frac{\uu^v(x + v^{-1} c^r)}{\uu^v(x)}
     & = \exp \Big\{\frac{ \h{v}{x}
   + \zz(v,x)}{\uu(x)}\Big\}\,,
  \end{align}
  for $\h{v}{x} := \langle \nabla_{r,v}\uu(x), \crw \rangle$ and a function $\zz(v,x)$ with $|\zz(v,x)| \le \zeta_*$.
  \label{l:taylor}
 \end{lm}

 \begin{proof} Since the number of possible $\PP$ and $r$ is finite, it
  suffices to prove the claim for fixed $\PP$ and $r \in \R(\PP)$.
  To this end,
  we conveniently write the \abbr{rhs} of \eref{d:zetar} as
  \begin{equ}
   \frac{\uu^v(x + v^{-1} c^r)}{\uu^v(x)} = \pc{1 + \frac{f}{v}}^{v} =
   \exp \pq{\frac{f \uu(x) - v \uu(x) R(f/v)}{\uu(x)}}~,\end{equ}
   where $f := v[ \uu(x+ v^{-1} c^r) - \uu(x)]/\uu(x)$ and $R(y):=y-\log(1+y)$.
   We then define
   \begin{equ}\label{e:exponential}
    \zz(v,x) := g_r(v,x) - v \uu(x) R(f/v) \end{equ} where $
    g_r(v,x):=f \uu(x) - \h{v}{x}\,,
    $ and proceed to bound the terms on the \abbr{RHS} of \eref{e:exponential} separately. To this end, with
    $\psi(b;c):= (b+c)\log(1+c/b)$, we have
    when $\supp\{x\} \subseteq \PP$ that
    \begin{equ}
     g_r(v,x)
     = \sum_{i \in \PP} \jpp{a_i} \psi(v x_i;c^r_i)
     - \langle \crw, \onevec \rangle ~.
    \end{equ}
    Since $\psi(b;c)$ decreases in $b \ge \max(1,-c)$,
    it is easy to verify that $|g_r(v,x)|$ is uniformly bounded
    over $vx$ as in the statement of the lemma.

    With
    $R(y) \le 2 y^2$ when $y \ge -1/2$, we
    globally bound the remainder
    in \eref{e:exponential} upon showing that for some $v_0$ finite and
    all $v \ge v_0$
    \begin{equ}\label{e:f/v}
     v \uu(x) \Big(\frac{2 f}{v}\Big)^2 \le \frac{8 \h{v}{x}^2}{v \uu(x)}
     + \frac{8 g_r(v,x)^2}{v \uu(x)}~,
    \end{equ}
    is uniformly bounded above by $v_0$.
    Since $\uu(x)\ge 1$, the right-most term
    is obviously $\OO(1/v)$. Further, to
    globally bound the first term on the \abbr{rhs} of \eqref{e:f/v},
    it suffices to control
    \[
     \sup_{v y \ge 1} \Big\{ \frac{|\log y|^2}{v [y(\log y -1) + 2]}~\Big\}.
    \]
    For $y \in [v^{-1},v]$ this {quantity} is at most
    $(\log v)^2/v \to 0$
    as $v \to \infty$, whereas for $y \ge v \ge e^2$ it is at
    most $2\log y/(vy) \le 2 \log v/v^2 \to 0$ as $v \to \infty$.
    Hence, some finite $v_0$ will bound
    the \abbr{lhs} of \eqref{e:f/v} uniformly in $v \ge v_0$ and $x$ as stated.
    \end{proof}

    From \cref{r:noindicatoratboundary} and \lref{l:taylor}, we see that a comparison of $\langle w,\crw \rangle$ and $\langle w , c^r_\inn - c_\inn^{r_-} \rangle$ is needed for bounding summands of \eref{e:dividesums} from above, as was done in \eref{e:lyapunovdominance}.

    \begin{lm}\label{l:epscr}
     For any $\PP \subseteq \S$, \jpp{$\weight \in \Rr_{> 0}^d$ as in \dref{d:Pendo1}\,,} $(j,\ii)\in \Ip^*$, there exist
     $\Cepsilon \in (0,1)$ and $\Cc {\in [1,\infty)}$, such that if $c_* \delta_j \leq \Cepsilon \epsilon_j$, $\epsilon_j < 1/2$, then for all $w \in \WWedd$:
     \begin{myenum}
      \item $\dtp{w}{\crw} \leq c_* \quad\qquad\qquad\qquad\,$ for all
      $\quad r \in \R(\PP)\,$,
      \item $\dtp{w}{\crw} <  c_* \delta_j\; \qquad\qquad\qquad$ for all $\quad r \in \R(\PP)_- \cup \R(\PP)_0\,$, $\supp c^r \subseteq \PP\,$,
      \item $\dtp{w}{c_\inn^r- c_\inn^{r_-}} < c_* \delta_j\qquad\qquad$ for all $\quad r \in \R(\PP)_- \cup \R (\PP)_0$\,, $r_- \in \R(\PP)_-$\,,
      \item $\dtp{w}{c_\inn^{r}-c_\inn^{r_-}} <- \Cepsilon \epsilon_j\;\qquad\;\;$ for all $\quad r \in \R(\PP)_+$\,, $r_- \in \R(\PP)_-$\,,
      \item $\dtp{w}{\jpp{c^{r_-,\weight}}} <-  \Cepsilon \epsilon_j\;\qquad\qquad\,$ for all $\quad r_- \in \R(\PP)_-$\,, $\supp c^{r-} \subseteq \PP\,$,
      \item  $\dtp{w}{\crw} \leq -\Cc \dtp{w}{c_\inn^r-c_\inn^{r_-}}\;$ for $\qquad \; r, r_-$ as in (d)\,.
     \end{myenum}
    \end{lm}

    \begin{proof} Fixing $\PP \subset \S$, \jpp{$\weight \in \mathbb R_{>0}^d$ as in \dref{d:Pendo1}}, $(j,\ii)\in \Ip^*$,
     we abbreviate throughout $\F^*=\WW(\PP)_{j,\ii}^*$, $\epsilon =\epsilon _j$,
     $\delta =\delta _j$ and $z=c_\inn^{r}-c_\inn^{r_-}$.
     Note that part (a) is merely the trivial inequality
     $$
     \dtp{w}{\crw} \leq \|\crw\|_2 \leq
     c_* \,.
     $$
     For any $w \in \Xed$ we have \begin{equ}
     \|w-w'\|_2 < \delta\quad \text{for some} \quad w' \in (\F^*)^{\epsilon,0} \subset \F^*~,\label{e:deltaargument}\end{equ}
     and part (b) similarly follows, as
     $\dtp{w'}{\crw} \le 0$ whenever
     $r \in \R(\PP)_- \cup \R(\PP)_0$
     and $\supp c^r \subseteq \PP$
     (see \dref{d:dissipative} and
     \rref{r:dec-Fstar}). Combining \eref{e:deltaargument}
     with the inequality
     $\dtp{w'}{z} \le 0$
     for any $w' \in \F^*$, $r \in \R(\PP)$ and
     $r_- \in \R(\PP)_-$ (from \rref{r:duality}), proves (c). Further,
     $\dtp{w'}{z}$ is negative
     whenever $r \in \R(\PP)_+$ and $r_- \in \R(\PP)_-$,
     with $\dtp{w'}{\jpp{c^{r_-,\weight}}}$ negative if in addition
     $\supp c^{r_-} \subseteq \PP$.
     Thus, considering $c_* \delta \le \Cepsilon \epsilon$ and the finite set
     \begin{equs}
      \Xi :=
      \big\{ z :
      r \in \R(\PP)_+, r_- \in \R(\PP)_-
      \big\}
      \bigcup \big\{ \jpp{c^{r_-,\weight}} : r_- \in \R(\PP)_-, \supp c^{r_-} \subseteq \PP
      \big\},
     \end{equs}
     parts (d) and (e) follow
     upon finding $K_2 \in (0,1)$ such that for
     $\epsilon>0$ and $\xi \in \Xi$,
     \begin{equ}\label{e:geo-bd}
      \inf_{w \in (\F^*)^{\epsilon,0}} \{ \, |\dtp{w}{\xi}| \,\}
      \ge 2 \Cepsilon \epsilon \,.
     \end{equ}
     Since $\Xi$ is finite, it suffices to establish
     \eqref{e:geo-bd} for each fixed $\xi \in \Xi$.
     Further scaling \abbr{wlog} such $\xi$ to be a unit vector, we can write
     \[
      |\dtp{w}{\xi}| =  \inf_{x \in H^\perp(\xi)}  \|w - x\|_2 := g(w) \,.
     \]
     Note that
     the hyperplane $H^\perp(\xi)$ perpendicular to $\xi \in \Xi$
     is disjoint of $\F^*$, or else we would have had
     $c_\inn^r \in \F$ for some $r \in \R(\PP)_+$ or
     $\jpp{c^{r_-,\weight}} \in \F$ for some $r_- \in \R(\PP)_-$ for which
     $\supp \{c^{r_-}\} \subseteq \PP$,
     in contradiction with our partition of $\R(\PP)$
     (see also \dref{d:dissipative}).
     Using this, we first consider the case $j=0$, where the cap
     $\F^*$ has a positive $(d_\PP-1)$-dimensional surface area.
     Then,
     as illustrated in \fref{f:geometry1}, we have
     for any $\epsilon<1/2$ and $w \in (\F^*)^{\epsilon,0}$, that
     \begin{equ}
      g(w) \geq
      \inf_{x \in \partial \F^*} \{
      \sin (\angle(w,x)) \} \ge
      \sin (\epsilon)> \epsilon/2 \,.
      \label{e:Hj1}\end{equ}
      \begin{figure}[t!]
       \centering
       \begin{subfigure}[b]{0.41\textwidth}
        \centering
        \def\svgwidth{.95\textwidth}
        \begingroup%
  \makeatletter%
  \providecommand\color[2][]{%
    \errmessage{(Inkscape) Color is used for the text in Inkscape, but the package 'color.sty' is not loaded}%
    \renewcommand\color[2][]{}%
  }%
  \providecommand\transparent[1]{%
    \errmessage{(Inkscape) Transparency is used (non-zero) for the text in Inkscape, but the package 'transparent.sty' is not loaded}%
    \renewcommand\transparent[1]{}%
  }%
  \providecommand\rotatebox[2]{#2}%
  \ifx\svgwidth\undefined%
    \setlength{\unitlength}{316.43317871bp}%
    \ifx\svgscale\undefined%
      \relax%
    \else%
      \setlength{\unitlength}{\unitlength * \real{\svgscale}}%
    \fi%
  \else%
    \setlength{\unitlength}{\svgwidth}%
  \fi%
  \global\let\svgwidth\undefined%
  \global\let\svgscale\undefined%
  \makeatother%
  \begin{picture}(1,1.07093073)%
    \put(0,0){\includegraphics[width=\unitlength,page=1]{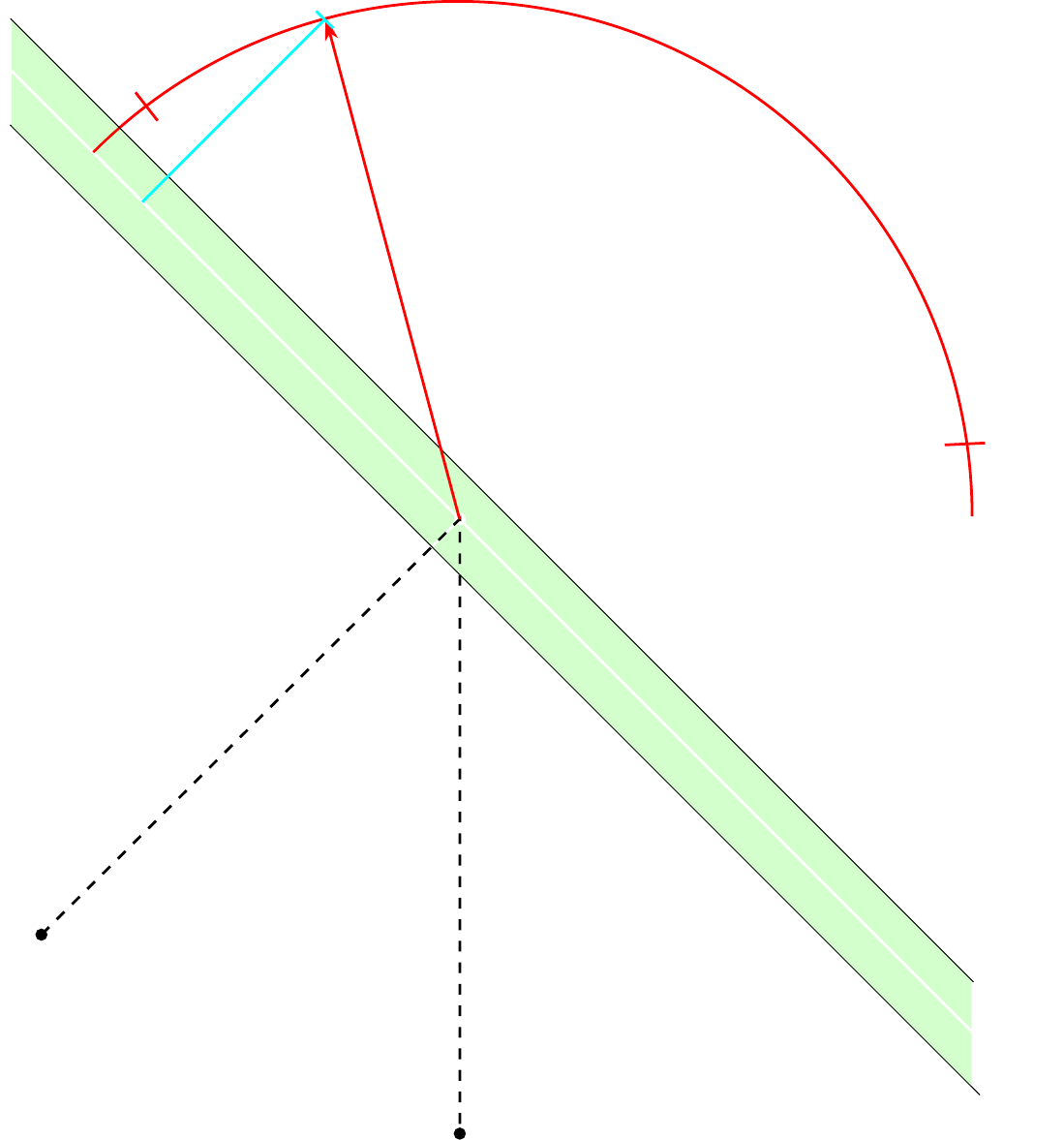}}%
    \put(0.77197396,0.94712039){\color[rgb]{0,0,0}\makebox(0,0)[lb]{\smash{$\color{red} \F^* \color{black}$}}}%
    \put(-0.00104929,0.66939714){\color[rgb]{0,0,0}\makebox(0,0)[lb]{\smash{$\color{blue} H^\perp(\xi)\color{black}$}}}%
    \put(0.25226603,0.34568936){\color[rgb]{0,0,0}\makebox(0,0)[lb]{\smash{$\xi$}}}%
    \put(0.37357376,0.86284342){\color[rgb]{0,0,0}\makebox(0,0)[lb]{\smash{$\color{red}w\color{black}$}}}%
    \put(0.0581736,0.13946617){\color[rgb]{0,0,0}\makebox(0,0)[lb]{\smash{$c_\inn^r$}}}%
    \put(0.44508149,0.01241224){\color[rgb]{0,0,0}\makebox(0,0)[lb]{\smash{}}}%
    \put(0.93553018,0.59811957){\color[rgb]{0,0,0}\makebox(0,0)[lb]{\smash{$\epsilon$}}}%
    \put(0,0){\includegraphics[width=\unitlength,page=2]{geometry1.pdf}}%
    \put(0.4634425,0.61258454){\color[rgb]{0,0,0}\makebox(0,0)[lb]{\smash{$\F$}}}%
    \put(0.66304457,0.44223245){\color[rgb]{0,0,0}\makebox(0,0)[lb]{\smash{$\color{green} (H^\perp(\xi))^{\epsilon/2}\color{black}$}}}%
    \put(0,0){\includegraphics[width=\unitlength,page=3]{geometry1.pdf}}%
  \end{picture}%
\endgroup%
        \caption{}
       \end{subfigure}
       $\qquad$
       \begin{subfigure}[b]{0.47\textwidth}
        \centering
        \def\svgwidth{.95\textwidth}
        \begingroup%
  \makeatletter%
  \providecommand\color[2][]{%
    \errmessage{(Inkscape) Color is used for the text in Inkscape, but the package 'color.sty' is not loaded}%
    \renewcommand\color[2][]{}%
  }%
  \providecommand\transparent[1]{%
    \errmessage{(Inkscape) Transparency is used (non-zero) for the text in Inkscape, but the package 'transparent.sty' is not loaded}%
    \renewcommand\transparent[1]{}%
  }%
  \providecommand\rotatebox[2]{#2}%
  \ifx\svgwidth\undefined%
    \setlength{\unitlength}{353.498368bp}%
    \ifx\svgscale\undefined%
      \relax%
    \else%
      \setlength{\unitlength}{\unitlength * \real{\svgscale}}%
    \fi%
  \else%
    \setlength{\unitlength}{\svgwidth}%
  \fi%
  \global\let\svgwidth\undefined%
  \global\let\svgscale\undefined%
  \makeatother%
  \begin{picture}(1,0.95864096)%
    \put(0,0){\includegraphics[width=\unitlength,page=1]{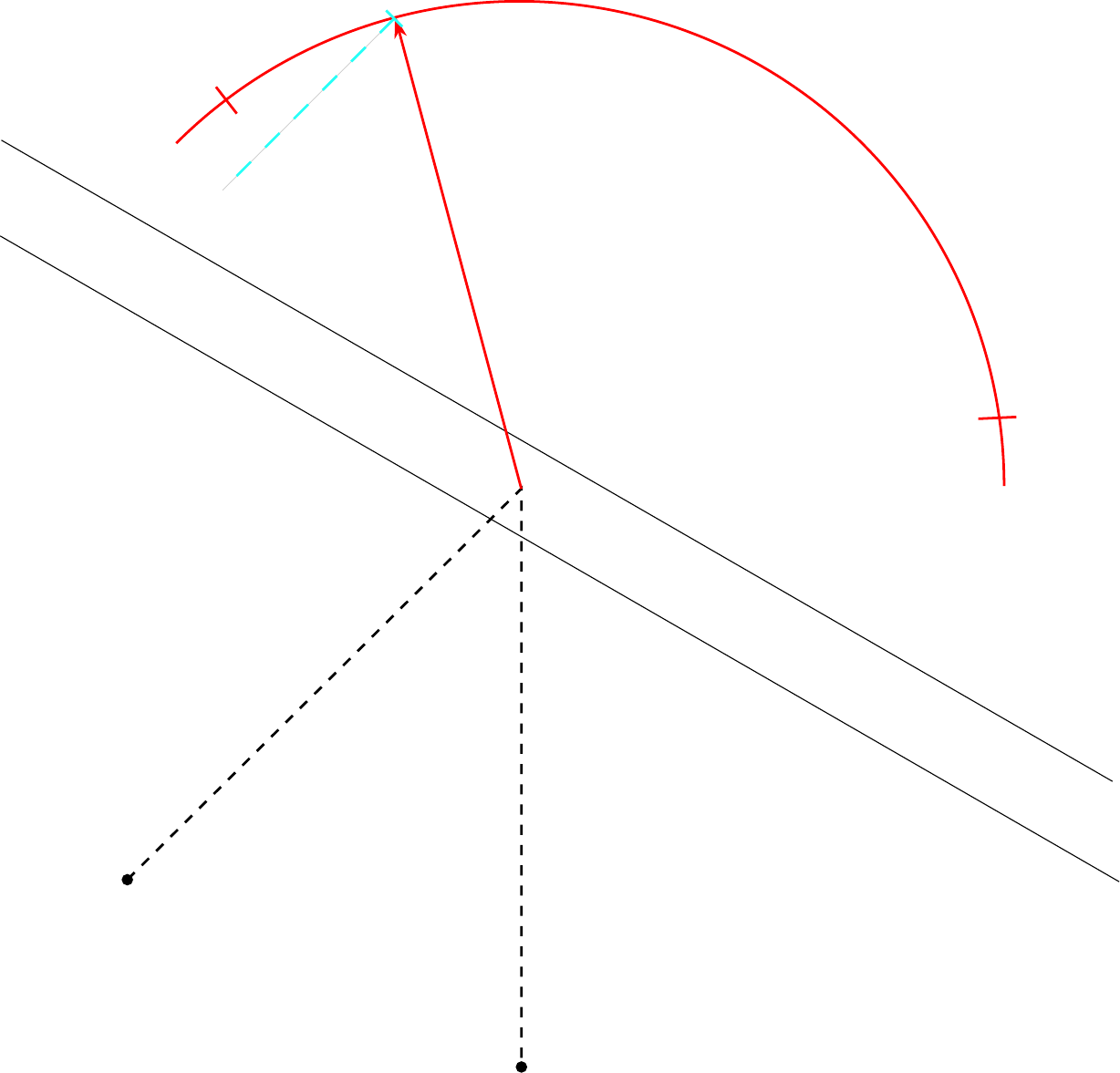}}%
    \put(0.76993495,0.84781244){\color[rgb]{0,0,0}\makebox(0,0)[lb]{\smash{$\color{red} \F^* \color{black}$}}}%
    \put(0.35402279,0.24518727){\color[rgb]{0,0,0}\makebox(0,0)[lb]{\smash{$\xi$}}}%
    \put(0.13097834,0.12484279){\color[rgb]{0,0,0}\makebox(0,0)[lb]{\smash{$c_\inn^r$}}}%
    \put(0.47731797,0.01111079){\color[rgb]{0,0,0}\makebox(0,0)[lb]{\smash{}}}%
    \put(0,0){\includegraphics[width=\unitlength,page=2]{geometry9.pdf}}%
    \put(0.68466034,0.45247598){\color[rgb]{0,0,0}\makebox(0,0)[lb]{\smash{$\color{green} (H^\perp(\xi))^{\epsilon/2}\color{black}$}}}%
    \put(0.06718421,0.54064341){\color[rgb]{0,0,0}\makebox(0,0)[lb]{\smash{$\color{blue} H^\perp(\xi)\color{black}$}}}%
    \put(0.47660834,0.56549911){\color[rgb]{0,0,0}\makebox(0,0)[lb]{\smash{$\F$}}}%
    \put(0.43622287,0.73813433){\color[rgb]{0,0,0}\makebox(0,0)[lb]{\smash{$\color{red}w\color{black}$}}}%
    \put(0,0){\includegraphics[width=\unitlength,page=3]{geometry9.pdf}}%
  \end{picture}%
\endgroup%
        \caption{}
       \end{subfigure}
       \caption{Representation of the geometric construction used in the proof of \eqref{e:geo-bd} for $j = 0$. The faces of $\WW(\PP)$ adjacent to $\F$ are denoted by dashed black lines, the dual $\F^*$ is drawn in red together with its $\epsilon$-boundaries and an element $w \in \F^*$ drawn as a solid red arrow.
        For a certain $\xi$ (solid black vector), we have drawn
        in green the $(\epsilon/2)$-neighborhood of the blue $H^\perp(\xi)$,
        with a light blue line-segment for the distance between
        $w$ and $H^\perp(\xi)$. That distance (a) equals or (b) is greater than
        $\inf_{x \in \partial \F^*} \{ \sin (\angle(w,x)) \}\,$ (denoted as a dashed blue line in (b)).}
       \label{f:geometry1}
      \end{figure}

      \noindent Next, if $1 \le j \le d_\PP-2$, then the unique $(d_\PP-j)$-dimensional hyperplane
      $H_{j+1}(\F^*)$ containing $0$ and all of $\F^*$, has a positive
      dihedral angle
      \begin{equs}
       \gamma := \angle \pc{H^{\perp}(\xi), H_{j+1}(\F^*)} \in (0,\pi/2]
      \end{equs}
      with
      $H^\perp(\xi)
      $, see \fref{f:geometry23} for an illustration.
      \footnote{where the angle between hyperplanes $B,B' \subset \Rr^{d_\PP}$ passing through the origin is
       \jpp{$\angle(B,B') := \arccos(\max\{\min\{|\langle b,b' \rangle|:b \in B, \|b\|_2 = 1\}:b' \in B', \|b'\|_2 = 1\})$}.}\\
      As depicted in \fref{f:geometry22}, we then bound
      $g(w)$ by further restricting $x$ to be on $ H_{j+1}(\xi)\,$. At the cost of a
      constant factor $\sin \gamma\,$, this reduces the problem to that of
      $j = 0$ treated above. Indeed, combining
      \eref{e:Hj1} with $\sin (\gamma) \geq 2\gamma/\pi$
      we obtain for any $\epsilon < 1/2$ and $w \in (\F^*)^{\epsilon,0}\,$ that
      \begin{equs}
       g(w) & = \sin (\gamma) \inf_{x \in H^\perp(\xi) \cap  H_{j+1}(\F^*)} \|w - x\|_2 \\ & \ge \frac{2 \gamma}{\pi}
       \inf_{x \in \partial \F^*} \{ \sin (\angle(w,x)) \} \ge
       \frac{\gamma \epsilon}{\pi} ~. \label{e:Hj2}
      \end{equs}

      \begin{figure}[t]
       \centering
       \begin{subfigure}[b]{0.39\textwidth}
        \centering
        \def\svgwidth{.95\textwidth}
        \begingroup%
  \makeatletter%
  \providecommand\color[2][]{%
    \errmessage{(Inkscape) Color is used for the text in Inkscape, but the package 'color.sty' is not loaded}%
    \renewcommand\color[2][]{}%
  }%
  \providecommand\transparent[1]{%
    \errmessage{(Inkscape) Transparency is used (non-zero) for the text in Inkscape, but the package 'transparent.sty' is not loaded}%
    \renewcommand\transparent[1]{}%
  }%
  \providecommand\rotatebox[2]{#2}%
  \ifx\svgwidth\undefined%
    \setlength{\unitlength}{379.94526978bp}%
    \ifx\svgscale\undefined%
      \relax%
    \else%
      \setlength{\unitlength}{\unitlength * \real{\svgscale}}%
    \fi%
  \else%
    \setlength{\unitlength}{\svgwidth}%
  \fi%
  \global\let\svgwidth\undefined%
  \global\let\svgscale\undefined%
  \makeatother%
  \begin{picture}(1,0.87637725)%
    \put(0,0){\includegraphics[width=\unitlength,page=1]{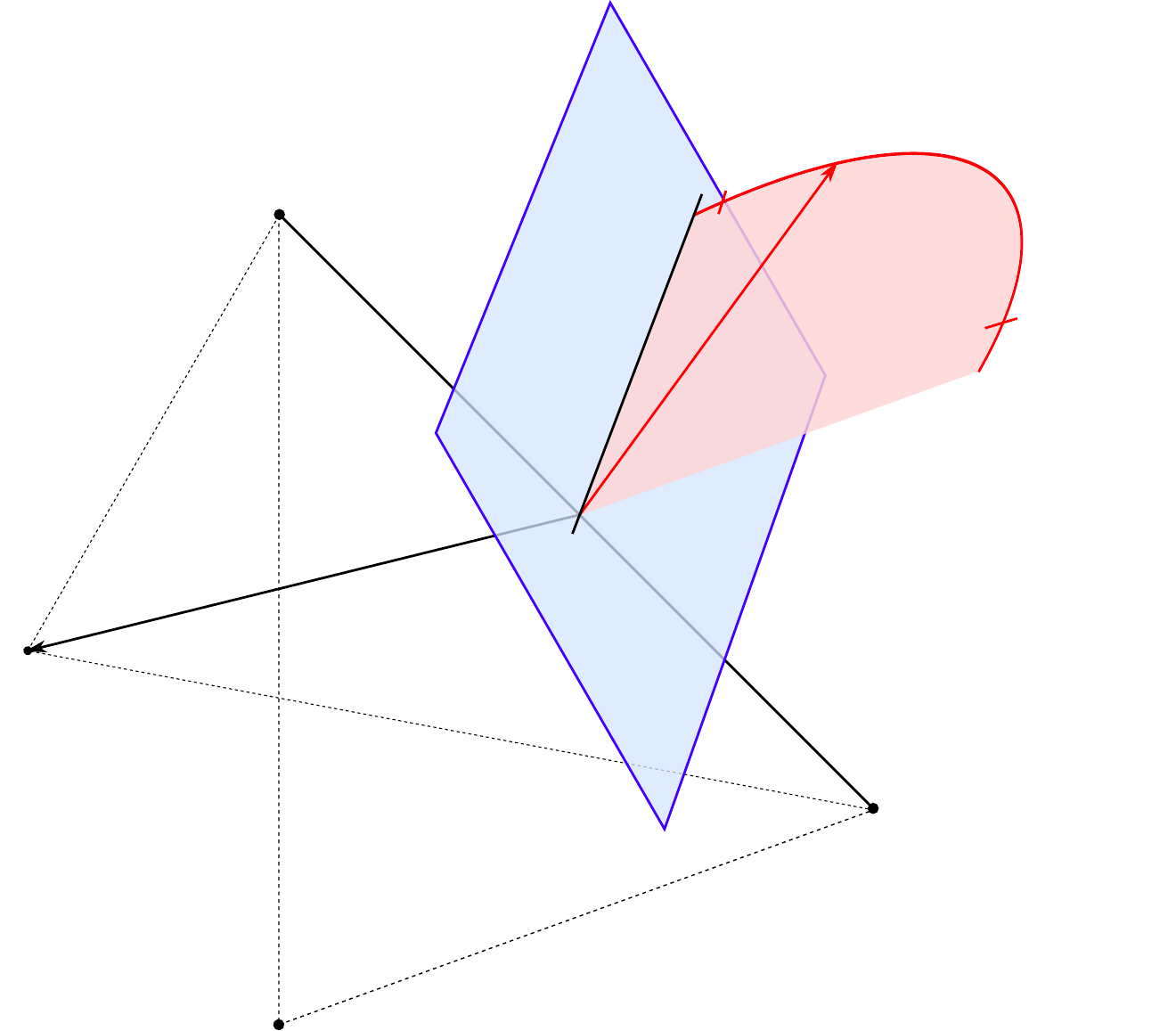}}%
    \put(0.82603888,0.75337306){\color[rgb]{0,0,0}\makebox(0,0)[lb]{\smash{$\color{red} \F^* \color{black}$}}}%
    \put(0.67130642,0.63923111){\color[rgb]{0,0,0}\makebox(0,0)[lb]{\smash{$\color{red}w\color{black}$}}}%
    \put(0.26864775,0.33028084){\color[rgb]{0,0,0}\makebox(0,0)[lb]{\smash{$\xi$}}}%
    \put(0.6701843,0.27376488){\color[rgb]{0,0,0}\makebox(0,0)[lb]{\smash{$\F$}}}%
    \put(-0.0048938,0.25394172){\color[rgb]{0,0,0}\makebox(0,0)[lb]{\smash{$c_\inn^r$}}}%
    \put(0,0){\includegraphics[width=\unitlength,page=2]{geometry2.pdf}}%
    \put(0.57156221,0.80264895){\color[rgb]{0,0,0}\makebox(0,0)[lb]{\smash{$\color{blue} H^\perp(\xi)\color{black}$}}}%
    \put(0.86101385,0.56727671){\color[rgb]{0,0,0}\makebox(0,0)[lb]{\smash{$\epsilon$}}}%
    \put(0,0){\includegraphics[width=\unitlength,page=3]{geometry2.pdf}}%
    \put(0.48959555,0.65901927){\color[rgb]{0,0,0}\makebox(0,0)[lb]{\smash{$\color{blue} \gamma \color{black}$}}}%
  \end{picture}%
\endgroup%
        \caption{}
        \label{f:geometry23}
       \end{subfigure}
       $\qquad$
       \begin{subfigure}[b]{0.53\textwidth}
        \centering
        \def\svgwidth{.95\textwidth}
        \begingroup%
  \makeatletter%
  \providecommand\color[2][]{%
    \errmessage{(Inkscape) Color is used for the text in Inkscape, but the package 'color.sty' is not loaded}%
    \renewcommand\color[2][]{}%
  }%
  \providecommand\transparent[1]{%
    \errmessage{(Inkscape) Transparency is used (non-zero) for the text in Inkscape, but the package 'transparent.sty' is not loaded}%
    \renewcommand\transparent[1]{}%
  }%
  \providecommand\rotatebox[2]{#2}%
  \ifx\svgwidth\undefined%
    \setlength{\unitlength}{403.19735508bp}%
    \ifx\svgscale\undefined%
      \relax%
    \else%
      \setlength{\unitlength}{\unitlength * \real{\svgscale}}%
    \fi%
  \else%
    \setlength{\unitlength}{\svgwidth}%
  \fi%
  \global\let\svgwidth\undefined%
  \global\let\svgscale\undefined%
  \makeatother%
  \begin{picture}(1,0.70284025)%
    \put(0,0){\includegraphics[width=\unitlength,page=1]{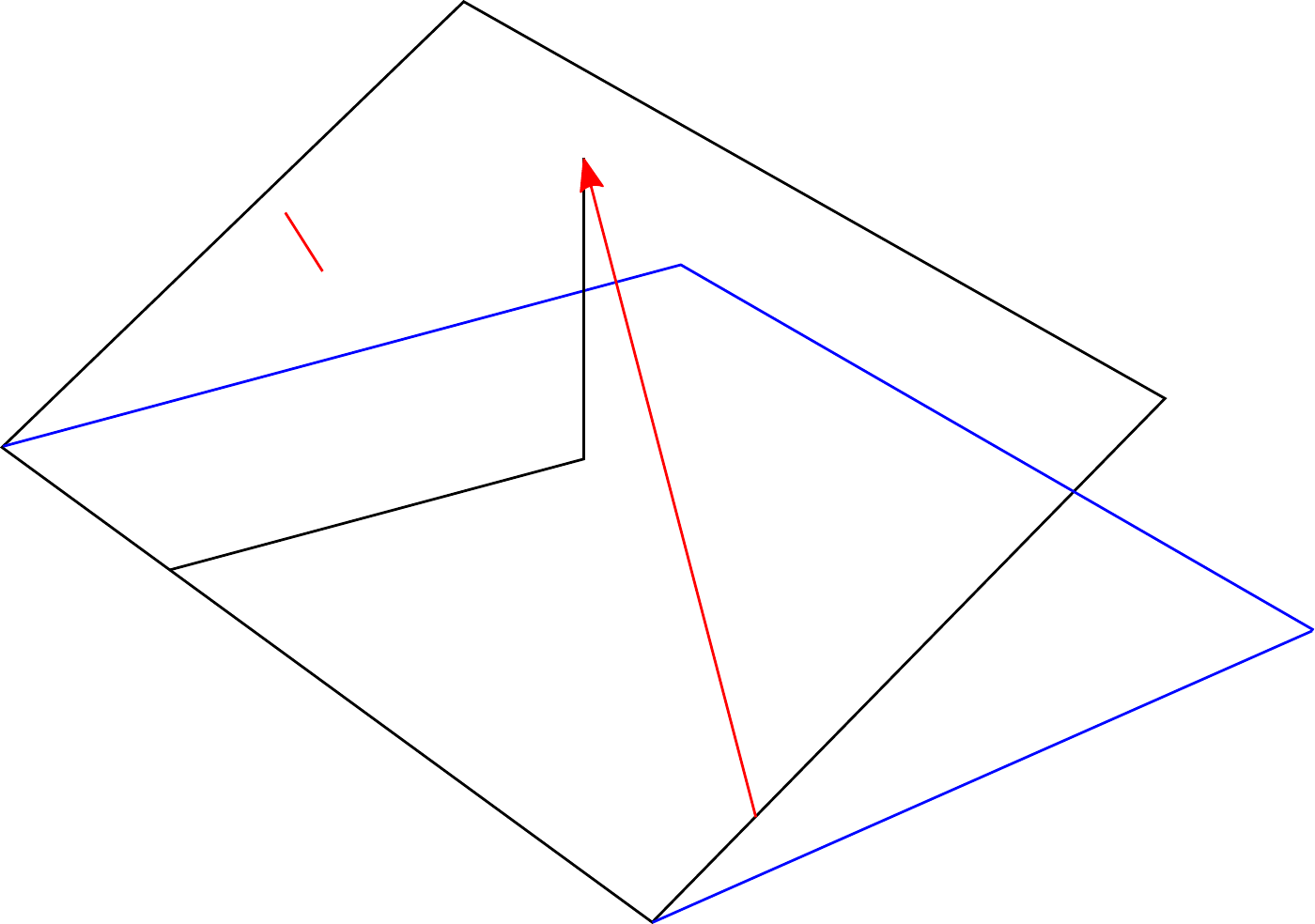}}%
    \put(0.05108292,0.45156609){\color[rgb]{0,0,0}\makebox(0,0)[lb]{\smash{$\epsilon$}}}%
    \put(0.43632646,0.30877153){\color[rgb]{0,0,0}\makebox(0,0)[lb]{\smash{$x^*$}}}%
    \put(0,0){\includegraphics[width=\unitlength,page=2]{geometry3.pdf}}%
    \put(0.70058002,0.01841994){\color[rgb]{0,0,0}\makebox(0,0)[lb]{\smash{$\color{blue} H^\perp(\xi)\color{black}$}}}%
    \put(0.33490958,0.58073864){\color[rgb]{0,0,0}\makebox(0,0)[lb]{\smash{$\color{red} \F^* \color{black}$}}}%
    \put(0,0){\includegraphics[width=\unitlength,page=3]{geometry3.pdf}}%
    \put(0.52080445,0.31382894){\color[rgb]{0,0,0}\makebox(0,0)[lb]{\smash{$\color{red}w\color{black}$}}}%
    \put(0.59585633,0.57410763){\color[rgb]{0,0,0}\makebox(0,0)[lb]{\smash{$ H_{j+1}(\F^*)$}}}%
    \put(0,0){\includegraphics[width=\unitlength,page=4]{geometry3.pdf}}%
    \put(0.28384268,0.3400746){\color[rgb]{0,0,0}\makebox(0,0)[lb]{\smash{$\color{blue} \gamma \color{black}$}}}%
  \end{picture}%
\endgroup%
        \caption{}
        \label{f:geometry22}
       \end{subfigure}
       \caption{Representation of the geometric construction used in the proof of \lref{l:epscr} for $j \neq 0$, with the same notations as \fref{f:geometry1}.
        The part of $H_{j+1}(\F^*)$ containing $\F^*$ is drawn in red and
        the dihedral angle $\gamma$ in dotted blue.
        Fig.~(b) is an inset
        of Fig.~(a), where the light blue distance of $w$ from
        $H^\perp(\xi) \cap H_{j+1}(\F^*)$, equals
        $\sin \gamma$ times the distance between $w$ and its
        projection $x^*$ onto $H^\perp(\xi)$.}
       \label{f:geometry2}
      \end{figure}
      \noindent Last, for $j=d_\PP-1$, the same reasoning can be applied by setting $(\F^*)^{\epsilon,0} = \F^*$. As $H_{j+1}(\F^*) \cap H^\perp(\xi) = 0$, we obtain $g(w) = \sin \gamma \geq 2 \gamma/\pi$, verifying \eqref{e:geo-bd} for $\rho_0 = \rho_0(K_2,\gamma)$ large enough.
      In conclusion, we have \eqref{e:geo-bd} with
      $\Cepsilon := 1/4$ when $j=0$ and
      $\Cepsilon := \gamma/(2\pi)$ when $j>0$,
      thanks to \eref{e:Hj1} and \eref{e:Hj2}, respectively.

      Part (f) asserts that
      $\dtp{w}{\crw + \Cc z}\leq 0$ on
      $(\F^*)^{\epsilon,\delta}$. Since
      $c_* \delta \le \Cepsilon \epsilon$, it suffices to show that
      \begin{equ}\label{e:partial}
       \dtp{w}{\crw + \Cc' z} \le 0 \qquad \forall\, w \in \F^*\,.
      \end{equ}
      Indeed, as $w \in (\F^*)^{\epsilon,\delta}$ is within $\delta$
      of some $w' \in \F^*$  for which
      $\dtp{w'}{z} \le - 2 \Cepsilon \epsilon$ (c.f. \eqref{e:geo-bd}),
      the argument leading to (d) then gives for $\Cc:=1+2\Cc'$
      \begin{equs}
       \dtp{w}{\crw + \Cc z} &\le \dtp{w'}{\crw + \Cc z} + c_* \delta (1+\Cc) \\
       &\le (\Cc - \Cc') \dtp{w'}{z} + \Cepsilon \epsilon (1+\Cc) \\
       &\le 2 \Cepsilon \epsilon (\Cc' - \Cc) + \Cepsilon \epsilon (1+\Cc) = 0 \,.
      \end{equs}
      Finally, to prove \eqref{e:partial} it suffices to check that
      the linear functional on its \abbr{lhs} is non-positive
      in the finite set of
      extreme directions $\NN(\WWji)$.
      This results from the fact that $r \in \R(\PP)_+$
      and \dref{d:Pendo1} of strongly endotactic \abbr{crn}: indeed
      either $z_n:= \dtp{n}{z} < 0$ (see \rref{r:duality}), or, if $c_\inn^r$ is on a $(d_\PP-1)$-dimensional
      face of $\WW(\PP)$ perpendicular to $n$, then $\dtp{n}{z} = 0$. In the latter case, however, by \dref{d:Pendo1} and our partition of $\R(\PP)$ we have
      $\dtp{n}{\crw} \le 0$. Thus, \eref{e:partial} holds
      for any $\Cc' \le \max \{ c_*/|z_n| : n \in \NN(\WWji), z_n \ne 0 \}$\,.
      \end{proof}

      Utilizing again \lref{l:key2}, we next bound
      the contribution of possibly dominant
      reactions with $\supp c_\outt^r \not \subset \PP$.
      \begin{lm}\label{l:pointingoutside}
       Choose $\Cv<\infty$. Then, \jpp{for any $\weight \in \Rr_{>0}^d$ as in \dref{d:Pendo1}} there exists $\rho_0$ such that under our standard conditions,
       for any $\PP \subseteq \S$, face $\F \in \WW(\PP)$ and
       $r \in \R(\PP)_-$ with $\supp c_\outt^r \not \subset \PP$, we have
       \begin{equ}
        \h{v}{\theta^w} < -\Cv \qquad \text{for all\;\; }\,
        w \in (\F^*)^{\epsilon,\delta} \,.
       \end{equ}
      \end{lm}

      \begin{proof} Fixing $r \in \R(\PP)_-$ with $\supp \{c_\outt^r\} \not \subset \PP$,
       set $\kappa := \sum_{i \not\in \PP} \crwi \log c_i^r$ finite and
       $\beta := \jpp{\min_i \{\weight_i\} > 0}$. Then,
\jpp{since $\sum_{i \not\in \PP} \crwi \jpp{\ge \beta}$,}
       for any $w = w_\PP \in (\F^*)^{\epsilon, \delta}$,
       \begin{equ}\label{e:cv}
        \kappa - \h{v}{x}
        \geq \jpp{\beta} \log v - \dtp{w}{\crw} \log \theta \,.
       \end{equ}
       If $c_i^r =0$ for all $i \in \PP$ then $\dtp{w}{\crw}=0$ and we are done.
       Otherwise, setting $M=\jpp{(\kappa+\Cv)/\beta}$, note that if $c^r_i<0$
       then $(c^r_\inn)_i \ge 1$ and \eqref{e:plus-wk} applies, while
       \eqref{e:minus-wk} applies for all $i \in \PP$.
       \jpp{Setting $\mu := \jpp{\beta}/\|\crw\|_1 \in (0,1]$\,,} we thus get from \eqref{e:cv}
       \[
        \kappa - \h{v}{x} \ge  \sum_{i \in \PP} |\crwi| \pc{ \mu \log v
         - \sign (\crwi) \jp{w_i} \log \theta}
        > M \sum_{i \in \PP} |\crwi| \ge \kappa + \Cv~.
       \]
       We can find $\rho_0(\Cv)$ large enough to apply for all
       faces, reactions and species.
      \end{proof}

      We next
      combine \cref{r:noindicatoratboundary},
      \lref{l:taylor} and \lref{l:epscr} to
      control the contribution
      to \eref{e:dividesums} by the
      reactions in $\R(\PP)_+ \cup \R(\PP)_0$.
      \begin{coro}\label{r:monomialdomination}
       \jpp{Fix $\weight \in \Rr_{>0}^d$ as in \ref{d:Pendo1} and} suppose $C_{2j+1} \ge c_* C_{2j}/\Cepsilon$ as in \lref{l:epscr}.
       Then for some $\Cmonomial=\Cmonomial(\Clambda,\Cc,\zeta_*)$ finite,
       any $\PP \subseteq \S$, $(j,\ii) \in \Ip^*$ and
       $r_- \in \R(\PP)_-$, under our standard conditions
       \begin{equ}
        \sum_{r \in \R(\PP)_+ \cup \R(\PP)_0}
        \frac{\Lambda_{r,v}(x)}{\Lambda_{r_-,v}(x)} {\qq{r}{x}}
        < m \Cmonomial e^{2 c_* C_{2j}}
        \,.\label{e:monomial2}
       \end{equ}
      \end{coro}

      \begin{proof} From \eqref{dfn:gradU-P} we see that
       a component outside $\PP$ with $(c^r_\outt)_i > 0$
       contributes
       negatively to $\h{v}{x}$ if $v \ge \max(v_0,c_* e^{2 \zeta_*})$.
       In other words, for $\rho_0$ large enough, restricting $c_\outt^r$ to $\PP$
       can only increase the value of {$\qq{r}{\cdot}$}. Therefore,
       if $\supp \{ c_\outt^r \} \not\subseteq \PP$,
       we consider instead \eref{e:monomial2} for $c^r_\outt$
       restricted to $\PP$. Hence, \abbr{wlog} we set $\PP = \S$ and
       having then
       $\h{v}{x} = \dtp{w}{\crw}\log \theta$, we can write
       $
       \h{v}{x} + \zz(v,x) \le
       2 \max(\dtp{w}{\crw}\log \theta,\zeta_*) =: f_r(x,\zeta_*)$.
       Furthermore, since $\uu(\cdot)$ has compact level sets,
       for large enough $\rho_0$ we have $\uu(x) \ge 4 \Cc$.
       Consequently, since $e^y-1 \le \max(y,0) e^y$ combining \eref{e:Q} and \eref{d:zetar} we bound Lyapunov terms from above by
       \begin{equs}
        \qq{r}{x} &= \uu(x)\pc{\exp\pq{\frac{\h{v}{x} + \zz(v,x))}{\uu(x)}}-1} \\
        &\le  f_r(x,\zeta_*) e^{f_r(x,\zeta_*)/4 \Cc} < 4 \Cc e^{f_r(x,\zeta_*)/2\Cc}~.\label{e:lyapunovbound1}
       \end{equs}
       We proceed to obtain \eref{e:monomial2} by considering reactions in $\R_+$ and $\R_0$ separately.
       For $r \in \R_+$, by \eref{e:lyapunovbound1} and \lref{l:epscr} (f) we have
       \begin{equ}\label{e:bd-Qrv}
        {\qq{r}{x}} < 4 \Cc \exp \pq{\max(-\dtp{w}{z} \log \theta, \zeta_*/\Cc)}
        \,,
       \end{equ}
       where $z:=c^r_\inn-c^{r_-}_\inn$.
       Combining this with \cref{r:noindicatoratboundary} and the inequality $\theta^{\dtp{w}{z}} \le 1$ from \lref{l:epscr} (d), we obtain for $r \in \R_+$
       \begin{equ}\label{e:boundr1}
        \frac{\Lambda_{r,v}(x)}{\Lambda_{r_-,v}(x)} {\qq{r}{x}} \leq \Clambda
        \theta^{\dtp{w}{z}} {\qq{r}{x}} <
        4 \Clambda \Cc e^{\zeta_*} =: \Cmonomial
        \,.
       \end{equ}
       For $r \in \R_0$, by the same reasoning, now using
       that $\delta_j \log \theta = C_{2j}$ and
       \lref{l:epscr} (b) and (c),
       instead of \lref{l:epscr} (f) and (d), respectively,
       we arrive at
       \begin{equs}
        \frac{\Lambda_{r,v}(x)}{\Lambda_{r_-,v}(x)} {\qq{r}{x}} <
        \Cmonomial e^{2 c_* C_{2j}}
        \,.
        \label{e:boundr0}
       \end{equs}
       Finally, since there are at most $m$ reactions, \eqref{e:boundr1} and \eqref{e:boundr0} imply \eqref{e:monomial2}.
      \end{proof}

      \subsubsection{Proof of \tref{t:thm}}
      Fixing $\PP \subseteq \S$\jpp{, $\weight \in \Rr_{>0}^d$ as in \dref{d:Pendo1}} and $(j,\ii) \in \Ip^*$,
      we consider the open set $(\F^*)^{\epsilon_j,\delta_j}$
      for the spherical image $\F^*$ of the face
      $\F = \WW(\PP)_{j,\ii}$ and
      proceed to prove \eref{e:dividesums} under our standard conditions,
      which by \rref{r:proofequivalence} suffices for proving \tref{t:thm}.
      Requiring that
      \begin{equ}\label{e:Cj-prelim}
       \max(2 \zeta_*,c_* C_{2j}) \le \Cepsilon C_{2j+1} \,,
      \end{equ}
      we have already bounded in \cref{r:monomialdomination}
      the \abbr{lhs} of \eref{e:dividesums}, so proceed to handle
      $r_- \in \R(\PP)_-$, starting with the case where
      $\supp \{c^{r_-}_\outt\} \subseteq \PP$.
      For such a reaction we have by
      \lref{l:taylor}, \lref{l:epscr} (e) and our choice of $\epsilon_j$ that
      \begin{equs}
       {\qq{r_-}{x}} &
       \leq \uu(x) \Big\{ e^{(\dtp{w}{\jpp{c^{r_-,\weight}}} \log \theta + \zeta_*)/\uu(x)}-1 \Big\}
       \\ &
       \leq \uu(x) \Big\{ e^{-\Cepsilon C_{2j+1}/(2\uu(x))}-1\Big\} \,.
      \end{equs}
      Further choosing $\rho_0(C_{2j+1})
      $ large enough to have $2 \uu(x) \ge \Cepsilon C_{2j+1}$ whenever
      $\|x\|_1 \ge \rho_0$ and using that $e^{-y} -1 \le -y/2$
      for $y \in [0,1]$, we deduce that for such a reaction
      \begin{equ}
       { \qq{r_-}{x}} \leq - \frac{\Cepsilon}{4} C_{2j+1} \,.
       \label{e:boundr-}
      \end{equ}
      Considering \lref{l:pointingoutside} for $\Cv = K_2 C_{2j+1}$
      extends \eqref{e:boundr-} to $r_- \in \R(\PP)_-$ with
      $\supp \{ c_\outt^{r_-} \} \not \in \PP$, provided
      $\rho_0(\Cv)$ is large enough. In particular, this
      proves the negativity of {$\qq{r_-}{x}$} which is
      necessary in order for \eref{e:dividesums} to imply \pref{p:lyap}.
      Combining \eqref{e:monomial2} and \eref{e:boundr-} we conclude that
      \begin{equs}
       \sum_{r \in \R(\PP)_+ \cup \R(\PP)_0} \frac{\Lambda_{r,v}(\theta^w)}
       {\Lambda_{r_-,v}(\theta^w)} {\qq{r}{\theta^w}}
       < - {\qq{r_-}{\theta^w}} \,,
      \end{equs}
      provided that for $j=0,1,\ldots,d-1$,
      \begin{equ}\label{e:Cj-choice}
       4 m \Cmonomial e^{2 c_* C_{2j}} \le \Cepsilon C_{2j+1}\,.
      \end{equ}
      Noting that $2 \zeta_* \le 2 e^{\zeta_*} \le \Cmonomial$ and
      $c_* C_{2j} \le e^{2 c_* C_{2j}}$, inequality \eqref{e:Cj-choice}
      guarantees that our preceding requirement \eqref{e:Cj-prelim} also holds when \jp{$c_* C_{2j} \ge 2 \zeta_*$}.
      In particular, \eqref{e:Cj-choice} allows us to set, $K_0$ in \dref{def:416} as
      \begin{equ}\label{e:K0-choice}
       K_0 := 4 m \Cmonomial/\Cepsilon\,.
      \end{equ}
      Combining this result with the negativity of {$\qq{r_-}{x}$} shown in \eref{e:boundr-} proves \eref{e:lyap} under our standard conditions.
      In case the reader worries about the exact order in which the various constants appear in this proof, we start by setting all those constants that depend exclusively on the structure of the network at hand, \ie $\Clambda$,
      $\zeta_*$, $v_0$, $\Cepsilon$, $\Cc$, $\Cmonomial$ \jpp{and on the chosen $\weight \in \Rr_{>0}^d$}\,. These
      determine $K_0$ by \eqref{e:K0-choice} and
      consequently all the constants $\{C_j\}$ in \eref{e:choiceofdeltas}
      (starting from $C_{0}= 2 \jp{\zeta_*/c_*}$ to satisfy \eref{e:Cj-prelim} with \eref{e:Cj-choice}).
      We then consider
      \lref{l:pointingoutside} for
      $\Cv = \Cepsilon C_{2j+1}$ and $j=0,\ldots,d_\PP-1$, as well as
      all other places where $\rho_0$ is to be enlarged as a function of all the above.
      Finally, for any $\rho > \rho_0$ we set $v^*(\rho):= e^{\rho}$, which
      in light of \rref{r:proofequivalence}, completes the
      proof of \tref{t:thm}. \qed

      \jpp{The asymptotic stability of the solutions to the \abbr{ode}s \eqref{e:ma} for the class of \abbr{crn}s introduced in \dref{d:Pendo1} can be proven by arguments similar to those presented above, extending the results of \cite{gopal13} to this larger class of networks.}

      \begin{remark}
       Closer inspection of the proof reveals that the constants $c_*$ and
       $K_0$ in \dref{def:416} only depend on the reaction rate constants $\{k_r\}$, the dimension $d$, the minimal angles
       of the faces of the polyhedron, \jpp{the chosen vector $\weight$ from \dref{d:Pendo1}} and the maximal value of
       $\|c\|_2$ for $c \in \C$\,. Our bounds on $\rho_0$ and $\rho$ are therefore universal and only depend on $\{k_r\}$ and on the geometry of the network at hand.
      \end{remark}

      \section{Wentzell-Freidlin theory: Some remarks}

      In this section and the next we deal with some \emph{local}
      questions. While large deviations theory is interesting both at
      large and small concentrations, there is of course the question of what
      kinds of transitions between stable regimes are possible. For such
      questions, \abbr{wf} theory is adequate. Here, we make contact
      between these two aspects.

      The scope of large deviations estimates can be extended from finite to
      infinite time intervals through \abbr{wf} theory \cite{fw98}.
      That framework allows for the asymptotic estimation of the exit time $\tau_{\mathcal D} := \inf\{t~:~X_t^v \not \in \mathcal D\}$ for sufficiently regular compact domains $\DD \subset \RR^d$ (see \cite[Ass.~A.3]{agazzi17}), transition times between attractors
      and invariant measure densities.
      The relevant quantity for such estimates is
      given by the \abbr{wf}-\emph{quasipotential}
      \begin{equ}
       \VV_{\mathcal D}(A,B) := \inf_{x \in A, y \in B} \inf_{\substack{{T > 0,~z \in AC_{0,T}(\mathcal D):} \\ z(0) = x,~z(T) = y}} \; I_{x,t}(z)~,
       \label{e:wfqp}
      \end{equ}
      for a fixed domain $\mathcal D$, any pair of sets $A, B \subset \mathcal D$  and $I(\cdot)$ of \eref{e:grf}.
      This quantity defines a notion of distance between attractors. This can be used to define an equivalence relation between attractors: For a fixed domain $\mathcal D$ and a pair of attractors $A,B$ are equivalent ($A \sim_{\mathcal D} B$) if $\VV_{\mathcal D}(A,B) = \VV_{\mathcal D}(B,A) = 0$\,.

      We will abstain to repeat here the ideas and  {remaining} definitions of \cite{fw98}. An interesting
      question is to find criteria where the fundamental assumptions of this theory can
      really be verified. Indeed, these conditions (summarized in points $(a)$-$(c)$ of \aref{a:2} below) speak to the post-asymptotic dynamics of the system under study,
      and there does not appear to be both simple and general conditions allowing to control, even for the relatively limited class of polynomial dynamical systems, the following necessary assumption:

      \begin{axx}[{{\cite[Ass. A4]{agazzi17}}}]
       There exist $\ell$ compact sets $K_i \subset \mathcal D$ such that:
       \begin{myenum}
        \item every $\omega$-limit set of \eref{e:ma} lying entirely in $\mathcal D$ is
        fully contained within one $K_i$\,,
        \item for any $x \in K_i$ we have $x \sim_{\mathcal D} y$ if and only if $y \in K_i$.
        \item  for all $K_i$ , the set $K_j$ minimizing $\VV_{\mathcal D}(K_i,K_j)$ is unique.
        \item $\text{Co}\{{c}^r\}_{r \in \R} = \Rr^d$.
       \end{myenum}
       \label{a:2}
      \end{axx}

   As shown in \cite{agazzi17}, this assumption, together with \aref{a:ase},
   yield the large deviations estimates on exit times $\tau_\DD$, under
   the mild regularity properties of $\DD$ from \cite[Ass.~A.3]{agazzi17}.

      \rmk{We are not aware of a general condition on an arbitrary polynomial vector
       field for which the points $(a)$ and $(b)$ in \aref{a:2} are satisfied. These points of our assumption form a weaker version of Hilbert's sixteenth problem (see \cite{hilbert16}), whereby having finitely many compact sets $K_i$ separated by finite potential barriers tolerates the existence of infinitely many attractors. Of course, in
       \emph{two} dimensions, the Poincar\'e-Bendixon theory \cite{bendixson01}
       gives the necessary information. This is also related to the
       fundamental work of Feinberg \cite{feinberg87} who actually spells out
       conditions which guarantee that the system is on a low-dimensional
       space.
       It would be interesting to see if the results in \cite{feinberg87} can be expanded in order to account for this kind of dynamical behavior.
      }


      The relevant quantity for the exponential in $v$ rate of growth of $\tau_\DD$, if
      $\DD$ contains one stable attractor $A$, is given by the quasipotential $\VV_{\mathcal D}(A, \partial \mathcal D)$ introduced in \eref{e:wfqp}. This quantity can be generalized to account for multiple attractors within $\DD$ and for estimates of invariant measure densities and transition times between attractors are also based on this quantity. However, such estimates are established in \cite{fw98} under assumption of compact phase space, that we lack here. Solely under our weaker assumptions, only concentration estimates on the transition times between attractors are possible. In particular, no first- or higher moment estimates of the distribution of transition times between attractors, nor invariant measure estimates are possible, in general, under the sole assumptions given here.
      The problem is that certain integrability of the hitting time of
      a compact {set} is needed for proving such results. Indeed, \cite{proj1} gives examples of \abbr{ase} networks not satisfying the integrability condition.

      \section{Bistable behavior}

      In this section we introduce an example \abbr{crn} displaying bi-stable behavior in order to concretely illustrate the results obtained in the previous sections, stressing in particular the extent to which they can or cannot be used.

      To stress the generality of the theory developed in this paper, the example introduced in this section has been chosen to display, in one of its attractors, a dynamic behavior that is more exotic than the one observed in examples of bi-stable systems studied in the literature \cite{qian00,wilhelm09}. Indeed, we have developed our example \abbr{crn} based on a system that is known, for a certain choice of reaction rate constants, to display chaotic behavior. This system is inspired by the well known Belusov-Zhabotinski reaction, a naturally occurring set of chemical reactions displaying chaotic and periodic behavior that can be modeled by the following set of chemical reactions \cite{chaos}:
      \begin{equ} \label{e:chaos}\emptyset \xrightharpoonup{k_0} X \xrightharpoonup{k_1} Y \xrightharpoonup{k_2} Z \xrightharpoonup{k_3} \emptyset\,, \quad Z \xrightharpoonup{k_4} X + Z\,, , \quad X + 2Y \xrightharpoonup{k_5} 3Y~.\end{equ}
      Despite the fact that the minimality of this system within the set of chemical reaction systems with chaotic dynamics in a non-vanishing region of parameter space has not been proven, some features of this model are strictly necessary (and therefore, in some sense, minimal) for the expected dynamic behavior to occur. For instance, the Poincar\'e-Bendixson theorem \cite{bendixson01} significantly limits the set of possible dynamical landscapes in 2 dimensions, ruling out among other things chaotic trajectories in that dimension and implying the necessity of a phase space of dimension $\geq 3$---as the one of \eref{e:chaos}---for the observation of chaos.

      To induce bi-stability in our model we add a new species, whose interactions are described by the simplest \abbr{crn} displaying bi-stability \cite{wilhelm09}, universally referred to as the Schl\"ogl model \cite{schlogl}:

      \begin{equ} \label{e:schlogl} \emptyset \xrightleftharpoons[k_6]{k_7} W \,,\quad 2W \xrightleftharpoons[k_{8}]{k_{9}} 3W\,.\end{equ}

      The two completely independent systems \eref{e:chaos} and \eref{e:schlogl} can now be coupled through some reactions involving species form both sides:

      \begin{equ} \label{e:coupling}
       W {\xrightharpoonup{k_{10}}} X + W\,,\qquad Z + W {\xrightharpoonup{k_{11}}} X + Z + W~.
      \end{equ}

      The choice of these reactions is not casual: As it can be seen in \eref{e:maexample}, the addition of these two reactions effectively shifts the reaction rates constants $k_0$ and $k_{ 4}$ by the $w$-dependent amounts $
      k_{{10}} w$ and $k_{{11}} w$. In particular, restricting our attention to the two equilibria $w^*, w^{**}$ of \eref{e:schlogl} (not modified by the addition of the coupling reactions), our system is equivalent to two Belusov-Zhabotinski reactions with
      $$k_0 \to k_0^{(\cdot)} := k_0 + k_{10} w^{(\cdot)}\,,\quad k_{4} \to k_{4}^{(\cdot)} := k_{4} + k_{{11}} w^{(\cdot)}\,,$$
      for $\cdot = *, **$, respectively. Therefore, choosing $k_{{10}}$ and
      $k_{{11}}$ appropriately, we obtain a \abbr{crn} displaying bi-stable behavior between a cycle and a chaotic attractor.

      The system resulting from the composition of \eref{e:chaos}, \eref{e:schlogl} and \eref{e:coupling} satisfies the asiphonic condition (\aref{a:1} (a)) but not the strongly endotactic one (\aref{a:1} (b)): Along directions $n_0 = (-3,3,3,1)$ and $n_1 = (3,0,3,1)$ there is a $n$-explosive reaction in $\R_{n}$ for $n \in \{n_0,n_1\}$. However, the system can be made strongly endotactic by addition of the following reactions:
      \begin{equ}\label{e:perturbative}3Y \xrightharpoonup{k_{12}} \emptyset\,, \qquad X + Z + W \xrightharpoonup{k_{13}} X\,.\end{equ}

      Remember that the strongly endotactic property guarantees stability of the system in the limit of large $\|x\|_1$, and is indifferent to the dynamic behavior of the system within a compact. This is reflected in the purely topological nature of this condition, which is in particular independent on the reaction rate constants $\{k_r\}_{r \in \R}$. Consequently, the reaction rates of \eref{e:perturbative} can be chosen small enough not to influence, qualitatively, the nature of the (compact) attractors, while ensuring the asymptotic stability and exponential tightness of the system. In \fref{f:chaos} a Poincar\'e section of one of the  attractors of this \abbr{crn} suggests that its chaotic behavior is indeed not influenced, for our choice of reaction rate constants, by the addition of reactions in \eref{e:perturbative}.

      \begin{figure}
       \centering
       \includegraphics[width=.6\linewidth]{\figuresfolder/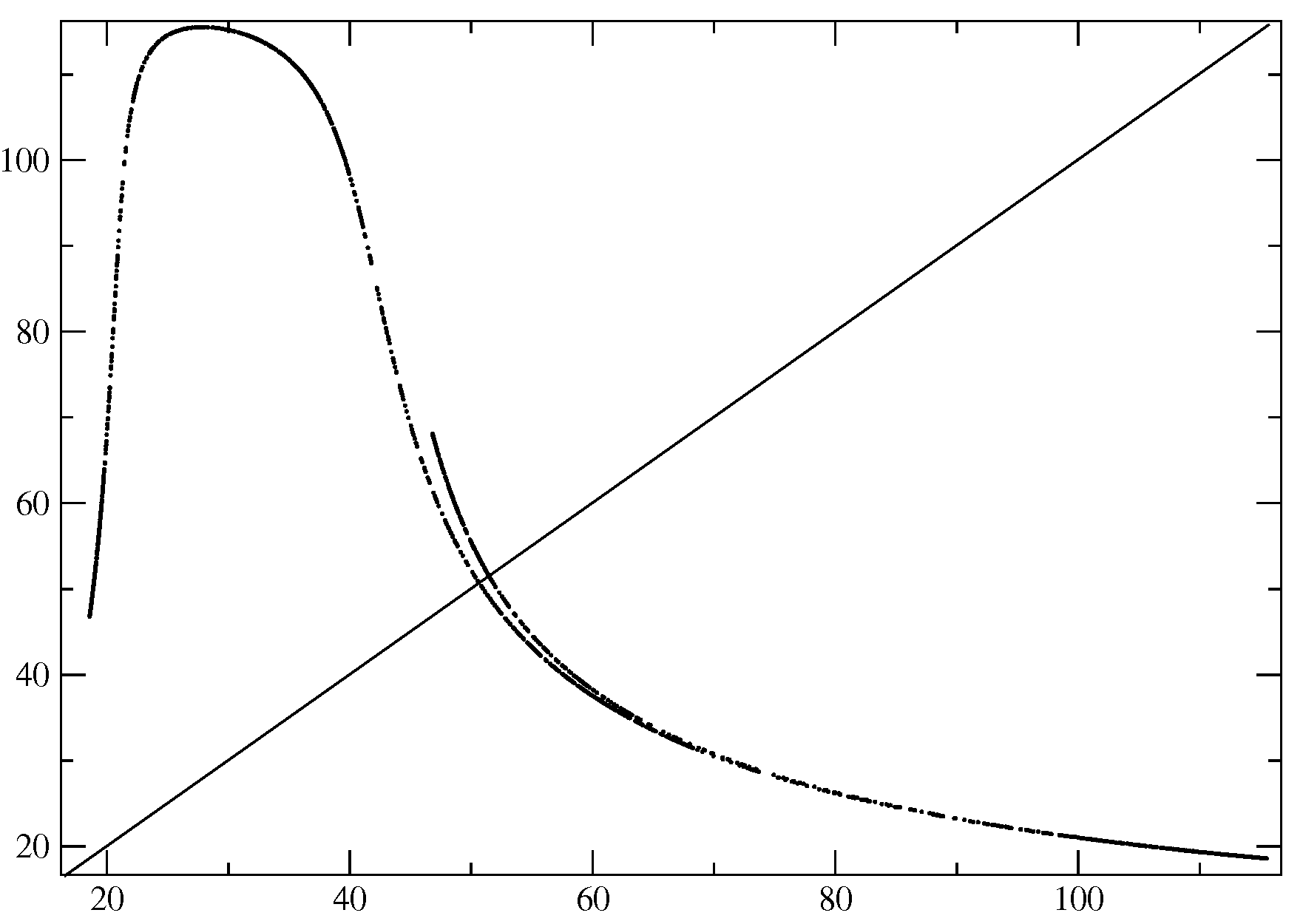}
       \caption{Poincar\'e section of the $x-y$ projection of the trajectory of the dynamical system \eref{e:maexample} for the choice of parameters of $k_{0} = 2.5, k_{1} = 0.0099, k_{2} = 1.9851, k_{3} = 0.4963, k_{4} = 0.0769, k_{5} = 0.6352,   k_{6} = 0.33, k_{7} = 2, k_{8} = 0.001, k_{9} = 0.001, k_{10} = 3, k_{11} = 10^{-9}, k_{12} = 0.01, k_{13} = 10^{-9}$ and $w \equiv w^* = 0.0389$. Assuming the existence of an analytic graph underlying the iterates of the Poincar\'e map, the associated kneading sequence hints for chaotic behavior of the attractor, conditioned on the negativity of the Schwartzian derivative of the map.  This property is stable under small perturbations of the parameters above. The solid black diagonal is the $x = y$ line. }
       \label{f:chaos}
      \end{figure}

      To summarize, the system
      $$\emptyset \xrightharpoonup{k_0} X \xrightharpoonup{k_1} Y \xrightharpoonup{k_2} Z \xrightharpoonup{k_3} \emptyset\,, \quad Z \xrightharpoonup{k_4} X + Z\,, \quad X + 2Y \xrightharpoonup{k_5} 3Y \xrightharpoonup{k_{12}} \emptyset\,,$$
      $$\emptyset \xrightleftharpoons[k_6]{k_7} W {\xrightharpoonup{k_{10}}} X + W\,,\quad 2W \xrightleftharpoons[k_{8}]{k_{9}} 3W\,,\quad Z + W {\xrightharpoonup{k_{11}}} X + Z + W \xrightharpoonup{k_{13}} X\,,$$
      governed by the deterministic mass action \abbr{ode}s
      \begin{equ} \label{e:maexample}
       \begin{cases}
        \dot x & = k_{0} + k_{10} w + z (k_{4} + k_{11} w)  - k_{1} x - k_{5} x y^2 \\
        \dot y & = k_{1} x + k_{5} x y^2 - k_{2} y - 3 k_{12} y^3                   \\
        \dot z & = k_{2} y - k_{3} z - k_{13} xzw                                   \\
        \dot w & = k_{{7}} + k_{{9}} w^2 - k_{{6}} w - k_{{8}} w^3 - k_{13} xzw
       \end{cases}~,
      \end{equ}
      displays, if modeled stochastically, spontaneous transitions between a chaotic and a cyclic attractor, as displayed in \fref{fig:sub2}. Moreover we note that the system respects \aref{a:ase} and  \aref{a:2} (for large enough $\mathcal D$).
      \begin{figure}
       \centering
       \begin{subfigure}{.5\textwidth}
        \centering
        \includegraphics[width=.8\linewidth]{\figuresfolder/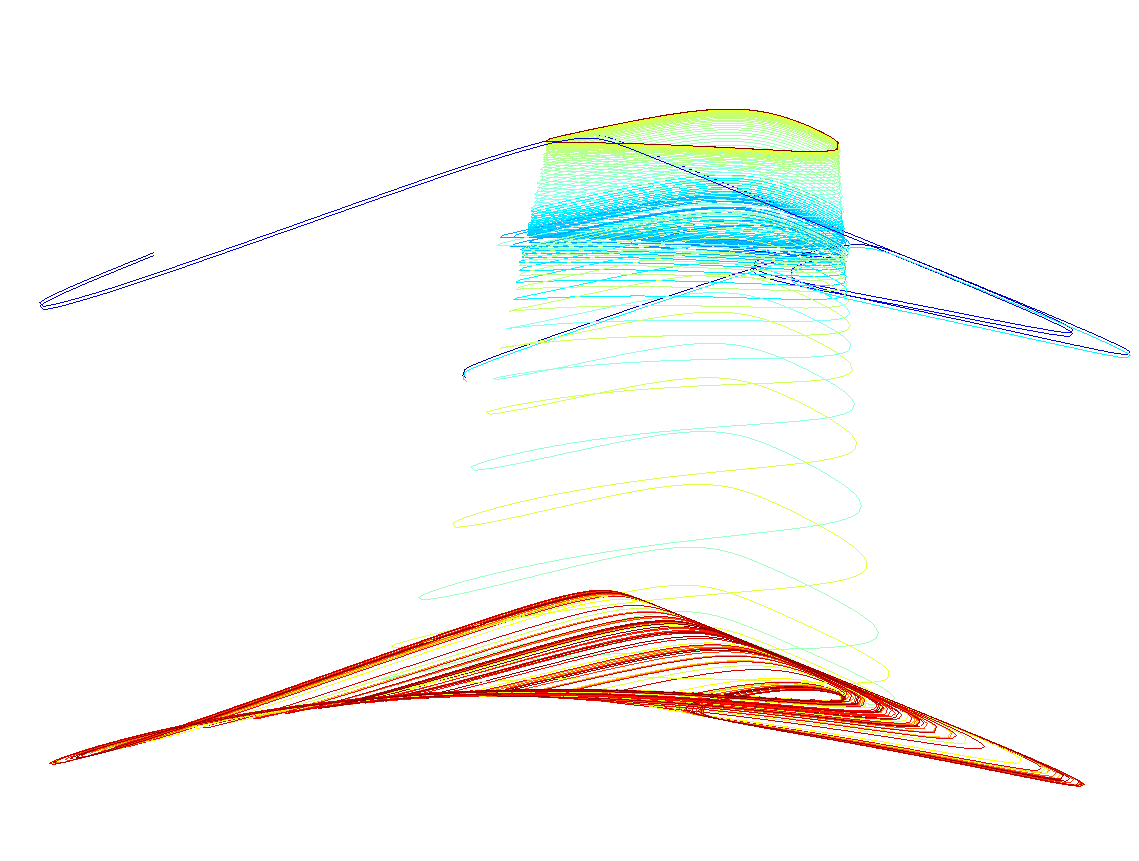}
        \caption{}
        \label{fig:sub1}
       \end{subfigure}%
       \begin{subfigure}{.5\textwidth}
        \centering
        \includegraphics[width=.8\linewidth]{\figuresfolder/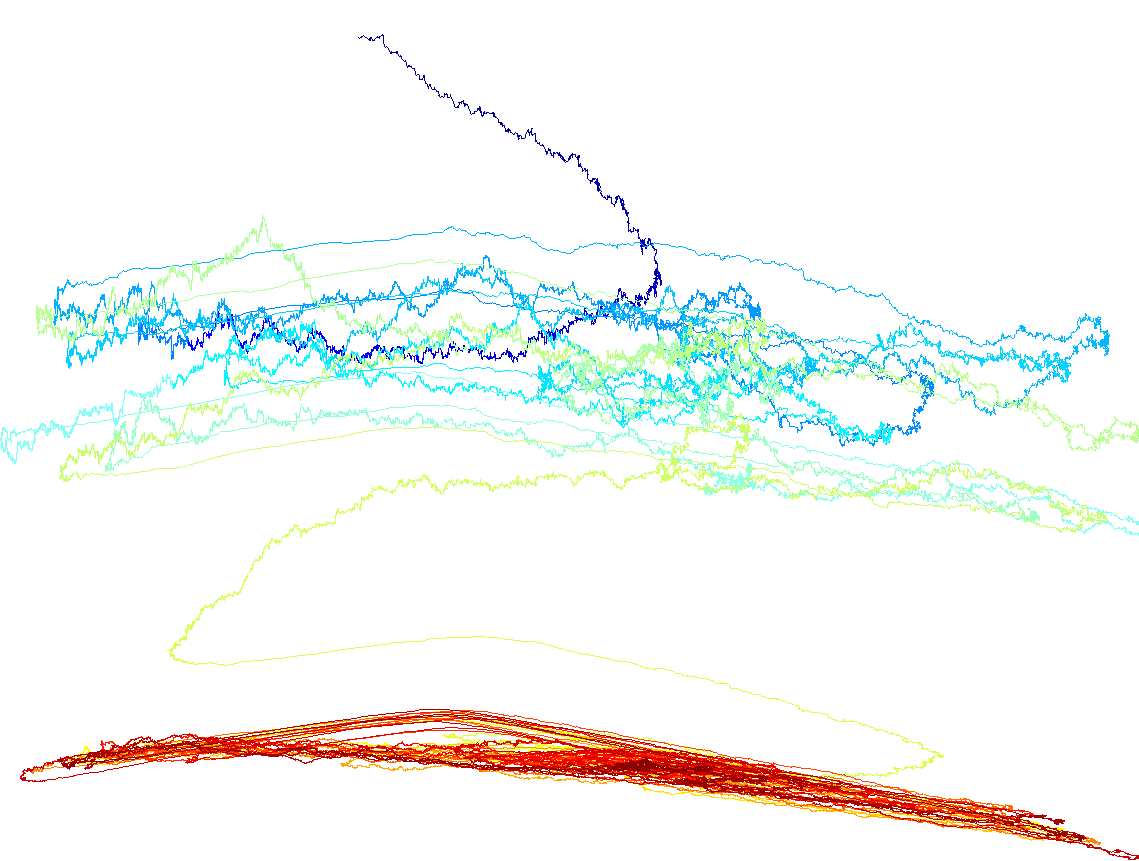}
        \caption{}
        \label{fig:sub2}
       \end{subfigure}
       \caption{Projection on the $xzw$-space of simulated deterministic (a) and stochastic (b) dynamics of a multi-attractor system, where one of the attractors is a limit cycle (in blue-green), while the other is a R\"ossler (chaotic) attractor (in red).
        In the deterministic case, multiple trajectories have been plotted, with starting points in each of the two basins of attraction. In the stochastic case, a spontaneous transition between the attractors is observed. The parameters of the simulation are those presented in \fref{f:chaos}.}
       \label{fig:simul}
      \end{figure}
      For a specific choice of reaction rate constants, some trajectories of this dynamical system, projected on the $x,y,w$ plane, are displayed in \fref{fig:simul}.

      \subsection{Conclusion - Non-equilibrium quasipotential}

      We have shown that the stochastic dynamics of certain
      \abbr{crn}-s can be studied through large deviations theory, {\it e.g.}, through \cite[Theorems 1.8, 1.9]{agazzi17},\footnote{\!\!\!assuming that the set $\mathcal D$ respects \cite[Assumption A.3]{agazzi17}.} and that it therefore represents an interesting and dynamically nontrivial example of the reach of this theory. Attention must be paid, though, about the extent of applicability of \abbr{wf} theory to this class of systems. For instance the establishment of first- and higher-moment estimates for the distribution of the transition times between attractors of a \abbr{crn} as well as exponential estimates on the invariant measure distribution is in general
      subject to stricter conditions than the ones
      provided here (c.f. \cite{proj1}).

      The family of large deviations rate functions defined in this paper appeared in the queuing theory literature \cite{dupuis16,LDT2,SW2}. In particular, quantities similar to \eref{e:grf} and the corresponding \abbr{WF}-quasipotential are introduced in the context of systems biology as candidates for potential functions in non-equilibrium systems \cite{neqpot2,neqpot1}, where they are used to describe the concentration of the invariant measure $\pi^v(x)$ of a diffusion process $Y_t^v$ approximating $X_t^v$ for $v \to \infty$ over finite time intervals \cite{vankampen92}.
      Furthermore, a Hamilton-Jacobi equation for the analytic calculation of $\VV(x)$ over the phase space $\RR^d$ is derived by studying the first order terms of a \abbr{wkb} expansion of the Fokker-Planck equation for the \abbr{pdf} of the process $Y_t^v$ \cite{toth99,neqpot2,neqpot1}.

      However, the approximation of the process $X_t^v$ through $Y_t^v$ only holds for finite time intervals, while \abbr{wf} estimates extend to infinite time. Furthermore, it is known that, by the different form of the Lagrangian of the \abbr{ldp} for Markov jump and diffusion processes, the exit time and invariant measure estimates predicted by \abbr{wf}-theory will differ, exponentially in $v$, in these two cases \cite{toth99}.
      Possible issues arising from the divergence of the \abbr{wkb} expansion as well as existence and uniqueness of the invariant measure for large enough $v$ should also be explicitly addressed, as they might limit the domain of application of the results presented in such papers. \\In this paper, we have confirmed the partial applicability of such results to systems satisfying Assumptions~\ref{a:1}, \ref{a:2} and \cite[Ass.~A.3]{agazzi17}, rigorously establishing an \abbr{ldp} in path space with the Lagrangian of Markov Jump processes and extending this result to infinite time intervals through the tools of \abbr{wf}-theory, guaranteeing that $\VV(\cdot)$ from \cite{fw98} correctly estimates exit times from compact sets. This suggests that $\VV$ can be considered as a fundamental quantity for the construction of a non-equilibrium potential.\\
      Still, not all the results of \cite{fw98} could be established: The existence of an invariant measure and exponential estimates on its concentration as well as the formulation of a \abbr{hje} for the \abbr{wf}-quasipotential in the case of multiple attractors are not presented in this paper. As shown in \cite{proj1} there exists a family of \abbr{ase} networks which lacks the integrability condition to extend the estimates established in this paper to ones of hitting times and of higher moments of the transition times distributions. Another problem is that the function $\VV$ is in general not everywhere differentiable, preventing its gradient to be a solution to the corresponding \abbr{hje}, at least in the fully rigorous mathematical sense \cite{gawedzki16}. Designing clear and general enough boundaries for the applicability of this far-reaching theory, giving solid theoretical grounds to the intuition developed in \cite{neqpot1} is an interesting topic for future research.

      \section*{Acknowledgements}

      AA's and JPE's research was partially supported by \emph{ERC advanced grant 290843 (BRIDGES)}. AA also acknowledges the \emph{SNSF grant 161866}. AD's research was partially supported by \emph{NSF grant DMS-1613091}.

      \bibliographystyle{sub_JSP}
      \bibliography{bib.bib}

\begin{thebibliography}{10}
\providecommand{\url}[1]{{#1}}
\providecommand{\urlprefix}{URL }
\expandafter\ifx\csname urlstyle\endcsname\relax
  \providecommand{\doi}[1]{DOI~\discretionary{}{}{}#1}\else
  \providecommand{\doi}{DOI~\discretionary{}{}{}\begingroup
  \urlstyle{rm}\Url}\fi

\bibitem{agazzi17}
{Agazzi}, A., {Dembo}, A., {Eckmann}, J.P.: {Large deviations theory for Markov
  jump models of chemical reaction networks}.
\newblock {To appear in Ann. Appl. Probab.}  (2017)

\bibitem{proj1}
Agazzi, A., Mattingly, J.: in preparation

\bibitem{alexandrov}
Alexandrov, A.D.: Convex polyhedra.
\newblock Springer Science \& Business Media (2005)

\bibitem{anderson15}
Anderson, D.F., Craciun, G., Gopalkrishnan, M., Wiuf, C.: Lyapunov functions,
  stationary distributions, and non-equilibrium potential for reaction
  networks.
\newblock Bulletin of Mathematical Biology \textbf{77}(9), 1744--1767 (2015).
\newblock \doi{10.1007/s11538-015-0102-8}.
\newblock \urlprefix\url{http://dx.doi.org/10.1007/s11538-015-0102-8}

\bibitem{bendixson01}
Bendixson, I.: Sur les courbes définies par des équations différentielles.
\newblock Acta Math. \textbf{24}, 1--88 (1901).
\newblock \doi{10.1007/BF02403068}.
\newblock \urlprefix\url{http://dx.doi.org/10.1007/BF02403068}

\bibitem{bergman71}
Bergman, G.M.: The logarithmic limit-set of an algebraic variety.
\newblock Transactions of the American Mathematical Society \textbf{157},
  459--469 (1971)

\bibitem{bieri84}
Bieri, R., Groves, J.R.: The geometry of the set of characters iduced by
  valuations.
\newblock Journal f{\"u}r die reine und angewandte Mathematik \textbf{347},
  168--195 (1984)

\bibitem{gawedzki16}
Bouchet, F., Gawedzki, K., Nardini, C.: {P}erturbative calculation of
  quasi-potential in non-equilibrium diffusions: A mean-field example.
\newblock Journal of Statistical Physics \textbf{163}(5), 1157--1210 (2016).
\newblock \doi{10.1007/s10955-016-1503-2}.
\newblock \urlprefix\url{http://dx.doi.org/10.1007/s10955-016-1503-2}

\bibitem{mihalkin15}
Brugall{\'e}, E., Itenberg, I., Mikhalkin, G., Shaw, K.: Brief introduction to
  tropical geometry.
\newblock arXiv preprint arXiv:1502.05950  (2015)

\bibitem{LDT1}
Dembo, A., Zeitouni, O.: Large deviations techniques and applications,
  \emph{Applications of Mathematics (New York)}, vol.~38, second edn.
\newblock Springer-Verlag, New York (1998).
\newblock \doi{10.1007/978-1-4612-5320-4}.
\newblock \urlprefix\url{http://dx.doi.org/10.1007/978-1-4612-5320-4}

\bibitem{dupuis16}
Dupuis, P., Ramanan, K., Wu, W.: Large deviation principle for finite-state
  mean field interacting particle systems.
\newblock ArXiv preprint arXiv:1601.06219  (2016)

\bibitem{either86}
Ethier, S.N., Kurtz, T.G.: Markov processes.
\newblock Wiley Series in Probability and Mathematical Statistics: Probability
  and Mathematical Statistics. John Wiley \& Sons, Inc., New York (1986).
\newblock \doi{10.1002/9780470316658}.
\newblock \urlprefix\url{http://dx.doi.org/10.1002/9780470316658}.
\newblock Characterization and convergence

\bibitem{feinberg87}
Feinberg, M.: Chemical reaction network structure and the stability of complex
  isothermal reactors—i. the deficiency zero and deficiency one theorems.
\newblock Chemical Engineering Science \textbf{42}(10), 2229 -- 2268 (1987).
\newblock \doi{http://dx.doi.org/10.1016/0009-2509(87)80099-4}.
\newblock
  \urlprefix\url{http://www.sciencedirect.com/science/article/pii/0009250987800994}

\bibitem{fw98}
Freidlin, M.I., Wentzell, A.D.: Random perturbations of dynamical systems,
  \emph{Grundlehren der Mathematischen Wissenschaften [Fundamental Principles
  of Mathematical Sciences]}, vol. 260.
\newblock Springer-Verlag, New York (1984).
\newblock \doi{10.1007/978-1-4684-0176-9}.
\newblock \urlprefix\url{http://dx.doi.org/10.1007/978-1-4684-0176-9}.
\newblock Translated from the Russian by Joseph Sz{\"u}cs

\bibitem{toth99}
Gaveau, B., Moreau, M., Toth, J.: Variational nonequilibrium thermodynamics of
  reaction-diffusion systems. i. the information potential.
\newblock The Journal of chemical physics \textbf{111}(17), 7736--7747 (1999)

\bibitem{qian00}
Ge, H., Qian, H.: Non-equilibrium phase transition in mesoscopic biochemical
  systems: from stochastic to nonlinear dynamics and beyond.
\newblock Journal of The Royal Society Interface \textbf{8}(54), 107--116
  (2010).
\newblock \doi{10.1098/rsif.2010.0202}.
\newblock
  \urlprefix\url{http://rsif.royalsocietypublishing.org/content/8/54/107}

\bibitem{gopal13}
Gopalkrishnan, M., Miller, E., Shiu, A.: A geometric approach to the global
  attractor conjecture.
\newblock SIAM Journal on Applied Dynamical Systems \textbf{13}(2), 758--797
  (2014).
\newblock \doi{10.1137/130928170}.
\newblock \urlprefix\url{http://dx.doi.org/10.1137/130928170}

\bibitem{neqpot2}
Graham, R., T\'el, T.: Nonequilibrium potential for coexisting attractors.
\newblock Phys. Rev. A \textbf{33}, 1322--1337 (1986).
\newblock \doi{10.1103/PhysRevA.33.1322}.
\newblock \urlprefix\url{https://link.aps.org/doi/10.1103/PhysRevA.33.1322}

\bibitem{gunawardena03}
Gunawardena, J.: Chemical reaction network theory for in-silico biologists.
\newblock preprint  (2003)

\bibitem{kammash14}
Gupta, A., Khammash, M.: Determining the long-term behavior of cell
  populations: A new procedure for detecting ergodicity in large stochastic
  reaction networks.
\newblock IFAC Proceedings Volumes \textbf{47}(3), 1711--1716 (2014)

\bibitem{horn72}
Horn, F., Jackson, R.: General mass action kinetics.
\newblock Archive for Rational Mechanics and Analysis \textbf{47}(2), 81--116
  (1972).
\newblock \doi{10.1007/BF00251225}.
\newblock \urlprefix\url{http://dx.doi.org/10.1007/BF00251225}

\bibitem{hilbert16}
Ilyashenko, Y.: Centennial history of {H}ilbert’s 16th problem.
\newblock Bulletin of the American Mathematical Society \textbf{39}(3),
  301--354 (2002)

\bibitem{strumfels15}
Maclagan, D., Sturmfels, B.: Introduction to tropical geometry, vol. 161.
\newblock American mathematical society Providence, RI (2015)

\bibitem{magnasco97}
Magnasco, M.O.: Chemical kinetics is turing universal.
\newblock Physical Review Letters \textbf{78}(6), 1190 (1997)

\bibitem{strumfels05}
Pachter, L., Sturmfels, B.: Algebraic statistics for computational biology,
  vol.~13.
\newblock Cambridge university press (2005)

\bibitem{petri62}
Petri, C.A.: Kommunikation mit automaten.
\newblock Ph.D. thesis, Universität Hamburg (1962)

\bibitem{vidal80}
Roux, J., Rossi, A., Bachelart, S., Vidal, C.: Representation of a strange
  attractor from an experimental study of chemical turbulence.
\newblock Physics Letters A \textbf{77}(6), 391--393 (1980)

\bibitem{schlogl}
Schl{\"o}gl, F.: Chemical reaction models for non-equilibrium phase
  transitions.
\newblock Zeitschrift f{\"u}r Physik \textbf{253}(2), 147--161 (1972).
\newblock \doi{10.1007/BF01379769}.
\newblock \urlprefix\url{http://dx.doi.org/10.1007/BF01379769}

\bibitem{LDT2}
Shwartz, A., Weiss, A.: Large deviations for performance analysis.
\newblock Stochastic Modeling Series. Chapman \& Hall, London (1995).
\newblock Queues, communications, and computing, With an appendix by Robert J.
  Vanderbei

\bibitem{SW2}
Shwartz, A., Weiss, A.: Large deviations with diminishing rates.
\newblock Math. Oper. Res. \textbf{30}(2), 281--310 (2005).
\newblock \doi{10.1287/moor.1040.0121}.
\newblock \urlprefix\url{http://dx.doi.org/10.1287/moor.1040.0121}

\bibitem{touchette09}
Touchette, H.: The large deviation approach to statistical mechanics.
\newblock Physics Reports \textbf{478}(1), 1--69 (2009)

\bibitem{vankampen92}
Van~Kampen, N.: Stochastic Processes in Physics and Chemistry, vol.~1.
\newblock Elsevier (1992)

\bibitem{wilhelm09}
Wilhelm, T.: The smallest chemical reaction system with bistability.
\newblock BMC Systems Biology \textbf{3}(1), 90 (2009).
\newblock \doi{10.1186/1752-0509-3-90}.
\newblock \urlprefix\url{http://dx.doi.org/10.1186/1752-0509-3-90}

\bibitem{chaos}
Zhang, D., Gy{\"o}rgyi, L., Peltier, W.R.: Deterministic chaos in the
  {B}elousov--{Z}habotinsky reaction: Experiments and simulations.
\newblock Chaos: An Interdisciplinary Journal of Nonlinear Science
  \textbf{3}(4), 723--745 (1993)

\bibitem{neqpot1}
Zhou, P., Li, T.: Construction of the landscape for multi-stable systems:
  Potential landscape, quasi-potential, {A}-type integral and beyond.
\newblock The Journal of Chemical Physics \textbf{144}(9), 094,109 (2016).
\newblock \doi{10.1063/1.4943096}.
\newblock \urlprefix\url{http://dx.doi.org/10.1063/1.4943096}

\end{thebibliography}

\end{document}